\documentclass[12pt]{article}
\usepackage{authblk}
\usepackage{amssymb}
\usepackage{amsthm}
\usepackage{amsmath}
\usepackage{bbm}
\usepackage{amscd}
\usepackage{stmaryrd}
\usepackage{array}
\usepackage{url}
\usepackage{hyperref}
\usepackage[latin2]{inputenc}
\usepackage{t1enc}
\usepackage{enumerate}
\usepackage[mathscr]{eucal}
\usepackage[british]{babel}
\usepackage[dvips]{graphicx}
\usepackage[dvips]{epsfig}
\usepackage{epsf}
\usepackage{color}
\usepackage{indentfirst}
\usepackage[all]{xy}

\newcommand{\Ker}{\operatorname{Ker}}
\newcommand{\GQpD}{G_{\mathbb{Q}_p,\Delta}}
\newcommand{\GQpa}{G_{\mathbb{Q}_p,\alpha}}
\newcommand{\HQpD}{H_{\mathbb{Q}_p,\Delta}}
\newcommand{\HQpa}{H_{\mathbb{Q}_p,\alpha}}
\newcommand{\GQpDa}{G_{\mathbb{Q}_p,\Delta\setminus\{\alpha\}}}
\newcommand{\HQpDa}{H_{\mathbb{Q}_p,\Delta\setminus\{\alpha\}}}

\newcommand{\RepDZp}{\mathrm{Rep}_{\mathbb{Z}_p}(G_{\mathbb{Q}_p,\Delta})}
\newcommand{\RepDQp}{\mathrm{Rep}_{\mathbb{Q}_p}(G_{\mathbb{Q}_p,\Delta})}

\newcommand{\mfp}{\mathfrak{p}}
\newcommand{\Zp}{\mathbb{Z}_p}
\newcommand{\Fp}{\mathbb{F}_p}
\newcommand{\Qp}{\mathbb{Q}_p}
\newcommand{\OK}{\mathcal{O}_K}
\newcommand{\OED}{\mathcal{O}_{\mathcal{E}_\Delta}}
\newcommand{\TD}{\mathbbm{1}_\Delta}
\newcommand{\DmpD}{\mathbb{D}(\mu_{p^\infty,\Delta})}

\newcommand{\Tr}{\operatorname{Tr}}

\newcommand{\Hom}{\operatorname{Hom}}

\newcommand{\Tor}{\operatorname{Tor}}
\newcommand{\Frac}{\operatorname{Frac}}

\newcommand{\Fil}{\operatorname{Fil}}
\newcommand{\GL}{\operatorname{GL}}
\newcommand{\Gal}{\operatorname{Gal}}
\newcommand{\res}{\operatorname{res}}
\newcommand{\id}{\operatorname{id}}

\newcommand{\rk}{\operatorname{rk}}

\newcommand{\Coker}{\operatorname{Coker}}

\newcommand{\bg}{(\hspace{-0.06cm}(}
\newcommand{\jg}{)\hspace{-0.06cm})}

\newcommand{\bs}{[\hspace{-0.04cm}[}
\newcommand{\js}{]\hspace{-0.04cm}]}
\newtheorem{thm}{Theorem}[subsection]

\newtheorem{pro}[thm]{Proposition}
\newtheorem{lem}[thm]{Lemma}
\newtheorem{cor}[thm]{Corollary}

\theoremstyle{definition}
\newtheorem*{rem}{Remark}

\begin{document}
\title{Cohomology and overconvergence for representations of powers of Galois groups}
\author[$\dagger$]{Aprameyo Pal} 
\author[$\ddag$]{Gergely Z\'abr\'adi}
\affil[$\dagger$]{Universit\"at Duisburg-Essen, Fakult\"at f\"ur Mathematik, Thea-Leymann-Stra\ss e 9, D-45127 Essen, Germany, \texttt{aprameyo.pal@uni-due.de}}
\affil[$\ddag$]{E\"otv\"os Lor\'and University, Institute of Mathematics, P\'azm\'any P\'eter s\'et\'any 1/C, H-1117 Budapest, Hungary, \texttt{zger@cs.elte.hu}}
\affil[$\ddag$]{Alfr\'ed R\'enyi Institute of Mathematics, Hungarian Academy of Sciences, POB 127, Budapest H-1364, Hungary, \texttt{zger@cs.elte.hu}}
\affil[$\ddag$]{MTA R\'enyi Int\'ezet Lend\"ulet Automorphic Research Group, \texttt{zger@cs.elte.hu}}

\maketitle

\begin{abstract}
We show that the Galois cohomology groups of $p$-adic representations of a direct power of $\Gal(\overline{\Qp}/\Qp)$ can be computed via the generalization of Herr's complex to multivariable $(\varphi,\Gamma)$-modules. Using Tate duality and a pairing for multivariable $(\varphi,\Gamma)$-modules we extend this to analogues of the Iwasawa cohomology. We show that all $p$-adic representations of a direct power of $\Gal(\overline{\Qp}/\Qp)$ are overconvergent and, moreover, passing to overconvergent multivariable $(\varphi,\Gamma)$-modules is an equivalence of categories. Finally, we prove that the overconvergent Herr complex also computes the Galois cohomology groups.
\end{abstract}

\tableofcontents

\section{Introduction}

In recent work \cite{MultVar,MultVarGal} of the second named author the relevance of $p$-adic representations of a direct power of the absolute Galois group $\Gal(\overline{\Qp}/\Qp)$ to the $p$-adic Langlands programme is pointed out. The main result of \cite{MultVarGal} is that for any finite set $\Delta$ the category of continuous representations of the group $\GQpD:=\prod_{\alpha\in\Delta}\Gal(\overline{\mathbb{Q}_p}/\mathbb{Q}_p)$ over $\Fp$ (resp.\ over $\Zp$, resp.\ over $\Qp$) is equivalent to the category of certain ``multivariable'' \'etale $(\varphi,\Gamma)$-modules: the coefficient ring modulo $p^n$ is the Laurent series ring $\Zp/(p^n)\bs X_\alpha\mid \alpha\in\Delta\js[X_\alpha^{-1}\mid \alpha\in\Delta]$ and there is an operator $\varphi_\alpha$ and group $\Gamma_\alpha\cong\Zp^\times$ for each variable $X_\alpha$ that acts via usual Frobenius lift, resp.\ cyclotomic character on the corresponding variable, and trivially on the other variables $X_\beta$ for all $\beta\in\Delta$ different from $\alpha$. On the other hand, in \cite{MultVar} a functor $D^\vee_\Delta$ is constructed from the category of smooth $p$-power torsion representations of the $\Qp$-points $G$ of a $\Qp$-split connected reductive group with connected centre to the category of multivariable $(\varphi,\Gamma)$-modules (together with a linear action of the centre $Z(G)$) where the index set $\Delta$ is the set of simple roots of $G$ with respect to a choice of Borel subgroup $B$ and maximal $\Qp$-split torus $T\leq B$. The definition of the functor $D^\vee_\Delta$ builds on earlier work of Breuil \cite{B} and has some promising properties: compatibility with tensor products and parabolic induction; right exactness in general and exactness on extensions of principal series; faithfulness on extensions of irreducible principal series. The reason why we strongly believe that representations of $\GQpD$ arise naturally in the $p$-adic Langlands programme for higher rank reductive groups over $\Qp$ is mainly that the representations theory of, say, $\GL_n(\Qp)$ ($n>2$) is much more complicated than that of $G_{\Qp}$ having $p$-cohomological dimension $2$. For instance, the work of Breuil and Herzig \cite{BrH} suggests that a generalized Montr\'eal functor \cite{B}, applied to Hecke isotypical components of the cohomology of certain Shimura varieties, should \emph{not} produce the Galois representation $\rho$ attached to the Hecke eigensystem, but a certain tensor induction
\begin{equation*}
\rho\otimes\wedge^2\rho\otimes\dots\otimes\wedge^{n-1}\rho
\end{equation*}
of it. The idea is to possibly interpret the higher exterior powers $\wedge^i\rho$, too, in this picture as representations of different copies of $G_{\Qp}$. One hint is that even in the case of $\GL_2(\Qp)$ the determinant $\wedge^2\rho$ appears on the automorphic side as a central character. What supports this is that the individual factors $\wedge^i\rho$ indeed appear in the Shimura cohomology of unitary groups of type $U(i,n-i)$ \cite{BHS} even though there is no evidence for the appearance of their tensor product so far. 

The goal of the present paper is to further develop the theory of ``multivariable $(\varphi,\Gamma)$-modules'' with an eye on possible applications to $p$-adic Hodge theory, to the $p$-adic Langlands programme, and to Iwasawa theory. 

\subsection{Outline of the paper}
In section \ref{secherrcomplexmodpn} we define the (Fontaine--)Herr complex $\Phi\Gamma^\bullet(D)$ of a multivariable $(\varphi,\Gamma)$-module $D$, which is, informally speaking, the Koszul-complex of the (commuting) operators $\{(\varphi_\alpha-1),(\gamma_\alpha-1)\mid \alpha\in\Delta\}$ acting on $D$ for some topological generator $\gamma_\alpha\in\Gamma_\alpha$ (with slight modification in case $p=2$). We show that the cochain complex $\Phi\Gamma^\bullet(D)$ computes the $\GQpD$-cohomology of the corresponding representation $V=\mathbb{V}(D)$. Our proof is new even in the classical case $|\Delta|=1$ (due to Herr \cite{H1}) and is more conceptual than the existing proofs: Instead of verifying \cite{H1} rather intrinsically that the $\delta$-functor $D\mapsto h^i\Phi\Gamma^\bullet(D)$ ($i\geq 0$) is coeffaceable or directly computing $h^1\Phi\Gamma^\bullet(D)$ \cite{Tsing} we extend the equivalence of categories $\mathbb{D}$ with \'etale $(\varphi,\Gamma)$-modules to the category of discrete $p$-primary abelian groups with continuous $\GQpD$-action by taking direct limits. The latter category has enough injectives, so by dimension shifting and checking $h^0\Phi\Gamma^\bullet(D)=V^{\GQpD}$ we may assume that $\mathbb{V}(D)$ is injective in which case the statement follows by a simple spectral sequence argument. Our proof is self-contained in the case $p=2$ (see section \ref{p=2}), too, (which is, to our taste, not satisfactorily covered by the existing literature in the classical case $|\Delta|=1$ either---for a rather sketchy proof see Thm.\ 3.3 in \cite{Liu}). 

In order to treat the Iwasawa cohomology groups $$H^i_{Iw}(\GQpD,V):=\varprojlim_{\HQpD\leq H\leq_o\GQpD}H^i(H,V)$$ in this context we develop Tate duality for multivariable $(\varphi,\Gamma)$-modules (section \ref{dualphigamma}). Note that the coefficient ring $\Zp/(p^n)\bs X_\alpha\mid \alpha\in\Delta\js[X_\alpha^{-1}\mid \alpha\in\Delta]$ is \emph{not} locally compact when $|\Delta|>1$, so we cannot hope for a perfect pairing $\{\cdot,\cdot\}$ between $D$ and its Tate dual $D^*(\TD)$. However, the pairing we construct is non-degenerate and $\varphi$- and $\Gamma$-equivariant. This allows us to show that the $\psi$-complex $\Psi^\bullet(D)$ (ie.\ the Koszul-complex of the operators $\psi_\alpha-1$ for $\alpha\in\Delta$) computes these Iwasawa-cohomology groups. Here $\psi_\alpha$ is the distinguished left inverse of the operator $\varphi_\alpha$ ($\alpha\in\Delta$). The technical difficulty towards this is to show that the cohomology groups $h^i\Psi^\bullet(D)$ are compact and hence the pairing $\{\cdot,\cdot\}\colon D\times D^*(\TD)\to \Qp/\Zp$---even though not perfect on the whole $D$---induces a perfect paring between $h^i\Psi^\bullet(D)$ and $h^{|\Delta|-i}\Phi^\bullet(D^*(\TD))$.

In section \ref{secherrpadic} we extend the results on the computation of $\GQpD$-cohomology to representations over $\Zp$ and $\Qp$. Our treatment here is inspired by the recent paper of Schneider and Venjakob \cite{SchVen}. We finish section \ref{secherrcomplex} by proving the analogue of the Euler--Poincar\'e characteristic formula in this context. Even though there is a simple proof using the Hochschild--Serre spectral sequence and the classical $|\Delta|=1$ case we chose to do this via the complex $\Psi^\bullet(D)$, since along the way we also show further finiteness properties of the Iwasawa cohomology groups that we need later on.

Section \ref{secoverconverge} is devoted to overconvergence. The fact that all continuous representations of $\GQpD$ are overconvergent follows rather easily from the one-variable case by induction. However, in order to show that passing to overconvergent $(\varphi,\Gamma)$-modules is an equivalence of categories (ie.\ this functor is essentially surjective) one needs to introduce multivariable analogues of ``extended Robba rings'' in the sense of Kedlaya \cite{K}. In the last section we use the observation (which follows from the above equivalence of categories) that for any fixed choice of $\alpha$ in $\Delta$ each multivariable overconvergent $(\varphi,\Gamma)$-module admits a basis in which the matrices of $\varphi_\alpha$ and $\gamma_\alpha\in\Gamma_\alpha$ contain only the variable $X_\alpha$, and no $X_\beta$ for $\beta\neq \alpha$ in $\Delta$. We combine this with uniform continuity of the operators $\gamma_\alpha-1$ and $\psi_\alpha-1$ \cite{CC,CC2} to verify that the natural map from the overconvergent to the $p$-completed Herr complex is a quasi-isomorphism. In particular, the former also computes the same $\GQpD$-cohomology groups. Here again, we treat first the Iwasawa cohomology and then deduce the statement on $\GQpD$-cohomology using a quasi-isomorphism between the cochain complexes $\Phi\Gamma^\bullet(D)$ and $\Psi\Gamma^\bullet(D)$.

Throughout the paper we decided to work with the coefficient field $\Qp$ (resp.\ $\Zp$, resp.\ $\Fp$) since using a finite extension $K/\Qp$ (resp.\ its ring of integers $\OK$, resp.\ residue field $\kappa$) would lead to the same statements as restricting the coefficients to $\Qp$ (resp.\ $\Zp$, resp.\ $\Fp$) do not change the $\GQpD$-cohomology groups, nor the overconvergence. We also decided not to replace $\Qp$ by a finite extension (or even by $|\Delta|$ distinct extensions) in $\GQpD$. One reason for this is that the paper \cite{MultVarGal} only covers representations of $\GQpD$. Further, group cohomology of finite index subgroups of $\GQpD$ can easily be computed via $\GQpD$-cohomology using Shapiro's Lemma. Regarding the overconvergence there would be two natural ways of passing to finite extensions $F/\Qp$: one could either work with cyclotomic or Lubin--Tate $(\varphi,\Gamma)$-modules. The cyclotomic case is covered in a recent paper \cite{CKZ} of the second named author with A.\ Carter and K.\ Kedlaya. However, there is strong evidence \cite{Col15} that Lubin--Tate $(\varphi,\Gamma)$-modules (or maybe even $(\varphi,\Gamma)$-modules over the character variety \cite{BSX}) are better suited for the extension of the $p$-adic Langlands correspondence to $\GL_2(F)$ where $F/\Qp$ is a finite extension. We expect that multivariable versions (for products of Galois groups) of these Lubin--Tate $(\varphi,\Gamma)$-modules will play a role in a future $p$-adic Langlands correspondence for reductive groups of higher rank over $F$. Note that the question of overconvergence is more subtle in Lubin--Tate theory even in the one variable case. The ``one-variable'' Lubin--Tate $(\varphi,\Gamma)$-module corresponding to a representation of $G_F$ is overconvergent if it is either $F$-analytic (Thm.\ C in \cite{Ber2}) or factors through $\Gamma_K$. Conversely, any overconvergent representation arises as a quotient of the tensor product of an $F$-analytic representation and a representation factoring through $\Gamma_K$ (Thm.\ A in \cite{BF}). It is natural to expect the same results in the product situation, as well. 

\subsection{Relation to Iwasawa theory}
This paper builds up the necessary technical tools to formulate Bloch--Kato exponential maps and $\varepsilon$-isomorphisms in this product situation. The natural next step would be to extend the equivalence between categories of continuous representations of the group $\GQpD$ and \'etale $(\varphi_\Delta,\Gamma_\Delta)$-modules to the multivariable Robba ring $\mathcal{R}_\Delta$ and show that the Herr complex of \'etale $(\varphi_\Delta,\Gamma_\Delta)$-modules over $\mathcal{R}_\Delta$ still computes Galois cohomology. This should follow similarly as in the overconvergent case, but we do not pursue this here to keep the length of the article reasonable. Once we have the multivariable analogue of $\mathrm{D_{dif}}$ extending Berger's work, it would be possible to define a general Bloch-Kato exponential map of $(\varphi_\Delta,\Gamma_\Delta)$-modules over $\mathcal{R}_\Delta$ following Nakamura \cite{Na14}. This would interpolate Perrin-Riou's big logarithm maps in this setting and would be related to the inverse of the isomorphism as in Theorem \ref{iwasawamodp} which is a generalization of classical $\mathrm{Log^{*}}$ (Theorem II.1.3, \cite{CC2}). Using the Bloch-Kato exponential map it would be possible to formulate a conjectural description of $\varepsilon$-isomorphisms in the multivariable setting \cite{Na17a}. We hope to prove many cases of these multivariable $\varepsilon$-conjectures using known one-variable results of Benois \cite{Ben}, Nakamura \cite{Na17b} etc. We speculate also another possible link with (abelian) equivariant epsilon conjecture as in Benois-Berger \cite{BB}, Bley-Cobbe \cite{BC}. We hope to relate the multivariable $\varepsilon$-conjecture with the abelian equivariant $\varepsilon$-conjecture. Using the results on multivariable $\varepsilon$-conjecture, it should provide us with new cases of classical, one-variable equivariant $\varepsilon$-conjectures by restricting to the diagonal embedding of $G_{\Qp}$ to $\GQpD$.

\subsection{Relation to $p$-adic geometry}
Products of local Galois groups show up rather naturally in modern $p$-adic geometry via Drinfeld's Lemma (Lemma 1.1.2 in \cite{SchW}). In particular, $p$-adic representations of $\GQpD$ are in one-to-one correspondence with certain local systems on the product $\operatorname{Spd}\Qp \times\dots\times\operatorname{Spd}\Qp$ of diamonds (see Thm.\ 16.3.1 in \emph{op.\ cit.}). The reason why the equivalence of categories between representations of $\GQpD$ and multivariable $(\varphi,\Gamma)$-modules is not a direct consequence of this general theory is that $\operatorname{Spd}\Qp \times\dots\times\operatorname{Spd}\Qp$ is not the adic spectrum of a fixed ring. However, using an embedding into the adic spectrum of the perfect closure of $\OED/(p)$ it is possible \cite{CKZ} to prove the main result of \cite{MultVarGal} in this fashion (see also Cor.\ 4.3.16 in \cite{KedArizona}) even in a more general form of classifying representations of a product $G_{F_1}\times\dots\times G_{F_d}$ where $F_1,\dots,F_d$ are finite extensions of $\Qp$. 

\subsection{Relation to other notions of multivariable $(\varphi,\Gamma)$-modules}
Our definition of multivariable overconvergent and Robba rings is somewhat different from that considered in \cite{Ber,K2} and in the possibly non-commutative version in \cite{ZSch}. Here the functions are required to converge on a full polyannulus whereas in these previous constructions the modulus of the variables have a fixed relation. The reason for this difference is that we have partial Frobenii to act on our rings $\OED^\dagger$ and $\mathcal{R}_\Delta$ and the relation of the moduli of variables changes under these operators. However, $\mathcal{R}_\Delta$ can naturally be viewed as a subring of the multivariable Robba ring considered in \cite{K2}. This relation is expected to have consequences on the structural properties of Berger's multivariable Lubin--Tate $(\varphi,\Gamma)$-modules.

\subsection{Notations}\label{secnotations}

For a finite set $\Delta$ let $\GQpD:=\prod_{\alpha\in\Delta}\Gal(\overline{\mathbb{Q}_p}/\mathbb{Q}_p)$ denote the direct power of the absolute Galois group of $\Qp$ indexed by $\Delta$. We denote by $\operatorname{Rep}_{\Fp}(\GQpD)$ (resp.\ by $\operatorname{Rep}_{\Zp}(\GQpD)$, resp.\ by $\operatorname{Rep}_{\Qp}(\GQpD)$) the category of continuous representations of the profinite group $\GQpD$ on finite dimensional $\Fp$-vector spaces (resp.\ finitely generated $\Zp$-modules, resp.\ finite dimensional $\Qp$-vector spaces). On the other hand, for independent commuting variables $X_\alpha$ ($\alpha\in\Delta$) we put 
\begin{eqnarray*}
E_{\Delta}&:=&\Fp\bs X_\alpha\mid \alpha\in\Delta\js [X_\alpha^{-1}\mid \alpha\in\Delta]\ , \\
\mathcal{O}_{\mathcal{E}_{\Delta}}&:=&\varprojlim_h \left(\Zp/\varpi^h\bs X_\alpha\mid \alpha\in\Delta\js [X_\alpha^{-1}\mid \alpha\in\Delta]\right)\ ,\\
\mathcal{E}_{\Delta}&:=&\mathcal{O}_{\mathcal{E}_{\Delta}}[p^{-1}]\ .
\end{eqnarray*}
Moreover, for each element $\alpha\in\Delta$ we have the partial Frobenius $\varphi_\alpha$, and group $\Gamma_\alpha\cong \Gal(\Qp(\mu_{p^\infty})/\Qp)\overset{\chi_\alpha,\sim}{\to}\Zp^\times$ acting on the variable $X_\alpha$ in the usual way  
$$\varphi_\alpha(X_\alpha):=(1+X_\alpha)^p-1\ ,\quad \gamma_\alpha(X_\alpha):=(1+X_\alpha)^{\chi_\alpha(\gamma_\alpha)}-1\ ,\ \gamma_\alpha\in\Gamma_\alpha$$
and commuting with the other variables $X_\beta$ ($\beta\in\Delta\setminus\{\alpha\}$) in the above rings. We put $\Gamma_\Delta:=\prod_{\alpha\in\Delta}\Gamma_\alpha$ which is naturally the quotient of the group $\GQpD$ by the normal subgroup $\HQpD:=\prod_{\alpha\in\Delta}\HQpa$ where $\Gal(\overline{\Qp}/\Qp(\mu_{p^\infty})\cong\HQpa\leq \GQpa\cong \Gal(\overline{\Qp}/\Qp)$. A $(\varphi_\Delta,\Gamma_\Delta)$-module over $E_{\Delta}$ (resp.\ over $\mathcal{O}_{\mathcal{E}_{\Delta}}$, resp.\ over $\mathcal{E}_{\Delta}$) is a finitely generated $E_{\Delta}$-module (resp.\ $\mathcal{O}_{\mathcal{E}_{\Delta}}$-module, resp.\ $\mathcal{E}_{\Delta}$-module) $D$ together with commuting semilinear actions of the operators $\varphi_\alpha$ and groups $\Gamma_\alpha$ ($\alpha\in\Delta$). In case the coefficient ring is $E_{\Delta}$ or $\mathcal{O}_{\mathcal{E}_{\Delta}}$, we say that $D$ is \'etale if the map $\id\otimes\varphi_\alpha\colon \varphi_\alpha^*D\to D$ is an isomorphism for all $\alpha\in\Delta$. For the coefficient ring $\mathcal{E}_{\Delta}$ we require the stronger assumption for the \'etale property that $D$ comes from an \'etale $(\varphi_\Delta,\Gamma_\Delta)$-module over $\mathcal{O}_{\mathcal{E}_{\Delta}}$ by inverting $p$. The main result of \cite{MultVarGal} is that $\operatorname{Rep}_{\Fp}(\GQpD)$ (resp.\ $\operatorname{Rep}_{\Zp}(\GQpD)$, resp.\ $\operatorname{Rep}_{\Qp}(\GQpD)$) is equivalent to the category of \'etale $(\varphi_\Delta,\Gamma_\Delta)$-modules over $E_{\Delta}$ (resp.\ over $\mathcal{O}_{\mathcal{E}_{\Delta}}$, resp.\ over $\mathcal{E}_{\Delta}$).

\section{Cohomology of $\GQpD$ via the Herr complex}\label{secherrcomplex}

\subsection{Cohomology of $p$-torsion representations}\label{secherrcomplexmodpn}

In order to prove our main result in this section we first extend the functor $\mathbb{D}$ originally defined in \cite{MultVarGal} for objects in $\operatorname{Rep}_{\Zp}(\GQpD)$ to the category $\operatorname{Rep}_{\Zp-tors}^{discr}(\GQpD)$ of discrete $p$-primary abelian groups with continuous action of $\GQpD$. This will be needed in the sequel as we shall use injective objects in this category which do not exist in the category of finitely generated modulo $p^n$ representations of $\GQpD$. For an object $V$ in $\operatorname{Rep}_{\Zp-tors}^{discr}(\GQpD)$ we put $$\mathbb{D}(V):=\left(\mathcal{O}_{\widehat{\mathcal{E}^{ur}_\Delta}}\otimes_{\Zp}V\right)^{\HQpD}$$ (see §4.2 in \cite{MultVarGal} for the definition of $\mathcal{O}_{\widehat{\mathcal{E}^{ur}_\Delta}}$). Any object in $\operatorname{Rep}_{\Zp-tors}^{discr}(\GQpD)$ is the filtered direct limit of $p$-torsion objects in $\operatorname{Rep}_{\Zp}(\GQpD)$. Moreover, $\mathbb{D}$ commutes with filtered direct limits since both the tensor product and taking $\HQpD$-invariants do so. Therefore $\mathbb{D}$ is an exact functor into the category $\varinjlim \mathcal{D}^{et}_{tors}(\varphi_{\Delta},\Gamma_\Delta,\mathcal{O}_{\mathcal{E}_\Delta})$ of injective limits of $p$-power torsion objects in $\mathcal{D}^{et}(\varphi_{\Delta},\Gamma_\Delta,\mathcal{O}_{\mathcal{E}_\Delta})$ by Cor.\ 4.8 in \cite{MultVarGal}. On the other hand, for an object $D$ in $\varinjlim \mathcal{D}^{et}_{tors}(\varphi_{\Delta},\Gamma_\Delta,\mathcal{O}_{\mathcal{E}_\Delta})$ we define $$\mathbb{V}(D):=\bigcap_{\alpha\in \Delta}(\mathcal{O}_{\widehat{\mathcal{E}^{ur}_\Delta}}\otimes_{\mathcal{O}_{\mathcal{E}_\Delta}}D)^{\varphi_\alpha=\id}\ .$$ Since $\mathbb{V}$ also commutes with direct limits, we deduce
\begin{cor}
The functors $\mathbb{D}$ and $\mathbb{V}$ are quasi-inverse equivalences of categories between $\operatorname{Rep}_{\Zp-tors}^{discr}(\GQpD)$ and $\varinjlim \mathcal{D}^{et}_{tors}(\varphi_{\Delta},\Gamma_\Delta,\mathcal{O}_{\mathcal{E}_\Delta})$.
\end{cor}
\begin{proof}
This follows from Thm.\ 4.10 in \cite{MultVarGal} by taking direct limits.
\end{proof}

Put $D^{sep}:=\mathcal{O}_{\widehat{\mathcal{E}^{ur}_\Delta}}\otimes_{\Zp}V=\mathcal{O}_{\widehat{\mathcal{E}^{ur}_\Delta}}\otimes_{\mathcal{O}_{\mathcal{E}_\Delta}}\mathbb{D}(V)$ and consider the cochain complex
\begin{align*}
\Phi^\bullet(D^{sep})\colon 0\to D^{sep}\to \bigoplus_{\alpha\in\Delta}D^{sep}\to \dots\to \bigoplus_{\{\alpha_1,\dots,\alpha_r\}\in \binom{\Delta}{r}}D^{sep}\to\dots \to D^{sep}\to 0
\end{align*}
where for all $0\leq r\leq |\Delta|-1$ the map $d_{\alpha_1,\dots,\alpha_r}^{\beta_1,\dots,\beta_{r+1}}\colon D^{sep}\to D^{sep}$ from the component in the $r$th term corresponding to $\{\alpha_1,\dots,\alpha_r\}\subseteq \Delta$ to the component corresponding to the $(r+1)$-tuple $\{\beta_1,\dots,\beta_{r+1}\}\subseteq\Delta$ is given by
\begin{equation*}
d_{\alpha_1,\dots,\alpha_r}^{\beta_1,\dots,\beta_{r+1}}=\begin{cases}0&\text{if }\{\alpha_1,\dots,\alpha_r\}\not\subseteq\{\beta_1,\dots,\beta_{r+1}\}\\ (-1)^{\varepsilon}(\id-\varphi_\beta)&\text{if }\{\beta_1,\dots,\beta_{r+1}\}=\{\alpha_1,\dots,\alpha_r\}\cup\{\beta\}\ ,\end{cases}
\end{equation*}
where $\varepsilon=\varepsilon(\alpha_1,\dots,\alpha_r,\beta)$ is the number of elements in the set $\{\alpha_1,\dots,\alpha_r\}$ smaller than $\beta$. 

\begin{lem}\label{augrep}
For any object $V$ in $\operatorname{Rep}_{\Zp-tors}^{discr}(\GQpD)$ the augmentation map $V[0]\to \Phi^\bullet(D^{sep})$ is a quasi-isomorphism of cochain complexes where $V[0]$ denotes the complex with $V$ in degree zero and $0$ everywhere else.
\end{lem}
\begin{proof}
By Proposition 4.2 in \cite{MultVarGal}, the augmentation map $\Fp[0]\to \Phi^\bullet(E^{sep}_\Delta)$ is a quasi-isomorphism. By devissage, the augmentation map $\mathbb{Z}/p^n[0]\to \Phi^\bullet(\mathcal{O}_{\widehat{\mathcal{E}^{ur}_\Delta}}/p^n)$ is also a quasi-isomorphism as each term is a flat $\mathbb{Z}/p^n$-module in both complexes. Now if $V$ is a finite abelian $p$-group then it is killed by $p^n$ for some $n$ and we have $\Phi^\bullet(\mathbb{D}(V)^{sep})=\Phi^\bullet(\mathcal{O}_{\widehat{\mathcal{E}^{ur}_\Delta}}/p^n)\otimes_{\mathbb{Z}/p^n}V$. Using again the flatness, the statement follows from the quasi-isomorphism $\mathbb{Z}/p^n[0]\to \Phi^\bullet(\mathcal{O}_{\widehat{\mathcal{E}^{ur}_\Delta}}/p^n)$ by tensoring with $V$. The case of general $V$ is deduced by taking the direct limit which is exact.
\end{proof}

\begin{lem}\label{pncohomtriv}
We have $H^i_{cont}(\HQpD, \mathcal{O}_{\widehat{\mathcal{E}^{ur}_\Delta}}/p^n)=0$ for all $n\geq 1$ and $i\geq 1$.
\end{lem}
\begin{proof}
By the long exact sequence of cohomology (devissage) we are reduced to the case $n=1$. The case $i=1$ is treated in Prop.\ 4.1 in \cite{MultVarGal} and the higher cohomology groups vanish for the same reason: using the notations therein $E'_\Delta$ is cohomologically trivial for all finite extensions $E'_\alpha/E_\alpha$ ($\alpha\in\Delta$) as it is induced as an $H'$-module.
\end{proof}

\begin{pro}\label{HcohomHerr}
The complex $\Phi^\bullet(\mathbb{D}(V))$ computes the $\HQpD$-cohomology of $V$, ie.\ we have $h^i\Phi^\bullet(\mathbb{D}(V))\cong H^i(\HQpD,V)$ as representations of $\Gamma_\Delta$.
\end{pro}
\begin{proof}
At first assume that $V$ is finite. By the definition of $\mathbb{D}(V)$ the complex $\Phi^\bullet(\mathbb{D}(V))$ is the $\HQpD$-invariant part of $\Phi^\bullet(D^{sep})$. However, the terms of $\Phi^\bullet(D^{sep})$ are direct sums of copies of $D^{sep}=E_\Delta^{sep}\otimes_{E_\Delta}\mathbb{D}(V)$ which are acyclic objects for the $\HQpD$-cohomology by Lemma \ref{pncohomtriv}. The statement is deduced from Lemma \ref{augrep}. The general case follows noting that both $h^i\Phi^\bullet(\mathbb{D}(\cdot))$ and $H^i(\HQpD,\cdot)$ commute with filtered direct limits.
\end{proof}

We denote by $C_\Delta$ the torsion subgroup of $\Gamma_\Delta\cong \prod_{\alpha\in\Delta}\Zp^\times$ and put $\HQpD^*$ for the kernel of the composite quotient map $\GQpD\twoheadrightarrow \Gamma_\Delta\twoheadrightarrow \Gamma_\Delta^*:=\Gamma_\Delta/C_\Delta$. Then $C_\Delta$ is isomorphic to $\prod_{\alpha\in\Delta}(\mathbb{Z}/2p\mathbb{Z})^\times$ (which has order prime to $p$ if and only if $p$ is odd). We have
\begin{cor}\label{HQpD*cohom}
The complex $\Phi^\bullet(\mathbb{D}(V)^{C_\Delta})$ computes the $\HQpD^*$-cohomology of $V$.
\end{cor}
\begin{proof}
In case $p$ is odd this follows from the Hochschild--Serre spectral sequence using Prop.\ \ref{HcohomHerr} since $C_\Delta$ is prime to $p$ and therefore has $p$-cohomological dimension $0$. The proof in case $p=2$ is postponed to section \ref{p=2} below.
\end{proof}

We choose topological generators $\gamma_\alpha\in \Gamma_\alpha^*:=\Gamma_\alpha/(\Gamma_\alpha\cap C_\Delta)$ for each $\alpha\in\Delta$. If $A$ is an arbitrary (for now abstract) representation of the group $\Gamma_\Delta^*\cong \prod_{\alpha\in\Delta}\Zp$ on a $\Zp$-module we denote by $\Gamma_\Delta^\bullet(A)$ the cochain complex
\begin{align*}
\Gamma_\Delta^\bullet(A)\colon 0\to A\to \bigoplus_{\alpha\in\Delta}A\to \dots\to \bigoplus_{\{\alpha_1,\dots,\alpha_r\}\in \binom{\Delta}{r}}A\to\dots \to A\to 0
\end{align*}
analogous to $\Phi^\bullet(\cdot)$ where for all $0\leq r\leq |\Delta|-1$ the map $d_{\alpha_1,\dots,\alpha_r}^{\beta_1,\dots,\beta_{r+1}}\colon A\to A$ from the component in the $r$th term corresponding to $\{\alpha_1,\dots,\alpha_r\}\subseteq \Delta$ to the component corresponding to the $(r+1)$-tuple $\{\beta_1,\dots,\beta_{r+1}\}\subseteq\Delta$ is given by
\begin{equation*}
d_{\alpha_1,\dots,\alpha_r}^{\beta_1,\dots,\beta_{r+1}}=\begin{cases}0&\text{if }\{\alpha_1,\dots,\alpha_r\}\not\subseteq\{\beta_1,\dots,\beta_{r+1}\}\\ (-1)^{\varepsilon}(\id-\gamma_\beta)&\text{if }\{\beta_1,\dots,\beta_{r+1}\}=\{\alpha_1,\dots,\alpha_r\}\cup\{\beta\}\ ,\end{cases}
\end{equation*}
where $\varepsilon=\varepsilon(\alpha_1,\dots,\alpha_r,\beta)$ is the number of elements in the set $\{\alpha_1,\dots,\alpha_r\}$ smaller than $\beta$.

\begin{lem}\label{deltafunctGamma}
The functors $A\mapsto h^n\Gamma_\Delta^\bullet(A)$ ($n\geq 0$) form a cohomological $\delta$-functor. Moreover, if $A$ is a discrete abelian group with continuous $\Gamma_\Delta^*$-action, then we have $h^0\Gamma_\Delta^\bullet(A)=A^{\Gamma_\Delta^*}$.
\end{lem}
\begin{proof}
Given a short exact sequence $0\to A\to B\to C\to 0$ of representations of $\Gamma_\Delta$, we obtain a short exact sequence of cochain complexes $0\to \Gamma_\Delta^\bullet(A)\to \Gamma_\Delta^\bullet(B)\to \Gamma_\Delta^\bullet(C)\to 0$ whose long exact sequence yields maps $\delta^n\colon h^n\Gamma_\Delta^\bullet(C)\to h^{n+1}\Gamma_\Delta^\bullet(A)$ that are functorial in the short exact sequence $0\to A\to B\to C\to 0$.

For the second statement note that the action of $\Gamma_\Delta^*$ on $A$ locally factors through a finite quotient. Therefore $A^{\Gamma_\Delta^*}=\bigcap_{\alpha\in\Delta}\Ker(\id-\gamma_\alpha)=h^0\Gamma_\Delta^\bullet(A)$ as the classes of the elements $\gamma_\alpha$ ($\alpha\in\Delta$) generate any finite quotient of $\Gamma_\Delta^*$.
\end{proof}

\begin{pro}\label{Gamma*cohom}
Assume that $A$ is a discrete $p$-primary abelian group on which $\Gamma_\Delta^*$ acts continuously. Then the complex $\Gamma_\Delta^\bullet(A)$ computes the $\Gamma_\Delta^*$-cohomology of $A$, ie.\ we have $h^i\Gamma_\Delta^\bullet(A)\cong H^i_{cont}(\Gamma_\Delta^*,A)$ for all $i\geq 0$.
\end{pro}
\begin{proof}
The case $|\Delta|=1$ is well-known, see for example exercise 2.2 in \cite{G}. However, for the convenience of the reader we give a proof even in this case. We proceed in $4$ steps.

\emph{Step 1:} Assume that $A$ is a finite abelian $p$-group and $|\Delta|=1$. Then the complex $\Gamma_\Delta^\bullet(A)$ reads $0\to A\overset{\id-\gamma}{\to} A\to 0$. Since $\Gamma_\Delta^*$ is generated topologically by $\gamma$ and acts on $A$ via a finite quotient, we have $H^0_{cont}(\Gamma_\Delta^*,A)=\Ker(\id-\gamma)$. Now recall that the continuous cohomology $H^1_{cont}(\Gamma_\Delta^*,A)$ is defined as $\varinjlim_n H^1(\Gamma_\Delta^*/\Gamma_{\Delta,n}^*,A^{\Gamma_{\Delta,n}^*})$ where $\Gamma_{\Delta,n}^*$ is the unique subgroup in $\Gamma_\Delta^*$ of index $p^n$. Since $A$ is finite, we have $A^{\Gamma_{\Delta,n}^*}=A$ for $n$ large enough. Now for the cohomology of the cyclic group $\Gamma_\Delta^*/\Gamma_{\Delta,n}^*$ we have $H^1(\Gamma_\Delta^*/\Gamma_{\Delta,n}^*,A)=\Ker(N)/\mathrm{Im}(\id-\gamma)$ where $N=\sum_{i=0}^{p^n-1}\gamma^i\colon A\to A$ is the norm map. Again, if $n$ is large enough, then even $\Gamma_{\Delta,n-k}^*$ acts trivially on $A$ where $|A|=p^k$ whence $N=p^k\sum_{i=0}^{p^{n-k}-1}\gamma^i$ is the zero map on $A$. The statement follows noting that all the other cohomology groups vanish as $\Zp$ has $p$-cohomological dimension $1$.

\emph{Step 2:} Assume that $A$ is any discrete $p$-primary abelian group and $|\Delta|=1$. By the continuity of the action of $\Gamma_\Delta^*$, $A$ is a direct limit of its finite $\Gamma_\Delta^*$-invariant subgroups and the statement follows from Step $1$ noting that both $H^n(\Gamma_\Delta^*,\cdot)$ and $h^n\Gamma^\bullet_\Delta(\cdot)$ commute with filtered direct limits.

\emph{Step 3:} Assume that $A$ is an  injective object in the category of discrete $p$-primary abelian groups with continuous $\Gamma_\Delta^*$-action and $|\Delta|>0$ arbitrary. We proceed by induction on $|\Delta|$. For a fixed element $\alpha\in\Delta$ consider the double complex $\Gamma_\alpha^\bullet(\Gamma_{\Delta\setminus\{\alpha\}}^\bullet(A))$ whose total complex is the cochain complex $\Gamma_\Delta^\bullet(A)$ by definition. There is a spectral sequence
$$E^{pq}_2=h^p\Gamma_\alpha^\bullet(h^q\Gamma_{\Delta\setminus\{\alpha\}}^\bullet(A))\Rightarrow h^{p+q}\Gamma_\Delta^\bullet(A)$$
associated to this double complex. By induction, $\Gamma_{\Delta\setminus\{\alpha\}}^\bullet(A)$ is acyclic in nonzero degrees with zeroth cohomology isomorphic to $H^0_{cont}(\Gamma_{\Delta\setminus\{\alpha\}}^*,A)$ which is an injective object in  the category of discrete $p$-primary abelian groups with continuous $\Gamma_\alpha^*$-action. Hence the spectral sequence degenerates at $E_1$ and $\Gamma_\Delta^\bullet(A)$ is acyclic outside degree zero where its cohomology is $H^0_{cont}(\Gamma^*_\Delta,A)$ by Step 1.

\emph{Step 4:} By Lemma \ref{deltafunctGamma} we have $H^0_{cont}(\Gamma^*_\Delta,\cdot)\cong h^0\Gamma_\Delta^\bullet(\cdot)$, so there is a unique map $H^n_{cont}(\Gamma^*_\Delta,\cdot)\to h^n\Gamma^\bullet_\Delta(\cdot)$ of cohomological $\delta$-functors as $H^n_{cont}(\Gamma^*_\Delta,\cdot)$ is a universal $\delta$-functor. The statement follows from Step 3 by dimension shifting.
\end{proof}

Now let $D$ be any object in $\varinjlim\mathcal{D}^{et}_{tors}(\varphi_{\Delta},\Gamma_\Delta,\mathcal{O}_{\mathcal{E}_\Delta})$. We define the cochain complex $\Phi\Gamma_\Delta^\bullet(D)$ as the total complex of the double complex $\Gamma_\Delta^\bullet(\Phi^\bullet(D^{C_\Delta}))$ and call it the \emph{Herr-complex} of $D$.

\begin{lem}\label{deltafunctPhiGamma}
The functors $(h^n\Phi\Gamma_\Delta^\bullet(\cdot))_{n\geq 0}$ form a cohomological $\delta$-functor from the category $\varinjlim\mathcal{D}^{et}_{tors}(\varphi_{\Delta},\Gamma_\Delta,\mathcal{O}_{\mathcal{E}_\Delta})$ to the category $\mathfrak{Ab}$ of abelian groups. Moreover, if $V$ is an object in $\operatorname{Rep}_{\Zp-tors}^{discr}(\GQpD)$, then we have $h^0\Phi\Gamma_\Delta^\bullet(\mathbb{D}(V))=V^{\GQpD}$.
\end{lem}
\begin{proof}
Given a short exact sequence $0\to D_1\to D_2\to D_3\to 0$ in $\varinjlim\mathcal{D}^{et}_{tors}(\varphi_{\Delta},\Gamma_\Delta,\mathcal{O}_{\mathcal{E}_\Delta})$, we obtain a short exact sequence of cochain complexes $0\to \Phi\Gamma_\Delta^\bullet(D_1)\to \Phi\Gamma_\Delta^\bullet(D_2)\to \Phi\Gamma_\Delta^\bullet(D_3)\to 0$ whose long exact sequence yields maps $\delta^n\colon h^n\Phi\Gamma_\Delta^\bullet(D_3)\to h^{n+1}\Phi\Gamma_\Delta^\bullet(D_1)$ that are functorial in the short exact sequence $0\to D_1\to D_2\to D_3\to 0$.

The second statement is a combination of Cor.\ \ref{HQpD*cohom} and Prop.\ \ref{Gamma*cohom} (both only used in degree $0$).
\end{proof}

\begin{thm}\label{herrcomplexmodp}
Let $V$ be an object in $\operatorname{Rep}_{\Zp-tors}^{discr}(\GQpD)$. The Herr complex $\Phi\Gamma_\Delta^\bullet(\mathbb{D}(V))$ computes the Galois cohomology of $\GQpD$ with coefficients in $V$, ie.\ we have an isomorphism $H^i(\GQpD,V)\cong h^i\Phi\Gamma_\Delta^\bullet(\mathbb{D}(V))$ natural in $V$ for all $i\geq 0$.
\end{thm}
\begin{proof}
Since $(H^n(\GQpD,\cdot))_{n\geq 0}$ is a universal $\delta$-functor, and $(h^n\Phi\Gamma_\Delta^\bullet(\mathbb{D}(\cdot)))_{n\geq 0}$ is a $\delta$-functor such that $H^0(\GQpD,\cdot)\cong h^0\Phi\Gamma_\Delta^\bullet(\mathbb{D}(\cdot))$, we obtain a natural transformation $H^n(\GQpD,\cdot)\to h^n\Phi\Gamma_\Delta^\bullet(\mathbb{D}(\cdot))$ of $\delta$-functors. Assume first that $V$ is injective in $\operatorname{Rep}_{\Zp-tors}^{discr}(\GQpD)$. We have a spectral sequence $$E^{pq}_2=h^p\Gamma_\Delta^\bullet(h^q\Phi^\bullet(\mathbb{D}(V)^{C_\Delta}))\Rightarrow h^{p+q}\Phi\Gamma_\Delta^\bullet(\mathbb{D}(V)) $$ associated to the double complex $\Gamma_\Delta^\bullet(\Phi^\bullet(D^{C_\Delta}))$. By Cor.\ \ref{HQpD*cohom} and the injectivity of $V$ the augmentation map $V^{\HQpD^*}[0]\to\Phi^\bullet(\mathbb{D}(V)^{C_\Delta})$ is a quasi-isomorphism. Moreover, $V^{\HQpD^*}$ is injective as a discrete representation of $\Gamma_\Delta^*$ whence $V^{\GQpD}[0]\to\Gamma_\Delta^\bullet(V^{\HQpD^*})$ is a quasi-isomorphism by Prop.\ \ref{Gamma*cohom}. Using the spectral sequence we deduce the statement in this case.

Now the case of general $V$ follows from Lemma \ref{deltafunctPhiGamma} by dimension shifting since the category $\operatorname{Rep}_{\Zp-tors}^{discr}(\GQpD)$ has enough injectives.
\end{proof}

\begin{rem}
If $V$ is a finite abelian $p$-group with a continuous action of $\GQpD$ then the cohomology groups $H^i(\GQpD,V)$ are finite for all $i\geq 0$. Indeed, this follows from the classical $|\Delta|=1$ case by the Hochschild--Serre spectral sequence.
\end{rem}

\subsection{The case $p=2$}\label{p=2}

We treat the case of $p=2$ here separately. We take this opportunity to mention that we find the literature on this a little unsatisfactory even in the classical case as the proof of the (modified) Herr complex computing Galois cohomology in Thm.\ 3.3 in \cite{Liu} is rather sketchy. In any case, our strategy is different from the one in the Tsinghua lecture notes \cite{Tsing} by Colmez.

Note that in this case we have $C_\Delta\cong \prod_{\alpha\in\Delta}C_\alpha$ where $C_\alpha$ is the group of order $2$ for each $\alpha\in\Delta$. Put $E_\Delta^*:=E_\Delta^{C_\Delta}$, $\OED^*:=\OED^{C_\Delta}$, and $E_\alpha^*:=E_\alpha^{C_\alpha}$ ($\alpha\in\Delta$). Now by a classical theorem of E.\ Artin on Galois theory, $E_\alpha/E_\alpha^*$ is a Galois extension of degree $2$ for each $\alpha\in\Delta$.

\begin{lem}\label{pncohomtrivp2}
We have $H^i_{cont}(\HQpD^*, \mathcal{O}_{\widehat{\mathcal{E}^{ur}_\Delta}}/p^n)=0$ for all $n\geq 1$ and $i\geq 1$.
\end{lem}
\begin{proof}
By devissage we are reduced to the case $n=1$ whence we have $\mathcal{O}_{\widehat{\mathcal{E}^{ur}_\Delta}}/p=E_\Delta^{sep}$. As an abstract field $E_\alpha^*$ ($\alpha\in\Delta$) is a local field of characteristic $2$ with residue field $\mathbb{F}_2$. By the classification of local fields, $E_\alpha^*$ is isomorphic to the field of formal Laurent series over $\mathbb{F}_2$, in particular, it is---non-canonically---isomorphic to $E_\alpha$. We fix such an isomorphism $\iota_\alpha\colon E_\alpha^*\overset{\sim}{\to}E_\alpha$ once and for all. Further, the natural inclusion $E_\alpha^*\subset E_\alpha\subset E_\alpha^{sep}$ is a separable closure of $E_\alpha^*$ since the extension $E_\alpha/E_\alpha^*$ is separable. Hence the absolute Galois group of $E_\alpha^*$ is $H_{\Qp,\alpha}^*$ which is therefore isomorphic to $\HQpa$ (being the absolute Galois group of $E_\alpha$) for all $\alpha\in\Delta$. We deduce $\HQpD^*\cong \HQpD$ by taking products. Moreover, the isomorphisms $\iota_\alpha$ ($\alpha\in\Delta$) yield an isomorphism $E_\Delta^*\cong E_\Delta$ as topological rings. Putting these together we obtain an automorphism $\iota\colon E_\Delta^{sep}\overset{\sim}{\to}E_\Delta^{sep}$ that---combined with the isomorphism $\HQpD^*\cong\HQpD$---induces an isomorphism between the pair $(E_\Delta^{sep},\HQpD^*)$ (ie.\ $E_\Delta^{sep}$ together with the action of $\HQpD^*$) and the pair $(E_\Delta^{sep},\HQpD)$. Once we have this isomorphism of pairs, we may apply Lemma \ref{pncohomtriv} in case $n=1$ to deduce the statement.
\end{proof}

Now we need the following
\begin{lem}\label{EDfingenp2}
Put $\Delta=\{\alpha_1,\dots,\alpha_n\}$. We have
\begin{equation*}
E_\Delta\cong E_{\alpha_1}\otimes_{E_{\alpha_1}^*}\left(E_{\alpha_2}\otimes_{E_{\alpha_2}^*}\left(\cdots(E_{\alpha_n}\otimes_{E_{\alpha_n}^*}E_\Delta^*)\right)\right)\ . 
\end{equation*}
In particular, $E_\Delta$ is a free module of rank $2^{|\Delta|}$ over $E_\Delta^*$. Moreover, we have $\Frac(E_\Delta)^{C_\Delta}=\Frac(E_\Delta^*)$.
\end{lem}
\begin{proof}
For the first statement we apply Lemma 3.2 in \cite{MultVarGal} in the situation $E_\Delta^*$ being the base, and $E_\Delta$ the extension. The containment $\Frac(E_\Delta)^{C_\Delta}\supseteq\Frac(E_\Delta^*)$ is clear. The other direction follows noting that the degrees $|E_\Delta:\Frac(E_\Delta)^{C_\Delta}|$ and $|E_\Delta:\Frac(E_\Delta^*)|$ are both equal to $2^{|\Delta|}$---one by E.\ Artin's theorem in Galois theory, the other by the first part.
\end{proof}

\begin{pro}\label{inducep2}
For any object $D$ in $\mathcal{D}^{et}_{tors}(\varphi_\Delta,\Gamma_\Delta,\OED)$ the natural map $f\colon\OED\otimes_{\OED^*}D^{C_\Delta}\to D$ is an isomorphism.
\end{pro}
\begin{proof}
By devissage we may assume without loss of generality that $2D=0$, ie.\ $D$ is an object in $\mathcal{D}^{et}(\varphi_\Delta,\Gamma_\Delta,E_\Delta)$. Note that $f$ is a morphism in $\mathcal{D}^{et}(\varphi_\Delta,\Gamma_\Delta,E_\Delta)$, so $\Ker(f)$ and $\Coker(f)$ are objects in $\mathcal{D}^{et}(\varphi_\Delta,\Gamma_\Delta,E_\Delta)$. In particular, they are free modules over $E_\Delta$ by Cor.\ 3.16 in \cite{MultVarGal}. Therefore it suffices to show that 
\begin{equation}\label{FracEDf}
\Frac(E_\Delta)\otimes f\colon \Frac(E_\Delta)\otimes_{\Frac(E_\Delta^*)}\Frac(D)^{C_\Delta}\to \Frac(D) 
\end{equation}
is an isomorphism where $\Frac(E_\Delta)$ (resp.\ $\Frac(E_\Delta^*)$) is the fraction field of $E_\Delta$ (resp.\ of $E_\Delta^*$) and $\Frac(D):=\Frac(E_\Delta)\otimes_{E_\Delta}D$. Now note that the $C_\Delta$-fixed part of the left hand side of \eqref{FracEDf} is also $\Frac(D)^{C_\Delta}$ which is the socle of the left hand side as a $C_\Delta$-representation since $C_\Delta$ is a $2$-group and $\Frac(E_\Delta^*)$ has characteristic $2$. Therefore $\Frac(E_\Delta)\otimes f$ is injective as a nontrivial kernel would intersect the socle nontrivially. For the surjectivity we show 
\begin{align*}
2^{|\Delta|}\dim_{\Frac(E_\Delta^*)}\Frac(D)=\dim_{\Frac(E_\Delta)}\Frac(D)\leq\\
\leq \dim_{\Frac(E_\Delta)}\Frac(E_\Delta)\otimes_{\Frac(E_\Delta^*)}\Frac(D)^{C_\Delta}= \dim_{\Frac(E_\Delta^*)}\Frac(D)^{C_\Delta}
\end{align*}
by induction on $|\Delta|$. Denote by $c_\alpha\in C_\alpha$ the nontrivial element for all $\alpha\in\Delta$. Then $(\id+c_\alpha)^2=0$ in $\Frac(E_\Delta^*)[C_\alpha]$, so as an operator on $\Frac(D)$ the image of $(\id+c_\alpha)$ is contained in its kernel $\Frac(D)^{C_\alpha}$. Therefore we have $\dim_{\Frac(E_\Delta^*)}\Frac(D)^{C_\alpha}\geq \frac{1}{2}\dim_{\Frac(E_\Delta^*)}\Frac(D)$. Iterating this for all $\alpha\in\Delta$ we deduce the statement.
\end{proof}

\begin{cor}
The complex $\Phi^\bullet(\mathbb{D}(V)^{C_\Delta})$ computes the $\HQpD^*$-cohomology of $V$, ie.\ Cor.\ \ref{HQpD*cohom} holds in case of $p=2$, too.
\end{cor}
\begin{proof}
By Lemma \ref{pncohomtrivp2} and Prop.\ \ref{inducep2} the proof of Prop.\ \ref{HcohomHerr} goes through to this statement, too.
\end{proof}

\begin{cor}\label{CDeltaexact}
The functor $V\mapsto \mathbb{D}(V)^{C_\Delta}$ is exact.
\end{cor}

\subsection{Tate duality}

Following III.7 in \cite{N} we make the following definitions for a profinite group $G$ with finite $p$-cohomological dimension $n$. For a (discrete) $G$-module $A$ we put $D_i(A):=\varinjlim_U H^i(U,A)^\vee$ where $U$ runs through the open normal subgroups of $G$ and $(\cdot)^\vee:=\Hom(\cdot,\mathbb{Q}/\mathbb{Z})$ stands for Pontryagin duality. The connecting maps in the inductive limit are the Pontryagin duals of the corestriction maps. Further, we define the \emph{dualizing module} of $G$ at $p$ by $I:=\varinjlim_h D_n(\mathbb{Z}/p^h\mathbb{Z})$. We have the functorial isomorphism 
\begin{equation*}
H^n(G,A)^\vee\cong \Hom_G(A,I)
\end{equation*}
for all $p$-primary discrete $G$-modules $A$. We call $G$ a \emph{duality group of dimension $n$} if $D_i(\mathbb{Z}/p\mathbb{Z})=0$ for all $i<n$. In this case the edge morphism for the Tate spectral sequence
\begin{equation*}
E_2^{p,q}=H^p(G,D_{n-q}(A))\Rightarrow H^{n-p-q}(G,A)^\vee
\end{equation*}
is a functorial isomorphism
\begin{equation*}
H^p(G,\Hom(A,I))\cong H^{n-p}(G,A)^\vee
\end{equation*}
for all $p$-primary discrete $G$-module $A$. This isomorphism is also obtained from the cup product
\begin{equation*}
H^p(G,\Hom(A,I))\times H^{n-p}(G,A)\overset{\cup}{\to} H^n(G,I)\to \Qp/\Zp\ .
\end{equation*}

A duality group of dimension $n$ is called a \emph{Poincar\'e group} at $p$ if the dualizing module $I$ is isomorphic to $\Qp/\Zp$ as an abelian group. The local duality theorem (7.2.6 in \cite{N}) states in particular, that the absolute Galois group $G_{\Qp}$ of $\Qp$ is a Poincar\'e group at $p$ of dimension $2$ with dualizing module $I=\mu_{p^\infty}$. Further, by Thm.\ 3.7.4 in \cite{N} the class of Poincar\'e groups at $p$ is closed under group extension. In particular, $\GQpD$ is also a Poincar\'e group at $p$ of dimension $2d$ where we put $d:=|\Delta|$. The dualizing module is  $I=\mu_{p^\infty,\Delta}$ (see Thm. 3.7.4(ii) in op.\ cit.) which is by definition the $\GQpD$-module isomorphic abstractly to $\mu_{p^\infty}$ (ie.\ to $\Qp/\Zp$) on which each component $\GQpa$ ($\alpha\in\Delta$) acts as on $\mu_{p^\infty}$ (ie.\ via the cyclotomic character).

Let $\Zp(\mathbbm{1}_\Delta):=T_p(\mu_{p^\infty,\Delta})=\varprojlim_n \mu_{p^n,\Delta}$ be the $p$-adic Tate module of $\mu_{p^\infty,\Delta}$ and for a $p$-primary discrete $\GQpD$-module $A$ we define the Tate twist $A(\mathbbm{1}_\Delta):=A\otimes_{\Zp}\Zp(\TD)$ and Tate dual $\Hom(A,\mu_{p^\infty,\Delta})=A^\vee(\TD)$.

\begin{thm}[Tate duality for $\GQpD$]\label{tate}
For any discrete $p$-primary $\GQpD$-module $A$ the cup product pairing induces an isomorphism $H^i(\GQpD,A)\cong H^{2d-i}(\GQpD,A^\vee(\TD))^\vee$.
\end{thm}

\subsection{Duality for $(\varphi_\Delta,\Gamma_\Delta)$-modules over $\OED$}\label{dualphigamma}

Let $D$ be an \'etale $(\varphi_\Delta,\Gamma_\Delta)$-module over $\OED$. Recall that the \'etale condition for the action of $\varphi_\alpha$ for an element $\alpha\in\Delta$ means that the map $\id\otimes\varphi_\alpha\colon \OED\otimes_{\OED,\varphi_\alpha}D\to D$ is bijective. Now $\OED$ is a free module over itself via the ring homomorphism $\varphi_\alpha$ with generators $\{(1+X_\alpha)^i\mid 0\leq i\leq p-1\}$. Therefore any $x\in D$ can uniquely be written as a sum $$x=\sum_{i=0}^{p-1}(1+X_\alpha)^i\varphi_\alpha(x_i)\ .$$
The distinguished left-inverse $\psi_\alpha$ of $\varphi_\alpha$ is defined as $\psi_\alpha(x):=x_0$.

Consider the multivariable $(\varphi_\Delta,\Gamma_\Delta)$-module $\DmpD$ corresponding to the $\GQpD$-module $\mu_{p^\infty,\Delta}$. We may identify $\DmpD$ with $\Qp/\Zp\otimes_{\Zp}(\OED e)=\mathcal{E}_\Delta e/\OED e$ where $\varphi_\alpha(e)=e$ and $\gamma_\alpha(e)=\chi_\alpha(\gamma_\alpha)e$ for all $\alpha\in\Delta$ and $\gamma_\alpha\in\Gamma_\alpha$. Here $\chi_\alpha\colon \Gamma_\alpha\overset{\sim}{\to}\Zp^\times$ stands for the cyclotomic character. Further, we define the residue map
\begin{equation*}
\res\colon \DmpD\to \Qp/\Zp
\end{equation*}
by sending an element $F(X_\bullet)e\in \DmpD$ to the coefficient $a_{-1_\bullet}\in\Qp/\Zp$ of $\frac{1}{X_\Delta}=\prod_{\alpha\in\Delta}X_\alpha^{-1}$ in the expansion of $\frac{F(X_\bullet)}{\prod_{\alpha\in\Delta}(1+X_\alpha)}$ as 
\begin{equation*}
\frac{F(X_\bullet)}{\prod_{\alpha\in\Delta}(1+X_\alpha)}=\sum_{i_\alpha\geq -N_F,\alpha\in\Delta}a_{i_\bullet}\prod_{\alpha\in\Delta}X_\alpha^{i_\alpha}
\end{equation*}
with $a_{i_\bullet}\in\Qp/\Zp$ for $i_\bullet=(i_\alpha)_{\alpha\in\Delta}\in\mathbb{Z}^\Delta$ and some integer $N_F\in \mathbb{Z}$ depending on $F$. (See I.2.3 in \cite{Mira} for the classical case $|\Delta|=1$.)

\begin{pro}\label{resphi}
We have $$\res(\gamma(\lambda))=\res(\varphi_\alpha(\lambda))=\res(\psi_\alpha(\lambda))=\res(\lambda)$$ for all $\lambda\in \DmpD$, $\gamma\in\Gamma_\Delta$, and $\alpha\in\Delta$.
\end{pro}
\begin{proof}
By $\Zp$-linearity and continuity of $\res$ we may assume without loss of generality that $\lambda=\prod_{\alpha\in\Delta}X_\alpha^{r_\alpha}e$ is a monomial for some $r_\alpha\in\mathbb{Z}$ ($\alpha\in\Delta$). For an element $\lambda_\alpha\in \mathcal{E}_\alpha/\mathcal{O}_{\mathcal{E}_\alpha}$ with some fixed $\alpha\in\Delta$ we denote by $\res_\alpha(\lambda_\alpha)\in \Qp/\Zp$ the coefficient of $X_\alpha^{-1}$ in the expansion of $\lambda_\alpha(1+X_\alpha)^{-1}\in\{\sum_{-\infty\ll i}a_iX_\alpha^i\mid a_i\in\Qp/\Zp\}$. Clearly, we have
$$\res(\prod_{\alpha\in\Delta}X_\alpha^{r_\alpha}e)=\prod_{\alpha\in\Delta}\res_\alpha(X_\alpha^{r_\alpha})\ .$$
So we are reduced to the case $|\Delta|=1$ which is covered e.g.\ by Prop.\ I.2.2 in \cite{Mira}.
\end{proof}

By Lemma 3.8 in \cite{MultVarGal} we have $\mathbb{D}(A^\vee(\TD))\cong \Hom(\mathbb{D}(A),\DmpD)$. For an \'etale $(\varphi_\Delta,\Gamma_\Delta)$-modules $D$ over $\OED$ we regard $D^*=\Hom(D,\mathcal{E}_\Delta/\OED)$ as an \'etale $(\varphi_\Delta,\Gamma_\Delta)$-module over $\OED$ the following way. First of all $\mathcal{E}_\Delta/\OED$ is a left and right module over $\OED$, and we regard $D$ as a left module, so $\Hom(D,\mathcal{E}_\Delta/\OED)$ becomes a right module over $\OED$ by the ``right multiplication on $\mathcal{E}_\Delta/\OED$''. Keeping in mind possible noncommutative generalizations we make $\Hom(D,\mathcal{E}_\Delta/\OED)$ into a left module over $\OED$ via the (anti-)involution $\#\colon \OED\to \OED$ sending the ``group elements'' $(1+X_\alpha)$ (ie.\ topological generators of $N_{\alpha,0}$ in the sense of \cite{MultVar}) to their inverse $(1+X_\alpha)^{-1}$ for all $\alpha\in\Delta$. This extends to an anti-involution to the whole ring $\OED$ by linearity and continuity. Further, for an $\OED$-linear map $f\colon D\to \mathcal{E}_\Delta/\OED$ we define $\varphi_\alpha(f)$ and $\gamma(f)$ ($\alpha\in\Delta$, $\gamma\in\Gamma_\Delta$) by the formulas $\varphi_\alpha(f)(\varphi_\alpha(x)):=\varphi_\alpha(f(x))$ and $\gamma(f)(\gamma(x)):=\gamma(f(x))$. The \'etale $(\varphi_\Delta,\Gamma_\Delta)$-module $D^*(\TD):=\Hom(D,\DmpD)$ has the same underlying $\varphi_\Delta$-module as $D^*$, but the action of $\Gamma_\Delta$ is twisted by the cyclotomic character.

\begin{rem}
Note that since $\OED$ is commutative, we could have omitted the anti-involution $\#$ when defining the left $\OED$-action on $D^*$ as done in \cite{Mira}. However, this way we do not need the modifying factor $\sigma_{-1}$ when defining the pairing $\{x,y\} \colon D\times D^*(\TD)\to \Qp/\Zp$: we can simply put $\{x,y\}:=\res(y(x))$ as we see below. Further, the $(\varphi_\Delta,\Gamma_\Delta)$-module $D^*$ is \emph{isomorphic} to the resulting $(\varphi_\Delta,\Gamma_\Delta)$-module not using the involution via the map defined by the multiplication by $\prod_{\alpha\in\Delta}\chi_\alpha^{-1}(-1)\in\Gamma_\Delta$.
\end{rem}

The following Lemma might be of independent interest.

\begin{lem}\label{decompose}
Let $D$ be a finitely generated $p$-power tosion \'etale $(\varphi_\Delta,\Gamma_\Delta)$-module over $\OED$. Then $D$ admits a decomposition $D\cong \bigoplus_{i=1}^{k}\OED/(p^{n_i})$ as a module over $\OED$.
\end{lem}
\begin{proof}
Since $D$ is finitely generated and torsion, we have $p^hD=0$ for some $h\geq 1$. We have the following filtration on the part $D[p]$ of $D$ killed by $p$: $$0=p^hD\cap D[p]\leq p^{h-1}D\cap D[p]\leq \dots \leq pD\cap D[p]\leq D[p]$$ consisting of \'etale $(\varphi_\Delta,\Gamma_\Delta)$-submodules. By Cor.\ 3.16 in \cite{MultVarGal} all the subquotients $(p^rD\cap D[p])/(p^{r+1}D\cap D[p])$ are free $\OED/(p)$-modules ($0\leq r\leq h-1$). So we may choose a basis $B_1\cup B_2\cup\dots\cup B_h$ of $D[p]$ such that for all $1\leq r$ the set $B_1\cup\dots \cup B_r$ is a $\OED/(p)$-basis of the module $p^{h-r}D\cap D[p]$. Now for each $1\leq r\leq h$ and $b\in B_r$ choose an element $b'\in D$ with $p^{h-r}b'=b$ and put $B'_r:=\{b'\in D\mid b\in B_r\}$. There is a surjective $\OED$-module homomorphism
\begin{equation*}
\bigoplus_{r=1}^{h}\bigoplus_{b'\in B_r}\OED/(p^{h-r+1})\twoheadrightarrow D
\end{equation*}
sending the generator of $\OED/(p^{h-r+1})$ to $b'_r$. This map is injective on the part killed by $p$ by construction therefore it is an isomorphism.
\end{proof}

For an \'etale $(\varphi_\Delta,\Gamma_\Delta)$-module $D$ we define the pairing
\begin{eqnarray*}
\{\cdot,\cdot\}\colon D\times D^*(\TD)&\to&\Qp/\Zp\\
(x,y)&\mapsto&\{x,y\}:=\res(y(x)) \ .
\end{eqnarray*}

\begin{pro}\label{adjoint}
Let $D$ be a finitely generated $p$-power torsion \'etale $(\varphi_\Delta,\Gamma_\Delta)$-module over $\OED$. Then the pairing $\{\cdot,\cdot\}$ is non-degenerate in the sense that the induced maps $D\to \Hom_{\Zp}(D^*(\TD),\Qp/\Zp)$ and $D^*(\TD)\to \Hom_{\Zp}(D,\Qp/\Zp)$ are injective. Moreover, we have 
\begin{eqnarray*}
\{x,\varphi_\alpha(y)\}=\{\psi_\alpha(x),y\}\ ,&&\{\varphi_\alpha(x),y\}=\{x,\psi_\alpha(y)\}\ ,\\ 
\{\gamma(x),\gamma(y)\}=\{x,y\}\text{ and }&&\{ux,uy\}=\{x,y\} 
\end{eqnarray*}
for all $\alpha\in\Delta$, $\gamma\in\Gamma$, $x\in D$, $y\in D^*(\TD)$, $u\in N_{\Delta,0}=\prod_{\alpha\in\Delta}(1+X_\alpha)^{\Zp}\subset \OED$.
\end{pro}
\begin{proof}
For any nonzero element $x\in D$ there exists an element $y\in D^*(\TD)$ such that $0\neq y(x)\in \mathbb{D}(\mu_{p^\infty,\Delta})$ by Lemma \ref{decompose} (also by noting that $D\cong (D^*(\TD))^*(\TD)$ by Thm.\ \ref{tate} and Thm.\ 3.15 in \cite{MultVarGal}). Further multiplying $y$ by a monomial $\prod_{\alpha\in\Delta}X_\alpha^{r_\alpha}$ ($r_\alpha\in\mathbb{Z}$, $\alpha\in\Delta$) we may ensure that the required coefficient $\{x,y\}$ is nonzero. Therefore the injectivity of the map $D\to \Hom_{\Zp}(D^*(\TD),\Qp/\Zp)$. The other statement follows similarly.

By the \'etale condition we may write $x=\sum_{i=0}^{p-1}(1+X_\alpha)^i\varphi_\alpha\circ\psi_\alpha((1+X_\alpha)^{-i}x)$ for all $\alpha\in\Delta$. So we compute
\begin{align*}
\{x,\varphi_\alpha(y)\}=\res(\varphi_\alpha(y)(x))=\res(\varphi_\alpha(y)(\sum_{i=0}^{p-1}(1+X_\alpha)^i\varphi_\alpha\circ\psi_\alpha((1+X_\alpha)^{-i}x)))=\\
=\sum_{i=0}^{p-1}\res((1+X_\alpha)^i\cdot\varphi_\alpha(y)(\varphi_\alpha\circ\psi_\alpha((1+X_\alpha)^{-i}x)))=\\
=\sum_{i=0}^{p-1}\res((1+X_\alpha)^i\cdot\varphi_\alpha(y(\psi_\alpha((1+X_\alpha)^{-i}x))))=\\
=\sum_{i=0}^{p-1}\res(\psi_\alpha((1+X_\alpha)^i\cdot\varphi_\alpha(y(\psi_\alpha((1+X_\alpha)^{-i}x)))))=\res(y(\psi_\alpha(x)))=\{\psi_\alpha(x),y\}
\end{align*}
using Prop.\ \ref{resphi} and the definition of $\varphi_\alpha$ on the dual $(\varphi_\Delta,\Gamma_\Delta)$-module $D^*(\TD)$. The other formulas follow similarly and more easily.
\end{proof}

Let $D$ be a finitely generated $p$-power torsion \'etale $(\varphi_\Delta,\Gamma_\Delta)$-module over $\OED$. An $\OED^+:=\Zp\bs X_\alpha,\alpha\in\Delta\js$-lattice in $D$ is a finitely generated $\OED^+$-submodule $M\subset D$ such that $D=M[X_\Delta^{-1}]$. We define the \emph{duality topology} on $D$ using the sets $$\sum_{\alpha\in\Delta}X_\alpha^NM[X_{\Delta\setminus\{\alpha\}}^{-1}]$$ for all $N>0$ as a system of neighbourhoods of $0$: a subset $U\subseteq D$ is declared to be open if for all $x\in U$ there is an integer $N>0$ such that $x+\sum_{\alpha\in\Delta}X_\alpha^NM[X_{\Delta\setminus\{\alpha\}}^{-1}]\subseteq U$. If $D$ is a finitely generated \'etale $(\varphi_\Delta,\Gamma_\Delta)$-module over $\OED$ then we define the duality topology on $D$ as the projective limit topology of the duality topologies on $D/p^nD$ ($n>0$). The principal goal of introducing this new topology is to describe the image of the inclusion $D^*(\TD)\hookrightarrow \Hom_{\Zp}(D,\Qp/\Zp)$ (see Prop.\ \ref{dualtopcont}). 

\begin{lem}\label{dualtopindep}
The duality topology does not depend on the choice of the $\OED^+$-lattice $M$.
\end{lem}
\begin{proof}
Note that if $M'\subset D$ is another $\OED^+$-lattice in $D$ then there exists an integer $k>0$ such that $X_\Delta^{k}M'\subseteq M\subseteq X_\Delta^{-k}M'$.
\end{proof}

Recall that the Iwasawa algebra $\OED^+$ is a local ring with maximal ideal $\operatorname{Jac}(\OED^+)$ generated by $p$ and $X_\alpha$ ($\alpha\in \Delta$) and residue field $\Fp\cong \OED^+/\operatorname{Jac}(\OED^+)$. Moreover, it is complete with respect to the filtration induced by the powers of $\operatorname{Jac}(\OED^+)$ therefore it is a pseudocompact ring (see chapter $22$ in \cite{SchpadicLie}). In particular, any finitely generated $\OED^+$-module admits a canonical topology which coincides with the $\operatorname{Jac}(\OED^+)$-adic topology. Since the residue field $\Fp$ is finite, any finitely generated $\OED^+$-module is the projective limit of finite modules hence it is compact in the canonical topology.

\begin{pro}\label{dualtophausdorff}
The duality topology on a finitely generated $p$-power torsion \'etale $(\varphi_\Delta,\Gamma_\Delta)$-module $D$ over $\OED$ induces the canonical compact topology on each finitely generated $\OED^+$-submodule of $D$. In particular, the duality topology is Hausdorff.
\end{pro}
\begin{proof}
By Lemma \ref{decompose} we may write $D$ as a direct sum $D\cong \bigoplus_{i=1}^{k}\OED/(p^{n_i})e_i$ where $e_i\in D$ ($1\leq i\leq k$) are generators such that $\OED/(p^{n_i})e_i\cong \OED/(p^{n_i})$. Then $M_0:=\bigoplus_{i=1}^{k}\OED^+/(p^{n_i})e_i$ is an $\OED^+$-lattice in $D$ and by Lemma \ref{dualtopindep} we may define the duality topology using $M_0$. Further, the $\Zp$-linear projection map $\pi_i\colon \OED/(p^{n_i})\to \OED^+/(p^{n_i})$ ($1\leq i\leq k$) having all those monomials $\prod_{\alpha\in\Delta}X_\alpha^{j_\alpha}$ with $j_\alpha<0$ for at least one $\alpha\in\Delta$ in the kernel induces a $\Zp$-linear projection map $\pi_{M_0}\colon D\to M_0$ whose restriction to $M_0$ is the identity. Comparing the coefficients we find that $\pi_{M_0}(X_\alpha^NM_0[X_{\Delta\setminus\{\alpha\}}^{-1}])=X_\alpha^NM_0$, so we have
\begin{equation*}
M_0\cap\left(\sum_{\alpha\in\Delta}X_\alpha^NM_0[X_{\Delta\setminus\{\alpha\}}^{-1}]\right)=\sum_{\alpha\in\Delta}X_\alpha^NM_0
\end{equation*}
showing that the duality topology induces the natural compact topology on $M_0$. Similarly, the duality topology on $X_\Delta^{-k}M_0$ is the usual one for all $k>0$. Finally, the statement follows noting that any finitely generated $\OED^+$-submodule $M\subset D$ is contained in $X_\Delta^{-k}M_0$ for some $k>0$.
\end{proof}

\begin{pro}\label{dualtopcont}
Let $D$ be a finitely generated $p$-power torsion \'etale $(\varphi_\Delta,\Gamma_\Delta)$-module over $\OED$ and $f\colon D\to \Qp/\Zp$ be a $\Zp$-linear map. Then there exists an element $y\in D^*(\TD)$ such that $f(x)=\{x,y\}$ for all $x\in D$ if and only if $f$ is continuous in the duality topology. In particular, we have $\Zp$-linear bijections $D\overset{\sim}{\to}\Hom_{\Zp}^{cont}(D^*(\TD),\Qp/\Zp)$ and $D^*(\TD)\overset{\sim}{\to}\Hom_{\Zp}^{cont}(D,\Qp/\Zp)$ in the duality topology.
\end{pro}
\begin{proof}
Using Lemma \ref{decompose} we write $D$ as a direct sum $\bigoplus_{i=1}^k\OED/(p^{n_i})e_i$ and put $h:=\max_i(n_i)$ so that $D$ is a module over $\OED/(p^h)$. Then the pairing $\{\cdot,\cdot\}$ on $D\times D^*(\TD)$ has values in $\Zp/(p^h)\cong p^{-h}\Zp/\Zp\subset \Qp/\Zp$. Further, $D^*(\TD)=\Hom(D,\mathbb{D}(\mu_{p^h,\Delta}))$ has a dual basis $b_1,\dots,b_k$ defined by the formula
\begin{equation*}
b_i(e_j)=\begin{cases}p^{h-n_i}\prod_{\alpha\in\Delta}\frac{1+X_\alpha}{X_\alpha}\in \mathbb{D}(\mu_{p^h,\Delta})=\OED/(p^h)(\TD)&\text{if }i=j\\ 0&\text{if }i\neq j\ ,\end{cases}
\end{equation*}
so we have $D^*(\TD)\cong \bigoplus_{i=1}^k\OED/(p^{n_i})b_i$. We put $M_0:=\bigoplus_{i=1}^k\OED^+/(p^{n_i})e_i$ and $M_0^*:=\bigoplus_{i=1}^k\OED/(p^{n_i})b_i$. Now if $y\in D^*(\TD)$ is arbitrary, then it is contained in $X_\Delta^{-N+1}M_0^*$ for some integer $N>0$. Let $x$ be in $X_\alpha^NM_0[X_{\Delta\setminus\{\alpha\}}^{-1}]$ for some $\alpha\in\Delta$. Then we have 
\begin{align*}
y(x)=((1+X_\alpha)^{-1}-1)^Ny)(X_\alpha^{-N}x)\in X_\alpha M_0^*(M_0[X_{\Delta\setminus\{\alpha\}}^{-1}])=\\
=X_\alpha\prod_{\alpha\in\Delta}\frac{1+X_\alpha}{X_\alpha}\OED^+/(p^h)[X_{\Delta\setminus\{\alpha\}}^{-1}]\subset \mathbb{D}(\mu_{p^h,\Delta})
\end{align*}
so that the exponent of $X_\alpha$ is nonnegative in all the monomials contained in the expansion of $\frac{y(x)}{\prod_{\alpha\in\Delta}(1+X_\alpha)}$. In particular, $\{x,y\}=0$, ie.\ $\{\cdot,y\}$ vanishes on $\sum_{\alpha\in\Delta}X_\alpha^NM_0[X_{\Delta\setminus\{\alpha\}}^{-1}]$. Therefore the kernel of $\{\cdot,y\}$ is open in the duality topology showing the continuity of $\{\cdot,y\}$.

Conversely, assume that $f\colon D\to \Zp/(p^h)$ is a $\Zp$-linear function that is continuous in the duality topology. Since the topology on $\Zp/(p^h)$ is discrete, this means that $$\sum_{\alpha\in\Delta}X_\alpha^NM_0[X_{\Delta\setminus\{\alpha\}}^{-1}]\subseteq \Ker(f)$$ for some $N>0$. We define
\begin{equation*}
y:=\sum_{i=1}^k\sum_{r=(r_\alpha)_\alpha\in\mathbb{Z}^\Delta}f(\prod_{\alpha\in\Delta}((1+X_\alpha)^{-1}-1)^{-r_\alpha}e_i)\prod_{\alpha\in\Delta}X_\alpha^{r_\alpha}b_i\ .
\end{equation*}
Note that if $-r_\alpha\geq N$ for some $\alpha\in\Delta$ then we have $\prod_{\alpha\in\Delta}((1+X_\alpha)^{-1}-1)^{-r_\alpha}e_i\in X_\alpha^NM_0[X_{\Delta\setminus\{\alpha\}}^{-1}]$ whence $f(\prod_{\alpha\in\Delta}((1+X_\alpha)^{-1}-1)^{-r_\alpha}e_i)=0$ by our assumption on $f$. Therefore the above formal sum indeed defines an element $y\in D^*(\TD)$. Finally, we have $\{x,y\}=f(x)$ by construction: this is true for elements of the form $\prod_{\alpha\in\Delta}((1+X_\alpha)^{-1}-1)^{-r_\alpha}e_i\in D$ for some $(r_\alpha)_\alpha\in \mathbb{Z}^\Delta$ and $1\leq i\leq k$ and any $x\in D$ can be written as a finite sum of elements of this form modulo $\sum_{\alpha\in\Delta}X_\alpha^NM_0[X_{\Delta\setminus\{\alpha\}}^{-1}]$ by Prop.\ \ref{dualtophausdorff}.
\end{proof}

Even though the pairing $\{\cdot,\cdot\}$ is separately continuous in the duality topologies on $D$ and $D^*(\TD)$, it is \emph{not} jointly continuous. However, the situation is better in these terms if we choose the \emph{weak topology} on both $D$ and $D^*(\TD)$: this is the inductive limit topology of the compact spaces $X_\Delta^{-n}M$ for some $\OED^+$-lattice $M\subset D$. Note that the weak topology does not depend on our choice of the lattice $M$ either. If $D$ is a finitely generated \'etale $(\varphi_\Delta,\Gamma_\Delta)$-module over $\OED$ then we define the weak topology on $D$ as the projective limit topology of the weak topologies on $D/p^nD$ ($n>0$).

\begin{pro}
Assume that $D$ is a finitely generated $p$-power torsion \'etale $(\varphi_\Delta,\Gamma_\Delta)$-module over $\OED$. The pairing $\{\cdot,\cdot\}$ is (jointly) continuous in the weak topology.
\end{pro}
\begin{proof}
By the definition of the inductive limit topology, it suffices to show that the restriction of the pairing to $M\times M'$ is continuous for any pair of $\OED^+$-lattices $M\subset D$ and $M'\subset D^*(\TD)$. However, for any fixed lattice $M\subset D$, the proof of Prop.\ \ref{dualtopcont} shows that there exists an integer $N>0$ such that $\{\cdot,\cdot\}$ is identically $0$ on $M\times \sum_{\alpha\in\Delta}X_\alpha^NM'[X_{\Delta\setminus\{\alpha\}}^{-1}]$ (resp.\ on $\sum_{\alpha\in\Delta}X_\alpha^NM[X_{\Delta\setminus\{\alpha\}}^{-1}]\times M'$), therefore also on the open subset $\sum_{\alpha\in\Delta}X_\alpha^NM\times \sum_{\alpha\in\Delta}X_\alpha^NM'$ of $M\times M'$.
\end{proof}

\begin{rem}
Note that the ring $E_\Delta=\OED/(p)$ is not locally compact for $|\Delta|>1$. Therefore the above pairing is definitely \emph{not} perfect for $|\Delta|>1$ (ie.\ the map $D\to \Hom^{cont,weak}_{\Zp}(D^*(\TD),\Qp/\Zp)$ is not a bijection) by \cite{MP}.  Consequently, the duality topology is strictly weaker than the weak topology.
\end{rem}

\subsection{Iwasawa cohomology}

Let $A$ be a finite $p$-power torsion abelian group with a continuous action of $\GQpD$. The Iwasawa cohomology groups $H^i_{Iw}(\GQpD,A)$ are defined as the projective limits
\begin{equation*}
H^i_{Iw}(\GQpD,A):=\varprojlim_{\HQpD\leq H\leq_o\GQpD}H^i(H,A)
\end{equation*}
where the transition maps are the cohomological corestriction maps and $H$ runs through all open subgroups of $\GQpD$ containing $\HQpD$. The Iwasawa cohomology groups naturally have the structure of modules over the Iwasawa algebra $\Zp\bs \Gamma_\Delta\js$. By Shapiro's Lemma we have the identification $H^i(H,A)\cong H^i(\GQpD,\Zp[\GQpD/H]\otimes_{\Zp}A)$ where $\GQpD$ acts diagonally on the right hand side. 
\begin{lem}\label{measure}
We have $H^i_{Iw}(\GQpD,A)\cong H^i(\GQpD,\Zp\bs \Gamma_\Delta\js\otimes_{\Zp}A)$ where the right hand side refers to continuous cochains via the diagonal action of $\GQpD$ on the coefficients.
\end{lem}
\begin{proof}
This is entirely analogous to the proof of Lemma 5.8 in \cite{SchVen}.
\end{proof}

By Theorem \ref{tate} we may further identify these cohomology groups using the Tate dual $A^\vee(\TD)$ as follows: 
\begin{align*}
H^i(\GQpD,\Zp[\GQpD/H]\otimes_{\Zp}A)\cong H^{2d-i}(\GQpD,(\Zp[\GQpD/H]\otimes_{\Zp}A)^\vee(\TD))^\vee\cong\\
\cong H^{2d-i}(\GQpD,\Zp[\GQpD/H]\otimes_{\Zp}(A^\vee(\TD)))^\vee\cong H^{2d-i}(H,A^\vee(\TD))^\vee
\end{align*}
since the index $|\GQpD:H|$ is finite. The Tate duals of the corestriction maps are the restriction maps, so we deduce
\begin{equation*}
H^i_{Iw}(\GQpD,A)\cong (\varinjlim_H H^{2d-i}(H,A^\vee(\TD)))^\vee=H^{2d-i}(\HQpD,A^\vee(\TD))^\vee\ .
\end{equation*}

Moreover, the complex $\Phi^\bullet\mathbb{D}(A^\vee(\TD))$ computes the $\HQpD$-cohomology of $A^\vee(\TD)$ by Proposition \ref{HcohomHerr} showing
\begin{equation*}
H^i_{Iw}(\GQpD,A)\cong (h^{2d-i}\Phi^\bullet\mathbb{D}(A^\vee(\TD)))^\vee\ .
\end{equation*}

In particular, $H^i_{Iw}(\GQpD,A)=0$ unless $d\leq i\leq 2d$ since $\Phi^\bullet\mathbb{D}(A^\vee(\TD))$ is concentrated into degrees $0$ to $d$. Our goal is to identify the above Iwasawa cohomology groups in terms of the $(\varphi_\Delta,\Gamma_\Delta)$-module $\mathbb{D}(A)$ using the pairing $\{\cdot,\cdot\}$ defined in section \ref{dualphigamma}. For a $(\varphi_\Delta,\Gamma_\Delta)$-module $D$ over $\OED$ we define the cochain complex
\begin{align}\label{psicomplexdef}
\Psi^\bullet(D)\colon 0\to D\to \bigoplus_{\alpha\in\Delta}D\to \dots\to \bigoplus_{\{\alpha_1,\dots,\alpha_r\}\in \binom{\Delta}{r}}D\to\dots \to D\to 0
\end{align}
where for all $0\leq r\leq |\Delta|-1$ the map $d_{\alpha_1,\dots,\alpha_r}^{\beta_1,\dots,\beta_{r+1}}\colon D\to D$ from the component in the $r$th term corresponding to $\{\alpha_1,\dots,\alpha_r\}\subseteq \Delta$ to the component corresponding to the $(r+1)$-tuple $\{\beta_1,\dots,\beta_{r+1}\}\subseteq\Delta$ is given by
\begin{equation*}
d_{\alpha_1,\dots,\alpha_r}^{\beta_1,\dots,\beta_{r+1}}=\begin{cases}0&\text{if }\{\alpha_1,\dots,\alpha_r\}\not\subseteq\{\beta_1,\dots,\beta_{r+1}\}\\ (-1)^{\eta}(\id-\psi_\beta)&\text{if }\{\beta_1,\dots,\beta_{r+1}\}=\{\alpha_1,\dots,\alpha_r\}\cup\{\beta\}\ ,\end{cases}
\end{equation*}
where $\eta=\eta(\alpha_1,\dots,\alpha_r,\beta)$ is the number of elements in the set $\Delta\setminus\{\alpha_1,\dots,\alpha_r\}$ smaller than $\beta$. Note that the sign convention here is different from the one defining the complex $\Phi^\bullet(D)$. The reason for this is that this way the differentials are adjoint to each other under the pairing $\{\cdot,\cdot\}$ as we shall see later on. Anyway, the complex defined this way is quasi-isomorphic to the complex defined with the usual sign convention $\varepsilon$ instead of $\eta$.

For a subset $S\subset \Delta$ with $r:=|S|$ we consider the pairing $\{\cdot,\cdot\}$ between the copy of $D^*(\TD)$ in degree $r$ in the complex $\Phi^\bullet(D^*(\TD))$ corresponding to the subset $S$ and the copy of $D$ in degree $d-r$ in the complex $\Psi^\bullet(D)$ corresponding to the subset $\Delta\setminus S$. We extend this to a pairing
\begin{equation}
\{\cdot,\cdot\}\colon \Phi^r(D^*(\TD))\times \Psi^{d-r}(D)\to \Qp/\Zp\label{pairingcomplex}
\end{equation}
bilinearly for all $0\leq r\leq d$ with the convention that the copy of $D^*(\TD)$ corresponding to $S$ in $\Phi^r(D^*(\TD))=\bigoplus_{S\in\binom{\Delta}{r}}D^*(\TD)$ is orthogonal to all copies of $D$ in $\Psi^{d-r}(D)=\bigoplus_{S'\in\binom{\Delta}{d-r}}D$ corresponding to some $S'$ different from $\Delta\setminus S$. Even though $\{\cdot,\cdot\}$ is not perfect if $|\Delta|>1$, we have the following main result whose proof will occupy the rest of the section.

\begin{thm}\label{iwasawamodp}
The above pairing $\{\cdot,\cdot\}$ between the cochain complexes $\Phi^\bullet(D^*(\TD))$ and $\Psi^\bullet(D)$ induces a perfect pairing
\begin{equation}
\{\cdot,\cdot\}\colon h^r\Phi^\bullet(D^*(\TD))\times h^{d-r}\Psi^\bullet(D)\to \Qp/\Zp\label{pairingcohom}
\end{equation}
on the cohomology groups for all $0\leq r\leq d$. In particular, we have an isomorphism $$H^i_{Iw}(\GQpD,\cdot)\cong (h^{2d-i}\Phi^\bullet(\mathbb{D}((\cdot)^\vee(\TD))))^\vee\cong h^{i-d}\Psi^\bullet(\mathbb{D}(\cdot))$$
of cohomological $\delta$-functors.
\end{thm}
\begin{proof}
The proof is long and will occupy the rest of this section. We proceed in $3$ steps.

\emph{Step 1. We show that the pairing \eqref{pairingcohom} is well-defined.} Let $(y_S)_{S\in\binom{\Delta}{r}}\in \Phi^r(D^*(\TD))$ and $(x_U)_{U\in \binom{\Delta}{d-r-1}}\in \Psi^{d-r-1}(D)$ be arbitrary. We compute
\begin{align*}
\{d^r_\Phi((y_S)_S),(x_U)_U\}=\{\sum_{S\in\binom{\Delta}{r}}\sum_{\beta\in\Delta\setminus S}(-1)^{\varepsilon(\beta,S)}(y_S-\varphi_\beta(y_S))_{S\cup\{\beta\}},(x_U)_U\}=\\
=\sum_{S\in\binom{\Delta}{r}}\sum_{\beta\in\Delta\setminus S}(-1)^{\varepsilon(\beta,S)}\{y_S-\varphi_\beta(y_S),x_{\Delta\setminus (S\cup\{\beta\})}\}=\\ =\sum_{S\in\binom{\Delta}{r}}\sum_{\beta\in\Delta\setminus S}(-1)^{\varepsilon(\beta,S)}\{y_S,x_{\Delta\setminus (S\cup\{\beta\})}-\psi_\beta(x_{\Delta\setminus (S\cup\{\beta\})})\}=\\
=\sum_{U\in\binom{\Delta}{d-r-1}}\sum_{\beta\in\Delta\setminus U}(-1)^{\varepsilon(\beta,\Delta\setminus (U\cup\{\beta\}))}\{y_{\Delta\setminus (U\cup\{\beta\})},x_U-\psi_\beta(x_{U})\}=\\
=\{y_{\Delta\setminus (U\cup\{\beta\})},\sum_{U\in\binom{\Delta}{d-r-1}}\sum_{\beta\in\Delta\setminus U}(-1)^{\eta(\beta,U)}(x_U-\psi_\beta(x_{U}))\}=\{(y_S)_S,d^{d-r-1}_\Psi((x_U)_U)\}\\
\end{align*}
using Prop.\ \ref{adjoint}. We deduce $\Ker(d^r_\Phi)^\perp\supseteq \mathrm{Im}(d^{d-r-1}_\Psi)$ and similarly $\Ker(d^r_\Psi)^\perp\supseteq \mathrm{Im}(d^{d-r-1}_\Phi)$ (applying the above formula with $r$ replaced by $d-r-1$). A simple diagram chasing on the short exact sequences of cochain complexes 
\begin{align*}
0\to \Phi^\bullet D_3^*(\TD)\to \Phi^\bullet D_2^*(\TD)\to \Phi^\bullet D_1^*(\TD)\to 0\ ,\text{ and}\\
0\to \Psi^\bullet D_1 \to \Psi^\bullet D_2 \to \Psi^\bullet D_3 \to 0
\end{align*}
attached to a short exact sequence $0\to D_1\to D_2\to D_3\to 0$ shows that \eqref{pairingcohom} induces a morphism 
\begin{equation*}
h^{d-r}\Psi^\bullet(\cdot)\to  (h^r\Phi^\bullet((\cdot)^*(\TD)))^\vee
\end{equation*}
of cohomological $\delta$-functors.

\emph{Step 2. We show the statement in case $\mathbb{V}(D)$ is an absolutely irreducible representation of $\GQpD$ over some sufficiently large finite field $\kappa$ of characteristic $p$. This main Step will require several Lemmas and Propositions which might be of independent interest.} 

The following general group-theoretic lemma is quite important in the proof. Even though not all representations of $\GQpD$ over $\Fp$ are the tensor products of representations of $G_{\Qp}$ (in which case the proofs would be much simpler), every representation becomes a successive extension of such representations after passing to a sufficiently large finite field of characteristic $p$.

\begin{lem}\label{productoftwofinitegroups}
Let $G_1$ and $G_2$ be finite groups and $W$ be an absolutely irreducible finite dimensional representation of $G_1\times G_2$ over a field $F$ of arbitrary characteristic. Then there exists a finite extension $K/F$ such that $K\otimes_F W$ is isomorphic to the tensor product $W_1\otimes_K W_2$ where $W_i$ is an irreducible representation of $G_i$ over $K$ ($i=1,2$).
\end{lem}
\begin{proof}
At first note that since the actions of $G_1$ and $G_2$ on $W_2$ commute, the $G_1$-socle of $W$ is $G_2$-stable. By the irreducibility of $W$ as a representation of $G_1\times G_2$ we deduce that $W$ is semisimple as a representation of $G_1$. Further, the $G_1$-isotypical components of $W$ are also $G_2$-stable therefore there is only one such component. By passing to the splitting field $K$ of the restriction of $W$ to $G_1$ we obtain that $K\otimes_F W_{\mid G_1}\cong \bigoplus_{j=1}^k W_1$ is the direct sum of copies of an absolutely irreducible representation $W_1$ of $G_1$. By Schur's Lemma the ring of endomorphisms of $K\otimes_F W_{\mid G_1}$ is the full matrix ring $K^{k\times k}$. Since $G_2$ acts by $G_1$-automorphisms on $K\otimes_F W$ we obtain a representation $G_2\to \GL_k(K)$ (denoted by $W_2$) and an isomorphism $K\otimes_F W\cong W_1\otimes_K W_2$.
\end{proof}
Since $\mathbb{V}(D)$ factors through a finite quotient of $\GQpD$, we can apply Lemma \ref{productoftwofinitegroups} to our situation. By possibly extending $\kappa$ and using induction on $|\Delta|$ we may assume there exists a finite dimensional representation $V_\alpha$ of $G_{\Qp,\alpha}$ for all $\alpha\in\Delta$ such that $\mathbb{V}(D)\cong \bigotimes_{\alpha\in\Delta,\kappa}V_\alpha$. We denote by $D_\alpha:=\mathbb{D}(V_\alpha)$ the $(\varphi_\alpha,\Gamma_\alpha)$-module over the $1$-variable ring $E_\alpha\cong \kappa\bg X_\alpha\jg$ corresponding to the Galois representation $V_\alpha$. Recall (Prop.\ II.4.2 in \cite{Mira}) that there exists a unique $\psi_\alpha$- and $\Gamma_\alpha$-stable $E_\alpha^+\cong\kappa\bs X_\alpha\js$-lattice $D_\alpha^\#\subset D_\alpha$ characterized by the following properties:
\begin{enumerate}[$(i)$]
\item For all $x\in D_\alpha$ there exists an integer $n=n(x,D)$ such that $\psi_\alpha^n(x)$ lies in $D_\alpha^\#$.
\item $\psi_\alpha\colon D_\alpha^\#\to D_\alpha^\#$ is surjective.
\end{enumerate}
We define $D^\#$ as the completed tensor product 
\begin{align*}
D^\#:=\widehat{\bigotimes_{\alpha\in\Delta,\kappa}}D_\alpha^\#:=\varprojlim_n\left(\bigotimes_{\alpha\in\Delta,\kappa}D_\alpha^\#/X_\alpha^n\right) 
\end{align*}
of the $D_\alpha^\#$ over $\kappa$. We regard $D^\#$ as a finitely generated $E_\Delta^+:=\kappa\bs X_\alpha,\alpha\in\Delta\js$-submodule in $D$ so that we have $D=D^\#[X_\Delta^{-1}]$.
\begin{pro}\label{Dhashquasiiso}
The natural map $\Psi^\bullet(D^\#)\to \Psi^\bullet(D)$ induced by the inclusion $D^\#\hookrightarrow D$ is a quasi-isomorphism of cochain complexes.
\end{pro}
\begin{proof}
By Prop.\ II.5.5 and II.5.6 in \cite{Mira} the map $\psi_\alpha-1\colon D_\alpha/D_\alpha^\#\to D_\alpha/D_\alpha^\#$ is bijective for all $\alpha\in\Delta$ since $X-1\in \kappa[X]$ is a polynomial with invertible constant term. In particular, the case $|\Delta|=1$ follows. For a fixed ordering of the finite set $\Delta$ there is an induced filtration on $D$ (indexed by the ordered set $\Delta\cup \{0\}$ with $0<\alpha$ for all $\alpha\in\Delta$) by putting $\Fil^\alpha:=D^\#[X_{S_\alpha}^{-1}]$ and $\Fil^0:=D^\#$ where $S_\alpha:=\{\beta\in\Delta\mid \beta\leq\alpha\}$ and $X_{S_\alpha}=\prod_{\beta\in S_\alpha}X_\beta$.
\begin{lem}
The graded pieces of the above filtration split as
\begin{equation*}
\mathrm{gr}^\alpha D=D^\#[X_{S_\alpha}^{-1}]/D^\#[X_{S_\alpha\setminus\{\alpha\}}^{-1}]\cong D_{\Delta\setminus\{\alpha\}}^\#[X_{S_\alpha\setminus\{\alpha\}}^{-1}]\otimes_\kappa (D_\alpha/D^\#_\alpha)
\end{equation*}
for any $\alpha\in\Delta$. Here $D_{\Delta\setminus\{\alpha\}}$ denotes the $(\varphi_{\Delta\setminus\{\alpha\}},\Gamma_{\Delta\setminus\{\alpha\}})$-module $\mathbb{D}(\bigotimes_{\beta\in\Delta\setminus\{\alpha\},\kappa}V_\beta)$ over $E_{\Delta\setminus\{\alpha\}}$.
\end{lem}
\begin{proof}
Since the direct limit is exact, we may write $$D^\#[X_{S_\alpha}^{-1}]/D^\#[X_{S_\alpha\setminus\{\alpha\}}^{-1}]=\varinjlim_i\varinjlim_j (X_\alpha^{-i}X_{S_\alpha\setminus\{\alpha\}}^{-j}D^\#)/(X_{S_\alpha\setminus\{\alpha\}}^{-j}D^\#)\ .$$ Further, by construction we compute 
\begin{align*}
X_\alpha^{-i}D^\#/D^\#\cong D^\#/X_\alpha^iD^\#\cong \varprojlim_n \left(\left(\bigotimes_{\beta\in\Delta\setminus\{\alpha\},\kappa}D^\#_\beta/X_\beta^n D^\#_\beta\right) \otimes_\kappa D_\alpha^\#/X_\alpha^iD^\#_\alpha \right)\cong\\
\cong D_{\Delta\setminus\{\alpha\}}^\#\otimes_\kappa(X_\alpha^{-i}D^\#_\alpha/D^\#_\alpha)
\end{align*}
since $X_\alpha^{-i}D^\#_\alpha/D^\#_\alpha$ is a finite dimensional $\kappa$-vector space whence $\cdot\otimes_\kappa X_\alpha^{-i}D^\#_\alpha/D^\#_\alpha$ commutes with inverse limits. Similarly, multiplication by $X_{S_\alpha\setminus\{\alpha\}}^{-j}$ yields the identification $$X_\alpha^{-i}X_{S_\alpha\setminus\{\alpha\}}^{-j}D^\#/X_{S_\alpha\setminus\{\alpha\}}^{-j}D^\#\cong X_{S_\alpha\setminus\{\alpha\}}^{-j}D_{\Delta\setminus\{\alpha\}}^\#\otimes_\kappa(X_\alpha^{-i}D^\#_\alpha/D^\#_\alpha)\ .$$ The statement follows taking the direct limit which commutes with tensor products.
\end{proof}
In view of the above Lemma the cochain complex $\Psi^\bullet(\mathrm{gr}^\alpha D)$ splits as the tensor product $$\Psi^\bullet(D_{\Delta\setminus\{\alpha\}}^\#[X_{S_\alpha\setminus\{\alpha\}}^{-1}])\otimes_\kappa \Psi^\bullet(D_\alpha/D^\#_\alpha)\ .$$ In particular, it is acyclic for any $\alpha\in\Delta$ since so is the complex $\Psi^\bullet(D_\alpha/D_\alpha^\#)$. By (finite) induction we deduce that the cochain complex $\Psi^\bullet(D/D^\#)$ is acyclic, too, whence the inclusion $\Psi^\bullet(D^\#)\to \Psi^\bullet(D)$ of complexes is a quasi-isomorphism.
\end{proof}

\begin{cor}\label{compactcohom}
The cohomology groups $h^i\Psi^\bullet(D)$ ($0\leq i\leq d$) are compact in the topology induced by the weak (resp.\ by the duality) topology on $D$.
\end{cor}
\begin{proof}
For each $\alpha\in\Delta$ the map $\psi_\alpha-1\colon D^\#\to D^\#$ is continuous by construction since it is continuous on $D_\alpha^\#$. Therefore the differentials in the complex $\Psi^\bullet(D^\#)$ are continuous and even strict by the compactness of $D^\#$. The result follows from Prop.\ \ref{Dhashquasiiso} noting that both the duality and the weak topologies induce the natural compact topology on $D^\#$.
\end{proof}

We denote by $D^\#_0$ the (uncompleted) tensor product $D^\#_0:=\bigotimes_{\alpha\in\Delta,\kappa}D^\#_\alpha$ which is a module over the ring $\bigotimes_{\alpha\in\Delta,\kappa}\kappa\bs X_\alpha\js$. It admits operators $\psi_\alpha$ for all $\alpha\in\Delta$ acting on the respective terms. The complex $\Psi^\bullet(D^\#_0)$ is by construction the tensor product of the complexes $\Psi^\bullet(D^\#_\alpha)$ over $\kappa$. Note that the natural map $D^\#_0\to D^\#$ is injective with dense image since $D^\#_\alpha$ is a finite free $\kappa\bs X_\alpha\js$-module for all $\alpha\in\Delta$. This inclusion induces a morphism $\Psi^\bullet(D^\#_0)\to \Psi^\bullet(D^\#)$ of cochain complexes.

\begin{pro}\label{tensordense}
The image of $\Ker(d^r_\Psi\colon \Psi^r(D^\#_0)\to \Psi^{r+1}(D^\#_0))$ is dense in $\Ker(d^r_\Psi\colon  \Psi^r(D^\#)\to \Psi^{r+1}(D^\#))$ for all $0\leq r\leq d$. Consequently, the induced map $h^r\Psi^\bullet(D^\#_0)\to h^r\Psi^\bullet(D^\#)$ also has dense image.
\end{pro}
\begin{proof}
In order to simplify notation we identify $D^\#_0$ with its image in $D^\#$. Let $(x_S)_{S\in \binom{\Delta}{r}}\in \Ker(d^r_\Psi\colon  \Psi^r(D^\#)\to \Psi^{r+1}(D^\#))$ and $N\in\mathbb{N}$ be arbitrary. By the density of $D^\#_0$ in $D^\#$ there exists an element $(y_S)_{S\in \binom{\Delta}{r}}\in \Psi^r(D^\#_0)$ such that $y_S-x_S$ lies in $\sum_{\alpha\in\Delta}X_\alpha^{pN}D^\#$ for all $S\in \binom{\Delta}{r}$. In particular, we have $$d^r_\Psi((y_S)_S)=d^r_\Psi((y_S)_S-(x_S)_S)\in d^r_\Psi(\bigoplus_{S\in\binom{\Delta}{r}}\sum_{\alpha\in\Delta}X_\alpha^{pN}D^\#)\subseteq \bigoplus_{S'\in\binom{\Delta}{r+1}}\sum_{\alpha\in\Delta}X_\alpha^{N}D^\#\ .$$
We claim that there exists an element $(y'_S)_{S\in \binom{\Delta}{r}}\in \bigoplus_{S\in\binom{\Delta}{r}}\sum_{\alpha\in\Delta}X_\alpha^{N-n_0}D^\#$ for some fixed integer $n_0=n_0(D)$ depending only on $D$ with $d^r_\Psi((y_S)_S)=d^r_\Psi((y'_S)_S)$ so that $x_S-(y_S-y'_S)\in \sum_{\alpha\in\Delta}X_\alpha^{N-n_0}D^\#$ such that $(y_S-y'_S)_S$ lies in $\Ker(d^r_\Psi\colon \Psi^r(D^\#_0)\to \Psi^{r+1}(D^\#_0))$. Equivalently, we state
\begin{lem}
We have
\begin{equation*}
\mathrm{Im}\left(d^r_\Psi\colon \Psi^r(D^\#_0)\to \Psi^{r+1}(D^\#_0)\right)\bigcap \left(\bigoplus_{U\in\binom{\Delta}{r+1}}\sum_{\alpha\in\Delta}X_\alpha^{p^{2^d}N}D_0^\#\right)\subseteq d^r_\Psi\left(\bigoplus_{S\in\binom{\Delta}{r}}\sum_{\alpha\in\Delta}X_\alpha^{N}D^\#_0\right)
\end{equation*}
for all $0\leq r\leq d$ and $N\geq n_0$ with some integer $n_0$ depending only on $D$.
\end{lem}
\begin{rem}
The above Lemma states in a quantitative way that the map $d^r_\Psi$ from  $\Psi^r(D^\#_0)$ onto its image is open, ie.\ that the differentials in the complex $\Psi^\bullet(D^\#_0)$ are strict. The analogous statement for $\Psi^\bullet(D^\#)$ is clear from the compactness.
\end{rem}
\begin{proof}
We proceed by induction on $d=|\Delta|$. For $d=1$ there exists an integer $n_0$ such that $X^{n_0}D^\#\subseteq D^{++}$. So if $x\in X^{p^2N}D^\#$ is arbitrary for some $p^2N\geq (N\geq) n_0$ then $z:=\sum_{i=1}^\infty\varphi^i(x)$ converges in $D$ so that $\psi(z)-z=x$. Further, we have $z\in \varphi(X^{p^2N-n_0}D^{++})\subset X^{p^3N-pn_0}D^\#\subseteq X^{N}D^\#$.

Now let $d>1$ and pick an $\alpha$ in $\Delta$. In order to use induction, we separate those subsets of $\Delta$ containing $\alpha$ from those not containing $\alpha$. We write $\Psi^\bullet(D^\#_0)$ as the tensor product of the complexes $\Psi^\bullet(D^\#_\alpha)=(D^\#_\alpha\overset{\psi_\alpha-1}{\to}D^\#_\alpha)$ and $(\Psi^\bullet(D^\#_{\Delta\setminus\{\alpha\},0}),d^\bullet_{\Delta\setminus\{\alpha\}})$ where $D^\#_{\Delta\setminus\{\alpha\},0}:=\bigotimes_{\beta\in\Delta\setminus\{\alpha\},\kappa}D^\#_\beta$. So the differential $d^r_\Psi\colon \Psi^r(D^\#_0)\to \Psi^{r+1}(D^\#_0)$ is split into $3$ maps (upto sign): 
\begin{equation}\label{keydiagramalpha}
\xymatrixcolsep{5pc}\xymatrix{
\bigoplus\limits_{\alpha\notin S\in\binom{\Delta}{r}}D^\#_0 \ar[r]^{\id_{D_\alpha^\#}\otimes d^r_{\Delta\setminus\{\alpha\}}}\ar[rd]|-{(\psi_\alpha-1)\otimes\id_{D^\#_{\Delta\setminus\{\alpha\},0}}} & \bigoplus\limits_{\alpha\notin U\in\binom{\Delta}{r+1}}D^\#_0\\
\bigoplus\limits_{\alpha\in S\in\binom{\Delta}{r}}D^\#_0\ar[r]_{\id_{D_\alpha^\#}\otimes d^{r-1}_{\Delta\setminus\{\alpha\}}} & \bigoplus\limits_{\alpha\in U\in\binom{\Delta}{r+1}}D^\#_0
}
\end{equation}
Let $\pi^r_\alpha\colon \Psi^r(D^\#_0)\to \Psi^r(D^\#_0)$ (resp.\ $\pi^{r+1}_\alpha\colon \Psi^{r+1}(D^\#_0)\to \Psi^{r+1}(D^\#_0)$) be the projection onto the direct summands corresponding to those $S\in\binom{\Delta}{r}$ (resp.\ those $U\in\binom{\Delta}{r+1}$) not containing $\alpha$. The plan is to use induction on both the horizontal arrows \eqref{keydiagramalpha}. However, we can only do so for elements in $\sum_{\beta\in\Delta\setminus\{\alpha\}}X_\beta^{p^{2^d}N}D^\#_0$, but not for those in $X_\alpha^{p^{2^d}N}D^\#_0$. So for any integer $j\geq 0$ we write $X_\alpha^{p^{j}N}D^\#_\alpha\oplus V_{\alpha,p^{j}N}=D^\#_\alpha$ as a $\kappa$-vector space for some finite dimensional subspace $V_{\alpha,p^{j}N}\subset D^\#_\alpha$ that we fix once and for all. Further, we denote by $\pi_{\leq p^jN}$ (resp.\ by $\pi_{\geq p^j N}$) the projection onto the direct summand $V_{\alpha,p^j N}$ (resp.\ onto $X_\alpha^{p^{j}N}D^\#_\alpha$). Using these we find
\begin{equation}\label{decomposebetapibig}
\sum_{\beta\in\Delta}X_\beta^{p^jN}D^\#_0=X_\alpha^{p^jN}D^\#_0\oplus V_{\alpha,p^jN}\otimes_\kappa\left(\sum_{\beta\in\Delta\setminus\{\alpha\}}X_\beta^{p^j N}D^\#_{\Delta\setminus\{\alpha\},0}\right)
\end{equation}
for all $j\geq 0$. Moreover, the projections $\pi_{\leq p^jN}$ and $\pi_{\geq p^j N}$ all commute with $\id_{D_\alpha^\#}\otimes d^r_{\Delta\setminus\{\alpha\}}$ for all $r\geq 0$. Now pick an element $(x_U)_U=d^r_\Psi((y_S)_S)$ in $\bigoplus_{U\in\binom{\Delta}{r+1}}\sum_{\alpha\in\Delta}X_\alpha^{p^{2^d}N}D_0^\#$. 

\emph{Step 1. We reduce to the case $\pi^{r+1}_\alpha((x_U)_U)$.} Note that both components $\pi_{\geq p^{2^d}N}(\pi^{r+1}_\alpha((x_U)_U))$ and $\pi_{\leq p^{2^d}N}(\pi^{r+1}_\alpha((x_U)_U))$ lie in the image of the differential $\id_{D^\#_\alpha}\otimes d^r_{\Delta\setminus\{\alpha\}}$, so first of all we find $(y^{(1)}_S)_S\in \bigoplus\limits_{\alpha\notin S\in\binom{\Delta}{r}}X_\alpha^{p^{2^d}N}D^\#_0$ with $(\id_{D^\#_\alpha}\otimes d^r_{\Delta\setminus\{\alpha\}})((y^{(1)}_S)_S)=\pi_{\geq p^{2^d}N}(\pi^{r+1}_\alpha((x_U)_U))$. On the other hand, by \eqref{decomposebetapibig} $\pi_{\leq p^{2^d}N}(\pi^{r+1}_\alpha((x_U)_U))$ must be in 
\begin{equation*}
V_{\alpha,p^{2^d}N}\otimes \bigoplus_{U\in\binom{\Delta\setminus\{\alpha\}}{r+1}}\sum_{\beta\in\Delta\setminus\{\alpha\}}X_\beta^{p^{2^d}N}D_{\Delta\setminus\{\alpha\},0}^\#\ .
\end{equation*}
By induction, there exists an element  
\begin{equation*}
(y^{(2)}_S)_S\in V_{\alpha,p^{2^d}N}\otimes \bigoplus_{S\in\binom{\Delta\setminus\{\alpha\}}{r}}\sum_{\beta\in\Delta\setminus\{\alpha\}}X_\beta^{p^{2^{d-1}}N}D_{\Delta\setminus\{\alpha\},0}^\#\subset V_{\alpha,p^{2^d}N}\otimes \Psi^r(D_{\Delta\setminus\{\alpha\},0}^\#)
\end{equation*}
with $(\id_{V_{\alpha,p^{2^d}N}}\otimes d^r_{\Delta\setminus\{\alpha\}})(y^{(2)}_S)_S=\pi_{\leq p^{2^d}N}(\pi^{r+1}_\alpha((x_U)_U))$. Now put $(y'_S)_S:=(y_S)_S-(y^{(1)}_S)_S-(y^{(2)}_S)_S$ and $(x'_U)_U:=d^r_\Psi((y'_S)_S)$. So we have 
\begin{align*}
\pi^{r+1}_\alpha((x_U)_U-(x'_U)_U)=\id_{D^\#_\alpha}\otimes d^r_{\Delta\setminus\{\alpha\}}((y_S)_S-(y'_S)_S)=(\id_{V_{\alpha,p^{2^d}N}}\otimes d^r_{\Delta\setminus\{\alpha\}})((y^{(1)}_S)_S+(y^{(2)}_S)_S)=\\
=\pi_{\geq p^{2^d}N}(\pi^{r+1}_\alpha((x_U)_U))+\pi_{\leq p^{2^d}N}(\pi^{r+1}_\alpha((x_U)_U))=\pi^{r+1}_\alpha((x_U)_U)
\end{align*}
whence $\pi^{r+1}_\alpha((x'_U)_U)=0$. Therefore we compute
\begin{align}
(x'_U)_U=(x'_U)_U-\pi^{r+1}_\alpha((x'_U)_U)=(x_U)_U-\pi^{r+1}_\alpha((x_U)_U)-(\psi_\alpha-1)\otimes\id_{D^\#_{\Delta\setminus\{\alpha\}}}((y^{(1)}_S)_S+(y^{(2)}_S)_S)\notag\\
\in \bigoplus_{\alpha\in U\in\binom{\Delta}{r+1}}{\Bigg(}\sum_{\beta\in\Delta}X_\beta^{p^{2^{d}}N}D_0^\#+(\psi_\alpha-1)(V_{\alpha,p^{2^d}N})\otimes \sum_{\beta\in\Delta\setminus\{\alpha\}}X_\beta^{p^{2^{d-1}}N}D_{\Delta\setminus\{\alpha\},0}^\# + \notag\\
+ (\psi_\alpha-1)(X_\alpha^{p^{2^d}N}D^\#_\alpha)\otimes D^\#_{\Delta\setminus\{\alpha\},0}{\Bigg)}
\subset\bigoplus_{\alpha\in U\in\binom{\Delta}{r+1}}\sum_{\beta\in\Delta}X_\beta^{p^{2^{d-1}}N}D_0^\#\ .\label{x'Usmall}
\end{align}
Moreover, $(x'_U)_U=d^r_\Psi((y'_S)_S)$ still lies in the image of $d^r_\Psi\colon \Psi^r(D^\#_0)\to \Psi^{r+1}(D^\#_0)$.

\emph{Step 2. We reduce to the case $\pi_{\geq p^{2^{d-1}}N}((x'_U)_U)=0$.} By the identity $\pi^{r+1}_\alpha\circ d^r_\Psi\circ \pi^r_\alpha=\pi^{r+1}_\alpha\circ d^r_\Psi$ we obtain $\pi^r_\alpha((y'_S)_S)$ lies in $D^\#_\alpha\otimes \Ker(d^r_{\Delta\setminus\{\alpha\}}\colon \Psi^r(D^\#_{\Delta\setminus\{\alpha\},0})\to \Psi^{r+1}(D^\#_{\Delta\setminus\{\alpha\},0}))$. Hence we may write 
\begin{align*}
\pi^r_\alpha((y'_S)_S)=\sum_i a^{(i)}_\alpha\otimes b^{(i)}_{\overline\alpha}\quad ,a^{(i)}_\alpha\in D^\#_\alpha\ ,\  b^{(i)}_{\overline\alpha}\in  \Ker(d^r_{\Delta\setminus\{\alpha\}}\colon \Psi^r(D^\#_{\Delta\setminus\{\alpha\},0})\to \Psi^{r+1}(D^\#_{\Delta\setminus\{\alpha\},0}))\ ,\\
(y'_S)_S-\pi^r_\alpha((y'_S)_S)= \sum_j c^{(j)}_\alpha\otimes e^{(j)}_{\overline\alpha}\quad ,c^{(j)}_\alpha\in D^\#_\alpha\ ,\  e^{(j)}_{\overline\alpha}\in \bigoplus_{\alpha\in S\in\binom{\Delta}{r}}D^\#_{\Delta\setminus\{\alpha\},0}\cong \Psi^{r-1}(D^\#_{\Delta\setminus\{\alpha\},0})\ ,
\end{align*}
so we compute
\begin{align*}
(x'_U)_U=d^r_\Psi((y'_S)_S)=\pm \sum_i (\psi_\alpha-1)(a^{(i)}_\alpha)\otimes b^{(i)}_{\overline\alpha}\pm \sum_j c^{(j)}_\alpha\otimes d^{r-1}_{\Delta\setminus\{\alpha\}}(e^{(j)}_{\overline\alpha})\ .
\end{align*}
Recall we have $D^\#_\alpha=X_\alpha^{p^{2^{d-1}}N}D^\#_\alpha\oplus V_{\alpha,p^{2^{d-1}}N}$ as a $\kappa$-vector space. Since $p^{2^{d-1}}N\geq N\geq n_0:=\max_\beta(n_0(D_\beta))$, we have 
\begin{align}
X_\alpha^{p^{2^{d-1}}N}D^\#_\alpha\subseteq X_\alpha^{p^{2^{d-1}}N-n_0}D_\alpha^{++}\subseteq (\psi_\alpha-1)(X_\alpha^{p^2N-pn_0}D^{++}_\alpha)\subset\notag\\
\subset (\psi_\alpha-1)(X_\alpha^{p^2N-pn_0}D^\#_\alpha)\subset (\psi_\alpha-1)(X_\alpha^{N}D^\#_\alpha)\ .\label{psicontained}
\end{align}
In particular, we have
\begin{align*}
\pi_{\geq p^{2^{d-1}}N}\left(\sum_i (\psi_\alpha-1)(a^{(i)}_\alpha)\otimes b^{(i)}_{\overline\alpha}\right)=\sum_i\pi_{\geq p^{2^{d-1}}N}((\psi_\alpha-1)(a^{(i)}_\alpha))\otimes b^{(i)}_{\overline\alpha}\\
\in (\psi_\alpha-1)\otimes\id_{D^\#_{\Delta\setminus\{\alpha\}}}\left(X_\alpha^ND_\alpha\otimes \Ker(d^r_{\Delta\setminus\{\alpha\}}\colon \Psi^r(D^\#_{\Delta\setminus\{\alpha\},0})\to \Psi^{r+1}(D^\#_{\Delta\setminus\{\alpha\},0}))\right)\ .
\end{align*}
On the other hand
\begin{align*}
\pi_{\geq p^{2^{d-1}}N}\left(\sum_j c^{(j)}_\alpha\otimes d^{r-1}_{\Delta\setminus\{\alpha\}}(e^{(j)}_{\overline\alpha})\right)=\sum_j\pi_{\geq p^{2^{d-1}}N}(c^{(j)}_\alpha)\otimes d^{r-1}_{\Delta\setminus\{\alpha\}}(e^{(j)}_{\overline\alpha})\\
\in \id_{D^\#_\alpha}\otimes d^{r-1}_{\Delta\setminus\{\alpha\}}\left(X_\alpha^{p^{2^{d-1}}N}D_\alpha\otimes \Psi^{r-1}(D^\#_{\Delta\setminus\{\alpha\},0})\right)\ ,
\end{align*}
so $\pi_{\geq p^{2^{d-1}}N}((x'_U)_U)$ lies in $d^r_\Psi\left(\bigoplus_{S\in\binom{\Delta}{r}}\sum_{\alpha\in\Delta}X_\alpha^{N}D^\#_0\right)$.

\emph{Step 3.} Finally, $\pi_{\leq p^{2^{d-1}}N}((x'_U)_U)=:(x''_U)_U$ lies in $$ V_{\alpha,p^{2^{d-1}}N}\otimes \bigoplus_{\alpha\in U\in\binom{\Delta}{r+1}}\sum_{\beta\in\Delta\setminus\{\alpha\}}X_\beta^{p^{d-1}N}D^\#_{\Delta\setminus\{\alpha\},0}$$
by \eqref{x'Usmall}. We choose a $\kappa$-basis $v_1,\dots,v_l$ in the finite dimensional vectorspace $\mathrm{Im}(\psi_\alpha-1)\cap V_{\alpha,p^{2^{d-1}}N}\subset D^\#_\alpha$ and extend it to a basis $v_1,\dots,v_l,v_{l+1},\dots,v_{l+l'}$ of $V_{\alpha,p^{2^{d-1}}N}$. Writing $\pi_{\leq p^{2^{d-1}}N}((\psi_\alpha-1)(a_\alpha^{(i)}))$ and $\pi_{\leq p^{2^{d-1}}N}(c_\alpha^{(j)})$ in the basis $v_1,\dots,v_{l+l'}$ we find using \eqref{psicontained} that the $v_{l+1},\dots,v_{l+l'}$-components of $\pi_{\leq p^{2^{d-1}}N}((\psi_\alpha-1)(a_\alpha^{(i)}))$ vanish for all $i$ whence the $v_{l+1},\dots,v_{l+l'}$-components of $\pi_{\leq p^{2^{d-1}}N}(\sum_j c_\alpha^{(j)}\otimes d^{r-1}_{\Delta\setminus\{\alpha\}}(e^{(j)}_{\overline\alpha}))$ must lie in $$D^\#_\alpha\otimes \left(\mathrm{Im}(d^{r-1}_{\Delta\setminus\{\alpha\},0})\cap \bigoplus_{S'\in \binom{\Delta\setminus\{\alpha\}}{r}} \sum_{\beta\in\Delta\setminus\{\alpha\}}X_\beta^{p^{2^{d-1}}N}D_{\Delta\setminus\{\alpha\},0}^\#\right)$$ showing that they are in $d^r_\Psi(\bigoplus_{\alpha\in S\in\binom{\Delta}{r}}\sum_{\beta\in\Delta}X_\beta^{N}D_{0}^\#)$ using induction again. Finally, the $v_1,\dots,v_l$-components altogether lie in 
$$(\psi_\alpha-1)(D^\#_\alpha)\otimes \Ker\left(d^{r}_{\Delta\setminus\{\alpha\}}\colon  \bigoplus_{S'\in \binom{\Delta\setminus\{\alpha\}}{r}} \sum_{\beta\in\Delta\setminus\{\alpha\}}X_\beta^{p^{2^{d-1}}N}D_{\Delta\setminus\{\alpha\},0}^\#\to \Psi^{r+1}(D^\#_{\Delta\setminus\{\alpha\},0})\right)$$ since $d^{r-1}_{\Delta\setminus\{\alpha\}}(e^{(j)}_{\overline\alpha})\in\mathrm{Im}(d^{r-1}_{\Delta\setminus\{\alpha\},0})\subseteq \Ker(d^{r}_{\Delta\setminus\{\alpha\}})$. We deduce that this part is even in the image $d^r_\Psi(\bigoplus_{\alpha\notin S\in\binom{\Delta}{r}}\sum_{\beta\in\Delta}X_\beta^{p^{2^{d-1}}N}D_{0}^\#)$.
\end{proof}
The second statement follows using Cor.\ \ref{compactcohom} noting that the image of a dense subset under a quotient map is dense.
\end{proof}

Now the inclusion $D^\#_0\hookrightarrow D^\#\hookrightarrow D$ induces a composite morphism $\Psi^\bullet(D^\#_0)\to \Psi^\bullet(D^\#)\to \Psi^\bullet(D)$ of cochain complexes and therefore a composite morphism
\begin{equation*}
h^r\Psi^\bullet(D^\#_0)\to h^r\Psi^\bullet(D^\#)\overset{\sim}{\to} h^r\Psi^\bullet(D)\to (h^{d-r}\Phi^\bullet(D^*(\TD)))^\vee
\end{equation*}
on the cohomologies for all $0\leq r\leq d$.
\begin{lem}\label{tensorcohomsplit}
We have
\begin{align*}
h^r\Psi^\bullet(D^\#_0)\cong \bigoplus_{S\in\binom{\Delta}{r}}\bigotimes_{\alpha\in\Delta}h^{i_\alpha(S)}\Psi^\bullet(D^\#_\alpha)\text{ , and}\\
h^{d-r}\Phi^\bullet(D^*(\TD))\cong \bigoplus_{S\in\binom{\Delta}{r}}\bigotimes_{\alpha\in\Delta}h^{1-i_\alpha(S)}\Phi^\bullet(D_\alpha^*(1))
\end{align*}
as representations of the group $\Gamma_\Delta$ where $i_\alpha(S)=\begin{cases}1&\text{if }\alpha\in S\\ 0&\text{if }\alpha\notin S\end{cases}$.
\end{lem}
\begin{proof}
Since $\Psi^\bullet(D^\#_0)$ is the tensor product of the cochain complexes $\Psi^\bullet(D^\#_\alpha)$ for $\alpha\in\Delta$, the first statement is simply the K\"unneth formula (Thm.\ 3.6.3 in \cite{W}) for cochain complexes (note that the tensor product is taken over the field $\kappa$). On the other hand, $\Phi^\bullet(D^*(\TD))$ computes the $\HQpD$-cohomology of $V^*(\TD)=\mathbb{V}(D^*(\TD))$ by Prop.\ \ref{HcohomHerr}. By construction, $V^*(\TD)$ is the external tensor product of the Galois representations $V_\alpha^*(1)=\mathbb{V}(D_\alpha^*(1))$ for $\alpha\in\Delta$. Taking injective resolutions $V_\alpha^*(1)\overset{\sim}{\to}I^\bullet(V_\alpha^*(1))$ as discrete $\HQpa$-modules over $\kappa$ for all $\alpha\in\Delta$, each tensor product $I^{n_\Delta}:=\bigotimes_{\alpha\in\Delta,\kappa}I^{n_\alpha}(V_\alpha^*(1))$ ($0\leq n_\alpha\in\mathbb{Z}$ for $\alpha\in\Delta$) is a cohomologically trivial $\HQpD$-module by the Hochschild--Serre spectral sequence 
\begin{equation*}
E_2^{pq}:=H^p(\HQpDa,H^q(\HQpa,I^{n_\Delta}))\Rightarrow H^{p+q}(\HQpD,I^{n_\Delta})
\end{equation*}
Indeed, $I^{n_\Delta}$ is isomorphic to the direct sum of copies of $I^{n_\alpha}(V_\alpha^*(1))$ indexed by a $\kappa$-basis of $\bigotimes_{\beta\in\Delta\setminus\{\alpha\},\kappa}I^{n_\beta}(V_\beta^*(1))$ as a representation of $\HQpa$ for fixed $\alpha\in\Delta$ therefore $H^q(\HQpa,I^{n_\Delta})$ vanishes for $q>0$. On the other hand, we have $H^0(\HQpa,I^{n_\Delta})=H^0(\HQpa,I^{n_\alpha}(V^*_\alpha(1)))\otimes_\kappa \bigotimes_{\beta\in\Delta\setminus\{\alpha\},\kappa}I^{n_\beta}(V_\beta^*(1))$ which is $\HQpDa$-cohomologically trivial by induction on $|\Delta|$. Therefore the complex $\bigotimes_{\alpha\in\Delta,\kappa}I^\bullet(V_\alpha^*(1))$ is a resolution of $V^*(\TD)$ by $\HQpD$-cohomologically trivial terms whence $$\left(\bigotimes_{\alpha\in\Delta,\kappa}I^\bullet(V_\alpha^*(1))\right)^{\HQpD}=\bigotimes_{\alpha\in\Delta,\kappa}I^\bullet(V_\alpha^*(1))^{\HQpa}$$ computes the $\HQpD$-cohomology of $V^*(\TD)$. The result follows again from the K\"unneth formula for cochain complexes.
\end{proof}
Now note that statement of Theorem \ref{iwasawamodp} in the classical case $|\Delta|=1$ (whence the pairing $\{\cdot,\cdot\}$ between $D$ and $D^*(1)$ is perfect) implies the isomorphism $h^i\Psi^\bullet(D^\#_\alpha)\cong (h^{1-i}\Phi^\bullet(D_\alpha^*(1)))^\vee$ for $i=0,1$ of $\kappa\bs \Gamma_\alpha\js$-modules. Further, the Pontryagin dual of the tensor product of discrete $\kappa\bs \Gamma_\alpha\js$-modules for $\alpha\in\Delta$ is the completed tensor product of the Pontryagin duals of the terms. We deduce the isomorphism
\begin{equation*}
\widehat{h^r\Psi^\bullet(D^\#_0)}\overset{\sim}{\to}  (h^{d-r}\Phi^\bullet(D^*(\TD)))^\vee
\end{equation*}
for all $0\leq r\leq d$. Here $\widehat{\cdot}$ stands for the completion $\varprojlim_n(\cdot)/(\sum_{\alpha\in\Delta}T_\alpha^n(\cdot))$ where $T_\alpha$ is the variable in $\kappa\bs\Gamma_\alpha\js\cong\kappa\bs \Zp^\times\js$ under the identification with $\kappa[\Gamma_{\alpha,tors}]\bs T_\alpha\js$ where $\kappa[\Gamma_{\alpha,tors}]$ is the group ring of the finite group of torsion elements in $\Gamma_\alpha$. By the compactness of $h^r\Psi^\bullet(D^\#)$ (Cor.\ \ref{compactcohom}), the map $h^r\Psi^\bullet(D^\#_0)\to h^r\Psi^\bullet(D^\#)$ factors through the completion $\widehat{h^r\Psi^\bullet(D^\#_0)}$. By the discussion above the composite map
\begin{equation*}
\widehat{h^r\Psi^\bullet(D^\#_0)}\to h^r\Psi^\bullet(D^\#)\to (h^{d-r}\Phi^\bullet(D^*(\TD)))^\vee
\end{equation*}
is an isomorphism and the first arrow is onto by Prop.\ \ref{tensordense}. We deduce that the second arrow is also an isomorphism.

\emph{Step 3. The general case.} By Step 1 we obtain a morphism of cohomological $\delta$-functors $h^{i-d}\Psi^\bullet(\cdot)\to h^{2d-i}\Phi^\bullet((\cdot)^*(\TD))^\vee$ ($d\leq i\leq 2d$). Now if $0\to D_1\to D_2\to D_3\to 0$ is a short exact sequence of $p$-power torsion \'etale $(\varphi_\Delta,\Gamma_\Delta)$-modules then we have a morphism 
\begin{equation*}
\xymatrix{
\cdots\ar[r]& (h^{2d-i}\Phi^\bullet(D_1^*(\TD)))^\vee\ar[r] & (h^{2d-i}\Phi^\bullet(D_2^*(\TD)))^\vee\ar[r] & (h^{2d-i}\Phi^\bullet(D_3^*(\TD)))^\vee\ar[r]^-{-1} & \cdots\\
\cdots\ar[r]& h^{i-d}\Psi^\bullet(D_1)\ar[r]\ar[u] & h^{i-d}\Psi^\bullet(D_2)\ar[r]\ar[u] & h^{i-d}\Psi^\bullet(D_3)\ar[r]^-{+1}\ar[u] & \cdots
}
\end{equation*}
between the long exact sequences. In particular, if the statement is true for $D_1$ and $D_3$ then it also follows for $D_2$ by the $5$-lemma. Therefore we are reduced to the case when $\mathbb{V}(D)$ is irreducible as a representation of $\GQpD$. By possibly enlarging the coefficient field $\kappa$ we are done by Step 2.
\end{proof}

\subsection{Cohomology for $p$-adic representations}\label{secherrpadic}

The goal in this section is to show that the Herr complex computes the continuous group cohomology of objects in $\operatorname{Rep}_{\Zp}(\GQpD)$ and $\operatorname{Rep}_{\Qp}(\GQpD)$. In the classical $|\Delta|=1$ case the proof of this fact is not properly explained in \cite{H1}. There is an intrinsic proof in Colmez's Tsinghua notes \cite{Tsing}. Our proof is more inspired by the more conceptual proof of Schneider and Venjakob \cite{SchVen} for the Iwasawa cohomology in the Lubin--Tate case.

Let $T$ be an object in $\operatorname{Rep}_{\Zp}(\GQpD)$. As usual we define the cohomology groups $H^i(\GQpD,T)$ using continuous cochains. 

\begin{lem}\label{projlim10}
We have $H^i(\GQpD,T)\cong \varprojlim_n H^i(\GQpD,T/p^nT)$ and $\varprojlim_n^1 H^i(\GQpD,T/p^nT)=0$.
\end{lem}
\begin{proof}
By definition, we have an isomorphism $C^\bullet(\GQpD,T)\cong \varprojlim_n C^\bullet(\GQpD,T/p^nT)$ on the level of continuous cochains. Since the transition maps are surjective, the first hypercohomology spectral sequence $E_2^{pq}:=h^p\varprojlim_n^qC^\bullet(\GQpD,T/p^nT)\Rightarrow R^{p+q}(\varprojlim_n(\cdot)^{\GQpD})((T/p^nT)_n)$ degenerates at $E_2$. Therefore the hypercohomology groups are simply $H^i(\GQpD,T)$. On the other hand,  the cohomology groups $H^i(\GQpD,T/p^nT)$ are finite, so the projective system $(H^i(\GQpD,T/p^nT))_n$ satisfies the Mittag--Leffler property for any fixed $i\geq 0$ yielding the second statement of the Lemma. This shows that the second hypercohomology spectral sequence $E_2^{pq}:=\varprojlim^p_n H^q(\GQpD,T/p^nT)\Rightarrow H^{p+q}(\GQpD,T)$ also degenerates at $E_2$ showing the first statement.
\end{proof}

In particular, we have 
\begin{align*}
H^i_{Iw}(\GQpD,T)\cong \varprojlim_{\HQpD\leq H\leq_o \GQpD}H^i(H,T)\cong \varprojlim_{\HQpD\leq H\leq_o \GQpD}\varprojlim_n H^i(H,T/p^nT)\cong\\
\cong \varprojlim_n\varprojlim_{\HQpD\leq H\leq_o \GQpD} H^i(H,T/p^nT)\cong \varprojlim_n H^i_{Iw}(\GQpD,T/p^nT)\ .
\end{align*}

Moreover, using again the finiteness of $H^i(H,T/p^nT)$ and noting that there are countably many open subgroups of $\Gamma_\Delta=\GQpD/\HQpD$, we deduce $ \varprojlim_n^1 H^i_{Iw}(\GQpD,T/p^nT)=0$.

\begin{thm}\label{herrcomplex+iwasawa}
Let $T$ (resp.\ $V$) be an object in $\operatorname{Rep}_{\Zp}(\GQpD)$ (resp.\ in $\operatorname{Rep}_{\Qp}(\GQpD)$). We have isomorphisms 
\begin{eqnarray*}
H^i(\GQpD,T)\cong h^i\Phi\Gamma_\Delta^\bullet(\mathbb{D}(T))\ ;\qquad H^i(\GQpD,V)\cong h^i\Phi\Gamma_\Delta^\bullet(\mathbb{D}(V))\ ;\\
H^i_{Iw}(\GQpD,T)\cong h^{i-d}\Psi^\bullet(\mathbb{D}(T))\ ;\qquad H^i_{Iw}(\GQpD,V)\cong h^{i-d}\Psi^\bullet(\mathbb{D}(V))
\end{eqnarray*}
natural in $T$ (resp.\ in $V$) for all $i\geq 0$.
\end{thm}
\begin{proof}
Consider the projective system $(\Phi\Gamma_\Delta^\bullet(\mathbb{D}(T/p^nT)))_n$ of cochain complexes of abelian groups. Since the transition maps are surjective, the first hypercohomology spectral sequence degenerates at $E_2$. Therefore the second hypercohomology spectral sequence becomes
\begin{equation*}
{\varprojlim_n}^i h^j\Phi\Gamma_\Delta^\bullet(\mathbb{D}(T/p^nT))\Rightarrow h^{i+j} \Phi\Gamma_\Delta^\bullet(\mathbb{D}(T))\ .
\end{equation*}
By Thm.\ \ref{herrcomplexmodp} and Lemma \ref{projlim10} this spectral sequence degenerates, too, and computes the continuous cohomology $H^{i+j}(\GQpD,T)$.

The proof of the statement on the Iwasawa cohomology groups is entirely analogous using Thm.\ \ref{iwasawamodp} instead. The result on $V$ follows by inverting $p$.
\end{proof}

\begin{cor}
The functors $H^i_{Iw}(\GQpD,\cdot)$ ($d\leq i\leq 2d$) form a cohomological $\delta$-functor on $\operatorname{Rep}_{\Zp}(\GQpD)$.
\end{cor}
\begin{proof}
For a short exact sequence $0\to T_1\to T_2\to T_3\to 0$ we have a short exact sequence $0\to \Psi^\bullet(\mathbb{D}(T_1))\to \Psi^\bullet(\mathbb{D}(T_2))\to \Psi^\bullet(\mathbb{D}(T_3))\to 0$ of cochain complexes yielding a long exact sequence of cohomology groups.
\end{proof}

Note that $\Gamma_\Delta$ acts $\Zp$-linearly on $H^i_{Iw}(\GQpD,T)$ by construction. The action is continuous and the Iwasawa cohomology groups are compact (Cor.\ \ref{compactcohom}) therefore this extends to an action of the Iwasawa algebra $\Zp\bs \Gamma_\Delta\js$.

\begin{cor}\label{iwasawafingen}
Let $T$ be an object in $\operatorname{Rep}_{\Zp}(\GQpD)$. The Iwasawa cohomology groups $H^i_{Iw}(\GQpD,T)$ are finitely generated $\Zp\bs \Gamma_\Delta\js $-modules.
\end{cor}
\begin{proof}
At first assume that $T$ is an object in $\operatorname{Rep}_{\Zp-tors}(\GQpD)$. By Lemma \ref{measure} we have an identification $H^i_{Iw}(\GQpD,T)\cong H^i(\GQpD,\Zp\bs \Gamma_\Delta\js\otimes_{\Zp}T)$. There is an action of the group $\GQpD\times \Gamma_\Delta$ on $\Zp\bs \Gamma_\Delta\js\otimes_{\Zp}T$ given by the formula $(g,\gamma)(\lambda\otimes x):=(\overline{g}\lambda\gamma^{-1}\otimes (gx)$ on elementary tensors ($g\in\GQpD$ with image $\overline{g}\in\Gamma_\Delta$; $\gamma\in\Gamma_\Delta$; $\lambda\in\Zp\bs \Gamma_\Delta\js$; and $x\in T$). The action of $\Gamma_\Delta$ extends to the Iwasawa algebra $\Zp\bs \Gamma_\Delta\js$ making $\Zp\bs \Gamma_\Delta\js\otimes_{\Zp}T$ into a left module over $\Zp\bs \Gamma_\Delta\js$ equipped with a linear action of $\GQpD$. Thus the cohomology groups $ H^i(\GQpD,\Zp\bs \Gamma_\Delta\js\otimes_{\Zp}T)$ are left modules over $\Zp\bs \Gamma_\Delta\js$. Pick a topological generator $\gamma_\alpha$ of $\Gamma_\alpha$ for some fixed $\alpha\in\Delta$. Since $\Zp\bs \Gamma_{\Delta\setminus\{\alpha\}}\js$ is $\Zp$-flat, we have a short exact sequence
\begin{equation*}
0\to \Zp\bs \Gamma_\Delta\js\otimes_{\Zp}T\overset{(\gamma_\alpha-1)\cdot}{\to} \Zp\bs \Gamma_\Delta\js\otimes_{\Zp}T\to \Zp\bs \Gamma_{\Delta\setminus\{\alpha\}}\js\otimes_{\Zp}T\to 0
\end{equation*}
of $\GQpD$-representations. Therefore we obtain a long exact sequence
\begin{align*}
\cdots\to H^i(\GQpD, \Zp\bs \Gamma_\Delta\js\otimes_{\Zp}T) \overset{(\gamma_\alpha-1)\cdot}{\to} H^i(\GQpD,\Zp\bs \Gamma_\Delta\js\otimes_{\Zp}T)\to\\
\to H^i(\GQpD,\Zp\bs \Gamma_{\Delta\setminus\{\alpha\}}\js\otimes_{\Zp}T)\to H^{i+1}(\GQpD, \Zp\bs \Gamma_\Delta\js\otimes_{\Zp}T)\to\cdots
\end{align*}
of compact $\Zp\bs \Gamma_\Delta\js$-modules. By the topological Nakayama Lemma \cite{BH} $H^i(\GQpD, \Zp\bs \Gamma_\Delta\js\otimes_{\Zp}T)$ is finitely generated over $\Zp\bs \Gamma_\Delta\js$ if and only if the cokernel of the multiplication by $(\gamma_\alpha-1)$ is finitely generated over $\Zp\bs \Gamma_{\Delta\setminus\{\alpha\}}\js$ which is true assuming that $H^i(\GQpD,\Zp\bs \Gamma_{\Delta\setminus\{\alpha\}}\js\otimes_{\Zp}T)$ is finitely generated over $\Zp\bs \Gamma_{\Delta\setminus\{\alpha\}}\js$. The statement follows by induction on $|\Delta|$ noting that  $H^i(\GQpD,\Zp\bs \Gamma_{\emptyset}\js\otimes_{\Zp}T)=H^i(\GQpD,T)$ is finitely generated over $\Zp\bs \Gamma_{\emptyset}\js=\Zp$.

Using the above and the long exact sequence of Iwasawa cohomology applied to the short exact sequence $0\to T_{tors}\to T\to T/T_{tors}\to 0$ we may assume without loss of generality that $T$ is free over $\Zp$. Then we use the long exact sequence associated to $0\to T\overset{p\cdot}{\to} T\to T/pT\to 0$ and the topological Nakayama Lemma as above to deduce the statement for a general object $T$ in $\operatorname{Rep}_{\Zp}(\GQpD)$.
\end{proof}

\subsection{The Euler--Poincar\'e characteristic formula}

Recall that whenever $A$ (resp.\ $T$, resp.\ $V$) is a finite $p$-power torsion abelian group (resp.\ finitely generated $\Zp$-module, resp.\ finite dimensional $\Qp$-vectorspace)  with a continuous action of $G_{\Qp}$ then the Euler-Poincar\'e characteristic of $A$ (resp.\ of $T$,  resp.\ of $V$) is defined as $\chi_{G_{\Qp}}(A):=\prod_{i=0}^{2}|H^i(G_{\Qp},A)|^{(-1)^i}$ (resp.\ as $\chi_{G_{\Qp}}(T):=\sum_{i=0}^2(-1)^i\rk_{\Zp}H^i(G_{\Qp},T)$, resp.\ as $\chi_{G_{\Qp}}(V):=\sum_{i=0}^2(-1)^i\dim_{\Qp}H^i(G_{\Qp},V)$). The classical Euler-Poincar\'e characteristic formula states that $$\chi_{G_{\Qp}}(A)=|A|\ ,\quad \chi_{G_{\Qp}}(T)=\rk_{\Zp}T\ ,\ \text{and}\quad\chi_{G_{\Qp}}(V)=\dim_{\Qp}V\ .$$

We define the Euler--Poincar\'e characteristic of representations of $\GQpD$ similarly:  whenever $A$ (resp.\ $T$, resp.\ $V$) is a finite $p$-power torsion abelian group (resp.\ finitely generated $\Zp$-module, resp.\ finite dimensional $\Qp$-vectorspace)  with a continuous action of $\GQpD$ then the Euler-Poincar\'e characteristic of $A$ (resp.\ of $T$,  resp.\ of $V$) is defined as \begin{eqnarray*}
\chi_{\GQpD}(A)&:=&\prod_{i=0}^{2|\Delta|}|H^i(\GQpD,A)|^{(-1)^i}\ ;\\
\chi_{\GQpD}(T)&:=&\sum_{i=0}^{2|\Delta|}(-1)^i\rk_{\Zp}H^i(\GQpD,T)\ ;\\
\chi_{\GQpD}(V)&:=&\sum_{i=0}^{2|\Delta|}(-1)^i\dim_{\Qp}H^i(\GQpD,V)\ .
\end{eqnarray*}
The analogous Euler--Poincar\'e characteristic formula follows from the classical $|\Delta|=1$ case by induction on $|\Delta|$ using the Hochschild--Serre spectral sequence. However, we present here a different proof using multivariable $(\varphi_\Delta,\Gamma_\Delta)$-modules as the statements proven along these lines might be of independent interest.

For an object $D$ in $\mathcal{D}^{et}(\varphi_\Delta,\Gamma_\Delta,\OED)$ or in $\mathcal{D}^{et}(\varphi_\Delta,\Gamma_\Delta,\mathcal{E}_{\Delta})$ and a subset $S\subseteq \Delta$ we define the cochain complex (slightly different from the one introduced in \eqref{psicomplexdef})
\begin{align*}
\Psi^\bullet_{0,S}(D)\colon 0\to D\to \bigoplus_{\alpha\in S}D\to \dots\to \bigoplus_{\{\alpha_1,\dots,\alpha_r\}\in \binom{ S}{r}}D\to\dots \to D\to 0
\end{align*}
where for all $0\leq r\leq | S|-1$ the map $d_{\alpha_1,\dots,\alpha_r}^{\beta_1,\dots,\beta_{r+1}}\colon D\to D$ from the component in the $r$th term corresponding to $\{\alpha_1,\dots,\alpha_r\}\subseteq  S$ to the component corresponding to the $(r+1)$-tuple $\{\beta_1,\dots,\beta_{r+1}\}\subseteq S$ is given by
\begin{equation*}
d_{\alpha_1,\dots,\alpha_r}^{\beta_1,\dots,\beta_{r+1}}=\begin{cases}0&\text{if }\{\alpha_1,\dots,\alpha_r\}\not\subseteq\{\beta_1,\dots,\beta_{r+1}\}\\ (-1)^{\eta}\psi_\beta&\text{if }\{\beta_1,\dots,\beta_{r+1}\}=\{\alpha_1,\dots,\alpha_r\}\cup\{\beta\}\ ,\end{cases}
\end{equation*}
where $\eta=\eta(\alpha_1,\dots,\alpha_r,\beta)$ is the number of elements in the set $ S\setminus\{\alpha_1,\dots,\alpha_r\}$ smaller than $\beta$. We put $\Psi^\bullet_0(D):=\Psi^\bullet_{0,\Delta}(D)$.

\begin{lem}\label{h0psiexact}
The complex $\Psi^\bullet_{0,S}(D)$ is acyclic in nonzero degrees. In particular, the functor $D\mapsto h^0\Psi^\bullet_{0,S}(D)=\bigcap_{\beta\in S}\Ker(\psi_\beta)$ is exact.
\end{lem}
\begin{proof}
We proceed by induction on $|S|$. If $|S|=1$ then this is the surjectivity of $\psi_\beta$ ($S=\{\beta\}$). Moreover, if $|S|>1$ then for any fixed $\beta\in S$ and $S':=S\setminus\{\beta\}$ the complex $\Psi^\bullet_{0,S}(D)$ is the total complex of the double complex
\begin{align*}
\xymatrix{
 0\ar[r] & D\ar[r] & \bigoplus\limits_{\alpha\in S'}D\ar[r] & \dots\ar[r] & \bigoplus\limits_{\{\alpha_1,\dots,\alpha_r\}\in \binom{ S'}{r}}D\ar[r] &\dots \ar[r] & D\ar[r] & 0\\
 0\ar[r] & D\ar[r]\ar[u]_{\psi_\beta} & \bigoplus\limits_{\alpha\in S'}D\ar[r]\ar[u]_{\bigoplus\psi_\beta} & \dots\ar[r] & \bigoplus\limits_{\{\alpha_1,\dots,\alpha_r\}\in \binom{ S'}{r}}D\ar[r]\ar[u]_{\bigoplus\psi_\beta} &\dots \ar[r] & D\ar[r]\ar[u]_{\psi_\beta} & 0
}
\end{align*}
where the rows of the above complex are isomorphic to $\Psi^\bullet_{0,S'}(D)$. By induction $\Psi^\bullet_{0,S'}(D)$ is acyclic in nonzero degrees, so the above total complex is quasi-isomorphic to 
\begin{align*}
0\to h^0\Psi^\bullet_{0,S'}(D)\overset{\psi_\beta}{\to}  h^0\Psi^\bullet_{0,S'}(D)\to 0\ .
\end{align*}
Finally, the map $\psi_\beta$ is surjective on $ h^0\Psi^\bullet_{0,S'}(D)$ as it has a right inverse $\varphi_\beta$. Indeed, $\varphi_\beta$ commutes with $\psi_\alpha$ for any $\alpha\in S'$ therefore maps $h^0\Psi^\bullet_{0,S'}(D)=\bigcap_{\alpha\in S'}\Ker(\psi_\alpha)$ into itself.
\end{proof}

The group $\Gamma_\Delta$ is isomorphic to the direct product $C_\Delta\times \Gamma_\Delta^*$ where $C_\Delta$ is a finite group and $\Gamma_\Delta^*\cong \prod_{\alpha\in\Delta}\Gamma_\alpha^*\cong\Zp^\Delta$. In particular, the Iwasawa algebra $E_\Delta^+(\Gamma_\Delta^*):=\Fp\bs \Gamma_\Delta^*\js$ of $\Gamma_\Delta^*$ over $\Fp$ is isomorphic to the power series ring $\Fp\bs Y_\alpha\mid \alpha\in\Delta\js$ where $1+Y_\alpha$ corresponds to a topological generator of the group $\Gamma_\alpha^*\cong \Zp$ for all $\alpha\in\Delta$. So we may form the ring $E_\Delta(\Gamma_\Delta^*):=E_\Delta^+(\Gamma_\Delta^*)[Y_\alpha^{-1}\mid \alpha\in\Delta]$. Finally, we put $E_\Delta(\Gamma_\Delta):=E_\Delta(\Gamma_\Delta^*)\otimes_{E_\Delta^+(\Gamma_\Delta^*)}\Fp\bs \Gamma_\Delta\js$.

\begin{pro}\label{freegammamod}
Let $D$ be an object in $\mathcal{D}^{et}(\varphi_\Delta,\Gamma_\Delta,E_\Delta)$. Then $h^0\Psi^\bullet_0(D)$ is a free $E_\Delta(\Gamma_\Delta)$-module of rank $\rk_{E_\Delta}D$.
\end{pro}
\begin{proof}
By passing to a finite extension of $\Fp$ and using Lemma \ref{h0psiexact} we are reduced to the case when $V:=\mathbb{V}_\Delta(D)$ is absolutely irreducible as a representation of $\GQpD$. In this case $V$ is an outer tensor product of representations $V_\alpha$ for $\alpha\in \Delta$. Therefore $D$ is the completed tensor product of $D_\alpha:=\mathbb{D}(V_\alpha)$ ($\alpha\in\Delta$). In particular, $h^0\Psi^\bullet_0(D)$ is the completed tensor product of $h^0\Psi^\bullet_0(D_\alpha)$ ($\alpha\in\Delta$). The result follows from the case $|\Delta|=1$ which is proven in Cor.\ VI.1.3 in \cite{Mira}.
\end{proof}

\begin{pro}\label{higherIwasawatorsion}
Let $V$ be a continuous $\Fp$-representation of $\GQpD$. The Iwasawa cohomology groups $H^i_{Iw}(\GQpD,V)$ are torsion $E_\Delta^+ (\Gamma_\Delta^*)$-modules for all $i>d=|\Delta|$.
\end{pro}
\begin{proof}
By passing to a finite extension of $\Fp$ and using the long exact sequence of the $\Psi^\bullet$ complex of $D:=\mathbb{D}(V)$ together with Thm.\ \ref{iwasawamodp} we are reduced to the case when $V$ is absolutely irreducible as a representation of $\GQpD$. In this case $V$ is an outer tensor product of representations $V_\alpha$ for $\alpha\in \Delta$. By Prop.\ \ref{Dhashquasiiso} and \ref{tensordense} the cohomology groups $h^{i-d}\Psi^\bullet(D^\#_0)$ are dense in $h^{i-d}\Psi^\bullet(D^\#)\cong H^i_{Iw}(\GQpD,V)$. The result follows from the description (Lemma \ref{tensorcohomsplit}) of $h^{i-d}\Psi^\bullet(D^\#_0)$ noting that $h^1\Psi^\bullet(D^\#_\alpha)$ is finite dimensional over $\Fp$ (Cor.\ I.7.4 in \cite{CC2}) hence killed by a nonzero element in $\Fp\bs \Gamma_\alpha^*\js$ for all $\alpha\in\Delta$. Indeed, we find a nonzero element in $\bigotimes_{\alpha\in\Delta}\Fp\bs \Gamma_\alpha^*\js\subset \Fp\bs \Gamma_\Delta^*\js$ annihilating $h^{i-d}\Psi^\bullet(D^\#_0)$ for any $i>d$ and by continuity this element also kills $h^{i-d}\Psi^\bullet(D^\#)$.
\end{proof}

\begin{pro}\label{higherIwasawatorsion2}
Let $V$ be a continuous $\Fp$-representation of $\GQpD$. The cohomology group $h^0\Psi^\bullet(\mathbb{D}(V)^{C_\Delta})$ has rank $\dim_{\Fp}V$ over $E_\Delta^+(\Gamma_\Delta^*)=\Fp\bs \Gamma_\Delta^*\js$.
\end{pro}
\begin{proof}
This follows by a similar argument as in the proof of Prop.\ \ref{higherIwasawatorsion}. However, for future applications we include a different proof here resembling the classical $|\Delta|=1$ case. Put $D:=\mathbb{D}(V)$ and consider the map
\begin{equation*}
\prod_{\alpha\in\Delta}(\varphi_\alpha-\id)\colon h^0\Psi^\bullet(D^{C_\Delta})\to  h^0\Psi_0^\bullet(D^{C_\Delta})\ .
\end{equation*}
By Prop.\ \ref{freegammamod} the right hand side is a free module over $E_\Delta(\Gamma_\Delta^*)\cong E_\Delta(\Gamma_\Delta)^{C_\Delta}$ of rank $\dim_{\Fp}V$. We show that both the kernel and cokernel of  $\prod_{\alpha\in\Delta}(\varphi_\alpha-\id)$ are torsion modules over $E_\Delta^+(\Gamma_\Delta^*)$ so it becomes an isomorphism after tensoring with the field of fractions $\operatorname{Frac}(E_\Delta^+(\Gamma_\Delta^*))$ of $E_\Delta^+(\Gamma_\Delta^*)$. However, we have $\operatorname{Frac}(E_\Delta^+(\Gamma_\Delta^*))\otimes_{E_\Delta^+(\Gamma_\Delta^*)}E_\Delta(\Gamma_\Delta^*)\cong \operatorname{Frac}(E_\Delta^+(\Gamma_\Delta^*))$ so this implies the statement. 

As in the classical case \cite{Mira} we define $D^{++}=\{y\in D\mid \lim_{k\to\infty}(\prod_{\alpha\in\Delta}\varphi_\alpha)^k(y)=0\} $ where the limit is considered in the $X_\Delta$-adic topology.

\begin{lem}\label{gammamultiplyintoD++}
For any element $x\in D$ and any choice of topological generators $\gamma_\alpha\in\Gamma_\alpha^*$ there exists an integer $k>0$ such that $\prod_{\alpha\in\Delta}(\gamma_\alpha-1)^k\cdot x$ lies in $D^{++}$.
\end{lem}
\begin{proof}
This is similar to the proof of the classical case $|\Delta|=1$ (see section III.4 in \cite{Mira}): At first note that $x$ lies in $X_\Delta^{-n}D^{++}$ for some $n\geq 0$ by Prop.\ 2.5 in \cite{MultVarGal}. Moreover, the subspace $M:=X_\Delta^{-n-1}D^{++}$ is invariant under the action of $\Gamma_\Delta$. Moreover, for $k\geq p^r\geq n+1$ the element $(\gamma_\alpha-1)^kX_\alpha$ is divisible by $X_\alpha^{p^r}$ in $E_\Delta^+$ by Lemme III.4.1 in \cite{Mira}. Finally, we compute
\begin{align*}
\prod_{\alpha\in\Delta}(\gamma_\alpha-1)^k\cdot x\in \prod_{\alpha\in\Delta}(\gamma_\alpha-1)^k\cdot X_\Delta M= \prod_{\alpha\in\Delta}\left((\gamma_\alpha-1)^kX_\alpha\right)\cdot M\subseteq X_\Delta^{n+1}M\subseteq D^{++}\ .
\end{align*}
\end{proof}
Since the map $\prod_{\alpha\in\Delta}(\varphi_\alpha-\id)$ is formally invertible on $D^{++}$ and $ h^0\Psi^\bullet(D^{C_\Delta})$ is finitely generated over $E_\Delta^+(\Gamma_\Delta^*)$ (Cor.\ \ref{iwasawafingen}), the statement follows from Lemma \ref{gammamultiplyintoD++}.
\end{proof}
\begin{rem}
The above proof also shows that $E_\Delta(\Gamma_\Delta^*)\otimes_{E_\Delta^+(\Gamma_\Delta^*)}h^0\Psi^\bullet(D^{C_\Delta})$ (resp.\ $E_\Delta(\Gamma_\Delta)\otimes_{E_\Delta^+(\Gamma_\Delta)}h^0\Psi^\bullet(D)$) is a free module of rank $\dim_{\Fp}V$ over $E_\Delta(\Gamma_\Delta^*)$ (resp.\ over $E_\Delta(\Gamma_\Delta)$).
\end{rem}

\begin{lem}\label{kerpsigammaacyclic}
Let $D$ be an object in $\mathcal{D}^{et}(\varphi_\Delta,\Gamma_\Delta,E_\Delta)$ and $S\subseteq \Delta$ be any non-empty subset. Then the complex $\Gamma^\bullet(\Ker(\prod_{\alpha\in S}\psi_\alpha\colon D^{C_\Delta}\to D^{C_\Delta}))$ is acyclic.
\end{lem}
\begin{proof}
Put $D^{C_\Delta,\psi_S=0}:=\Ker(\prod_{\alpha\in S}\psi_\alpha\colon D^{C_\Delta}\to D^{C_\Delta})$ for simplicity. For any $\alpha\in S$ we have a short exact sequence
\begin{equation*}
0\longrightarrow D^{C_\Delta,\psi_\alpha=0} \longrightarrow D^{C_\Delta,\psi_S=0} \overset{\psi_\alpha}{\longrightarrow} D^{C_\Delta,\psi_{S\setminus\{\alpha\}}=0}\longrightarrow 0
\end{equation*}
having a splitting $\varphi_\alpha\colon D^{C_\Delta,\psi_{S\setminus\{\alpha\}}=0}\hookrightarrow D^{C_\Delta,\psi_S=0}$. So we are reduced to the case $S=\{\alpha\}$ by induction. However, the map $\gamma_\alpha-1$ is bijective on $D^{C_\Delta,\psi_\alpha=0}$ by \ref{freegammamod} since we have an embedding
\begin{equation*}
\prod_{\beta\in\Delta\setminus\{\alpha\}}\left(\sum_{\gamma_\beta\in C_\beta}\gamma_\beta(1+X_\beta)\varphi_\beta\right)\colon D^{C_\Delta,\psi_\alpha=0}\hookrightarrow \bigcap_{\beta\in \Delta}\Ker(\psi_\beta)^{C_\Delta}
\end{equation*}
with left-inverse $\prod_{\beta\in\Delta\setminus\{\alpha\}}\psi_\beta\circ(1+X_\beta)^{-1}$ making $D^{C_\Delta,\psi_\alpha=0}$ a direct summand in $\bigcap\limits_{\beta\in \Delta}\Ker(\psi_\beta)^{C_\Delta}$ as a $\Zp\bs\gamma_\alpha-1\js$-module.
\end{proof}

For an \'etale $(\varphi_\Delta,\Gamma_\Delta)$-module over any of the rings $E_\Delta$, $\OED$, or $\mathcal{E}_\Delta$, we denote by $\Psi\Gamma^\bullet(D)$ the total complex of the double complex $\Gamma^\bullet(\Psi^\bullet(D)^{C_\Delta})$.

\begin{thm}\label{psigamma=phigamma}
Let $D$ be an object in $\mathcal{D}^{et}(\varphi_\Delta,\Gamma_\Delta,E_\Delta)$. Then the complex $\Psi\Gamma^\bullet(D)$ is quasi-isomorphic to $\Phi\Gamma^\bullet(D)$. In particular, both compute the Galois cohomology groups $H^\bullet(\GQpD,\mathbb{V}(D))$.
\end{thm}
\begin{proof}
Consider the morphism
\begin{align*}
\psi^\bullet\colon \Phi^\bullet(D)^{C_\Delta}\to \Psi^\bullet(D)^{C_\Delta}
\end{align*}
of cochain complexes that is given by $(-1)^{\varepsilon(S)}\prod_{\alpha\in S}\psi_\alpha$ on the copy of $D$ corresponding to a subset $S\subseteq \Delta$ with $|S|=r$ in $\Phi^r(D)^{C_\Delta}$ mapping onto the copy of $D$ corresponding to $S$ in $\Psi^r(D)^{C_\Delta}$. Here $\varepsilon(S)$ is the sum of the indices of elements of $S$ when the set $\Delta$ is index by the numbers $1,\dots,|\Delta|$ (ie.\ after choosing a total ordering of $\Delta$ on which both cochain complexes depend). This is surjective in each degree therefore by the long exact sequence corresponding to the short exact sequence $0\to \Ker(\psi^\bullet)\to \Phi^\bullet(D)^{C_\Delta}\to \Psi^\bullet(D)^{C_\Delta}\to 0$ we are reduced to showing that the total complex of the double complex $\Gamma^\bullet(\Ker(\psi^\bullet))$ is acyclic. This follows using the (second) spectral sequence of the double complex since the columns of the double complex are acyclic by Lemma \ref{kerpsigammaacyclic}.
\end{proof}

\begin{lem}\label{gammahomol}
For any finitely generated $E_\Delta^+(\Gamma_\Delta^*)$-module $M$ we have $\sum_{j=0}^d(-1)^j\dim_{\Fp} h^j\Gamma^\bullet(M)=\rk_{E_\Delta^+(\Gamma_\Delta^*)}M$.
\end{lem}
\begin{proof}
Whenever $M$ is a free module, the statement is clear as $\Gamma^\bullet(M)$ is acyclic in nonzero degrees in this case. The general case follows from the long exact sequence of $\Gamma^\bullet$ using a free resolution of $M$.
\end{proof}

\begin{thm}\label{torsionEulerPoincare}
Let $A$ be a finite $p$-primary abelian group together with a continuous action of $\GQpD$. Then we have $\chi_{G_{\Qp}}(A)=|A|$.
\end{thm}
\begin{proof}
By the long exact sequence of group cohomology we may assume without loss of generality that $pA=0$. Put $D:=\mathbb{D}(A)$. By Thm.\ \ref{psigamma=phigamma} we have $$\chi_{G_{\Qp}}(A)=p^{\sum_{i=0}^{2|\Delta|}(-1)^i\dim_{\Fp}h^i\Psi\Gamma^\bullet(D)}\ .$$ By the first $E_1$ spectral sequence of the double complex defining $\Psi\Gamma^\bullet(D)$ we compute
\begin{align*}
\sum_{i=0}^{2|\Delta|}(-1)^i\dim_{\Fp}h^i\Psi\Gamma^\bullet(D)=\sum_{i=0}^d\sum_{j=0}^d(-1)^{i+j}\dim_{\Fp}h^j\Gamma^\bullet(h^i\Psi^\bullet(D)^{C_\Delta})=\\
=\sum_{i=0}^d(-1)^i\rk_{E_\Delta^+(\Gamma_\Delta^*)}h^i\Psi^\bullet(D)^{C_\Delta}=\rk_{E_\Delta^+(\Gamma_\Delta^*)}h^0\Psi^\bullet(D)^{C_\Delta}=\dim_{\Fp}A
\end{align*}
using Prop.\ \ref{higherIwasawatorsion} and \ref{higherIwasawatorsion2}, and Lemma \ref{gammahomol}.
\end{proof}

\begin{lem}\label{pntor}
Let $M$ be a finitely generated $\Zp$-module of (generic) rank $r$. Then for any $n\geq 1$ we have $p^{nr}=\frac{|M\otimes_{\Zp}(\Zp/p^n)|}{|\Tor_1^{\Zp}(M,\Zp/p^n)|}$.
\end{lem}
\begin{proof}
This is classical but for the convenience of the reader we include a proof. By the theorem of elementary divisors we may assume $M\cong \Zp/p^k$ or $M\cong\Zp$ is cyclic. The lemma follows directly using a projective resolution of $M$: we have $M\otimes_{\Zp}(\Zp/p^n)\cong M/p^nM$ and $\Tor_1^{\Zp}(M,\Zp/p^n)\cong M[p^n]$.
\end{proof}

\begin{cor}\label{padiceulerpoincare}
Let $T$ (resp.\ $V$) be a finitely generated $\Zp$-module (resp.\ finite dimensional $\Qp$-vectorspace)  with a continuous action of $\GQpD$. Then we have
\begin{align*}
\chi_{\GQpD}(T)=\rk_{\Zp}T\ ;\quad \chi_{\GQpD}(V)=\dim_{\Qp}V
\end{align*}
for the Euler-Poincar\'e characteristic of $T$ (resp.\ of $V$).
\end{cor}
\begin{proof}
The second statement follows from the first one inverting $p$ (which is exact). For the first statement we may assume without loss of generality that $T$ is free over $\Zp$ using the long exact sequence of $\GQpD$-cohomology. In this case $\Phi\Gamma^\bullet(\mathbb{D}(T/pT))$ is the derived tensor product of $\Phi\Gamma^\bullet(\mathbb{D}(T))$ with $\Fp$ over $\Zp$. The statement is a combination of Thm.\ \ref{herrcomplexmodp}, Thm.\ \ref{torsionEulerPoincare}, and Lemma \ref{pntor} with the spectral sequence
\begin{equation*}
E_2^{ij}=\Tor^{\Zp}_{-i}(h^j\Phi\Gamma^\bullet(\mathbb{D}(T)),\Fp)\Rightarrow h^{i+j}\Phi\Gamma^\bullet(\mathbb{D}(T/pT))\ .
\end{equation*}
\end{proof}

\section{Overconvergence}\label{secoverconverge}

The goal in this section is to construct a multivariable analogue $\OED^\dagger$ of the ring $\mathcal{O}_{\mathcal{E}}^\dagger$ and prove the overconvergence of multivariable $(\varphi,\Gamma)$-modules (see \cite{CC} for the classical, $1$-variable case). We show, moreover, that the category of \'etale $(\varphi_\Delta,\Gamma_\Delta)$-modules over $\OED^\dagger$ is equivalent to the category of \'etale $(\varphi_\Delta,\Gamma_\Delta)$-modules over $\OED$ (hence also to the category of continuous representations of $\GQpD$ over $\Zp$). Finally, we verify that the overconvergent Herr complex also computes Galois cohomology.

Our definition of multivariable overconvergent and Robba rings is somewhat different from that considered in \cite{Ber,K2} and in the possibly non-commutative version in \cite{ZSch}. Here the functions are required to converge on a full polyannulus whereas in these previous constructions the modulus of the variables have a fixed relation. The reason for this difference is that we have partial Frobenii to act on our rings $\OED^\dagger$ and $\mathcal{R}_\Delta$ and the relation of the moduli of variables changes under these operators. However, $\mathcal{R}_\Delta$ can naturally be viewed as a subring of the multivariable Robba ring considered in \cite{K2}.

\subsection{Multivariate overconvergent and Robba rings}

Recall (Thm.\ 4.11 in \cite{MultVarGal}) that the functors
\begin{eqnarray*}
V&\mapsto&\mathbb{D}(V):=\left(\widehat{\mathcal{E}^{ur}_\Delta}\otimes_{\Qp}V\right)^{\HQpD}\\
D&\mapsto&\mathbb{V}(D):=\bigcap_{\alpha\in\Delta}\left(\widehat{\mathcal{E}^{ur}_\Delta}\otimes_{\mathcal{E}_\Delta}D\right)^{\varphi_\alpha=\id}=h^0\Phi^\bullet(D^{ur})
\end{eqnarray*}
are quasi-inverse equivalences of categories between the Tannakian categories $\RepDQp$ and $\mathcal{D}^{et}(\varphi_{\Delta},\Gamma_\Delta,\mathcal{E}_\Delta)$. Further, this also has an integral version: for $\Zp$-representations $T$ of $\GQpD$ we define
\begin{equation*}
\mathbb{D}(T):=\left(\mathcal{O}_{\widehat{\mathcal{E}^{ur}_\Delta}}\otimes_{\Zp}T\right)^{\HQpD}
\end{equation*}
which is an equivalence of categories from $\RepDZp$ to $\mathcal{D}^{et}(\varphi_{\Delta},\Gamma_\Delta,\OED)$ with quasi-inverse $\mathbb{T}$ mapping an object $D$ to
\begin{equation*}
\mathbb{T}(D):=\bigcap_{\alpha\in\Delta}\left(\mathcal{O}_{\widehat{\mathcal{E}^{ur}_\Delta}}\otimes_{\OED}D\right)^{\varphi_\alpha=\id}=h^0\Phi^\bullet(D^{ur})\ .
\end{equation*}
Here we put $D^{ur}:=\widehat{\mathcal{E}^{ur}_\Delta}\otimes_{\mathcal{E}_\Delta}D$ (resp.\ $D^{ur}:=\mathcal{O}_{\widehat{\mathcal{E}^{ur}_\Delta}}\otimes_{\OED}D$) for an object $D$ in $\mathcal{D}^{et}(\varphi_{\Delta},\Gamma_\Delta,\mathcal{E}_\Delta)$ (resp.\ in $\mathcal{D}^{et}(\varphi_{\Delta},\Gamma_\Delta,\OED)$). 

Recall, moreover, that the ring $\mathcal{O}_{\mathcal{E}}^\dagger$ of integral \emph{overconvergent} Laurent series is defined as
\begin{equation*}
\mathcal{O}_{\mathcal{E}}^\dagger:=\left\{\sum_{i\in\mathbb{Z}}a_iX^i\mid a_i\in\Zp\ ,\text{ convergent on }\rho<|X|_p<1\text{ for some }0<\rho<1\right\}
\end{equation*}
and we put $\mathcal{E}^\dagger:=\mathcal{O}_{\mathcal{E}}^\dagger[p^{-1}]$. Both these rings are subrings of the Robba ring $\mathcal{R}:=\bigcup_{0<\rho<1}\mathcal{R}^{(\rho,1)}$ where we put
\begin{equation*}
\mathcal{R}^{(\rho,1)}:=\left\{\sum_{i\in\mathbb{Z}}a_iX^i\mid a_i\in\Qp\ ,\text{ convergent on }\rho<|X|_p<1\right\}\ .
\end{equation*}
Further, for each real number $\rho<r<1$ we have the $r$-norm on $\mathcal{R}^{(\rho,1)}$ given by the formula $|\sum_{i\in\mathbb{Z}}a_iX^i|_r:=\sup_i |a_i|_pr^i\in\mathbb{R}^{\geq 0}$ (which we extend to the whole Robba ring $\mathcal{R}$ by the same formula possibly taking $+\infty$ as a value). Recall that $\mathcal{E}^\dagger$ (resp.\ $\mathcal{O}_{\mathcal{E}}^\dagger$) consists of those elements $f\in\mathcal{R}$ such that $\lim_{r\to1}|f|_r<\infty$ (resp.\ $\lim_{r\to 1}|f|_r\leq 1$). 

Consider a copy $\mathcal{O}_{\mathcal{E}_\alpha}^\dagger$, $\mathcal{E}_\alpha^\dagger$, and $\mathcal{R}_\alpha$ of these rings for each $\alpha\in\Delta$ using the variable $X_\alpha$. We define $\mathcal{R}_\Delta:=\bigcup_{\rho\in (0,1)^\Delta} \mathcal{R}_\Delta^{({\bf \rho, 1})}$ as the ascending union of the ring of multivariable power series
\begin{align*}
\mathcal{R}_\Delta^{({\bf \rho, 1})}:= \left\{\sum_{{\bf i}=(i_\alpha)_{\alpha\in\Delta}\in\mathbb{Z}^\Delta}a_{\bf i}{\bf X}^{\bf i}\mid a_{\bf i}\in\Qp\ ,\text{ convergent on }B^{({\bf \rho, 1})}\right\}
\end{align*}
where ${\bf X}^{\bf i}:=\prod_{\alpha\in\Delta}X_\alpha^{i_\alpha}$ and the polyannulus $B^{({\bf \rho, 1})}$ for the tuple ${\bf \rho}=(\rho_\alpha)_{\alpha\in\Delta}\in(0,1)^\Delta$ is defined as
\begin{equation*}
B^{({\bf \rho, 1})}:=\left\{{\bf X}=(X_\alpha)_{\alpha\in\Delta}\in\mathbb{C}_p^\Delta\mid \rho_\alpha<|X_\alpha|_p<1\text{ for all }\alpha\in\Delta\right\}\ .
\end{equation*}
For each tuple ${\bf r}=(r_\alpha)_{\alpha\in\Delta}\in(0,1)^\Delta$ of real numbers with $\rho_\alpha<r_\alpha<1$ for all $\alpha\in\Delta$ the ${\bf r}$-norm on $\mathcal{R}_\Delta^{({\bf \rho, 1})}$ is defined as $$|\sum_{{\bf i}\in\mathbb{Z}^\Delta}a_{\bf i}{\bf X}^{\bf i}|_{\bf r}:=\sup_{\bf i}|a_{\bf i}|_p\prod_{\alpha\in\Delta}r_\alpha^{i_\alpha}\ .$$ 
Note that the ${\bf r}$-norm is multiplicative on $\mathcal{R}_\Delta^{({\bf \rho, 1})}$ for any tuple ${\bf \rho}$ with $\rho_\alpha<r_\alpha$ for all $\alpha\in\Delta$. We extend the ${\bf r}$-norm as a function to the whole Robba ring defined by the same formula taking possibly $+\infty$ as a value. We define the rings $\mathcal{E}_\Delta^\dagger$ (resp.\ $\OED^\dagger$) as subrings of $\mathcal{R}_\Delta$ consisting of functions that are bounded (resp.\ bounded by $1$) on the boundary. More precisely, we set
\begin{eqnarray*}
\mathcal{E}_\Delta^\dagger&:=&\{f\in\mathcal{R}_\Delta\mid \limsup_{{\bf r}\to (1)_{\alpha\in\Delta}}|f|_{\bf r}<\infty\}\\
\OED^\dagger&:=&\{f\in\mathcal{R}_\Delta\mid \limsup_{{\bf r}\to (1)_{\alpha\in\Delta}}|f|_{\bf r}\leq1\}\ .
\end{eqnarray*}
In other words $f\in\mathcal{R}_\Delta$ lies in $\OED^\dagger$ if and only if for each $\delta\in\mathbb{R}^{>0}$ there exists an $\varepsilon\in (0,1)$ such that for each tuple ${\bf r}\in (1-\varepsilon,1)^\Delta$ we have $|f|_{\bf r}<1+\delta$. Further, we have $\mathcal{E}_\Delta^\dagger=\OED^\dagger[p^{-1}]$. For each $\alpha\in\Delta$ we identify $\mathcal{O}_{\mathcal{E}_\alpha}^\dagger$ (resp.\ $\mathcal{E}_\alpha^\dagger$, resp.\ $\mathcal{R}_\alpha$) with a subring of $\OED^\dagger$ (resp.\ of $\mathcal{E}_\Delta^\dagger$, resp.\ of $\mathcal{R}_\Delta$).
\begin{lem}\label{coeffOEDdag}
If $f=\sum_{{\bf i}\in\mathbb{Z}^\Delta}a_{\bf i}{\bf X}^{\bf i}$ lies in $\OED^\dagger$ then we have $a_{\bf i}\in\Zp$ for all ${\bf i}\in\mathbb{Z}^\Delta$.
\end{lem}
\begin{proof}
Assume that $|a_{\bf j}|_p>1$ for some fixed ${\bf j}\in\mathbb{Z}^\Delta$ and choose the real number $0<\delta<|a_{\bf j}|_p-1$. Then for ${\bf r}$ close enough to $(1)_{\alpha\in\Delta}$ (depending on ${\bf j}$ and $|a_{\bf j}|_p-1-\delta$) we have $1+\delta<|a_{\bf j}|_p\prod_\alpha r_\alpha^{j_\alpha}\leq |f|_{\bf r}$ showing that $ \limsup_{{\bf r}\to (1)_{\alpha\in\Delta}}|f|_{\bf r}>1+\delta$ whence $f\notin \OED^\dagger$.
\end{proof}
\begin{rem}
The converse of the above Lemma is not true whenever $|\Delta|>1$. For example the Laurent series $\sum_{i=1}^{\infty}X_\alpha^{2^i}X_\beta^{-i}$ has coefficients in $\Zp$ and belongs to $\mathcal{R}_\Delta$, but not to $\OED^\dagger$ if $\alpha\neq\beta\in\Delta$.
\end{rem}
\begin{pro}\label{OEDdagpcomp}
The ring $\OED^\dagger$ is $p$-adically separated and its $p$-adic completion $\varprojlim_n\OED^\dagger/(p^n)$ is isomorphic to $\OED$. In particular, we have an injective ring homomorphism $\OED^\dagger\hookrightarrow \OED$.
\end{pro}
\begin{proof}
The $p$-adic separatedness follows directly from Lemma \ref{coeffOEDdag}. Further, it is obvious that $\OED^+=\Zp\bs X_\alpha\mid\alpha\in\Delta\js\subseteq \OED^\dagger$ and $X_\alpha^{-1}\in\OED^\dagger$ for all $\alpha\in\Delta$ whence $\OED^+[X_\Delta^{-1}]\subseteq \OED^\dagger$. Note that we have $\OED/(p^n)\cong \OED^+[X_\Delta^{-1}]/(p^n)$, and $\OED^+[X_\Delta^{-1}]/(p^n)\hookrightarrow \OED^\dagger/(p^n)$ by Lemma \ref{coeffOEDdag}, so it remains to show that $\OED^+[X_\Delta^{-1}]/(p^n)\twoheadrightarrow \OED^\dagger/(p^n)$. Namely, we need to verify that the monomials in an element $f=\sum_{{\bf i}\in\mathbb{Z}^\Delta}a_{\bf i}{\bf X}^{\bf i}\in\OED^\dagger$ with coefficients not divisible by $p^n$ have bounded denominators for any fixed $n$. Assume for contradiction that for some $n>0$ and $\alpha\in\Delta$ there exists a sequence ${\bf i}(j)_{j\geq 1}\subset \mathbb{Z}^\Delta$ of indices such that ${\bf i}(j)_\alpha\to -\infty$ and $| a_{{\bf i}(j)}|_p>p^{-n}$ for all $j\geq 1$. We claim that $f\notin \mathcal{E}_\Delta^\dagger$. Choose real numbers $0<C$ and $0<\varepsilon<1$. For any fixed $r_\alpha$ in the open interval $(1-\varepsilon,1)$, there exists a positive integer $j$ such hat $r_\alpha^{i(j)_\alpha}>Cp^n$ since the sequence $i(j)_\alpha$ tends to $-\infty$. Now for given $r_\alpha$ and this chosen $j\geq 1$ we may choose $r_\beta\in (1-\varepsilon,1)$ close enough to $1$ for all $\beta\in\Delta\setminus\{\alpha\}$ such that we still have $\prod_{\beta\in\Delta}r_\beta^{i(j)_\beta}>Cp^n$ whence we have $|f|_{\bf r}\geq |a_{{\bf i}(j)}|_p\prod_{\beta\in\Delta}r_\beta^{i(j)_\beta}>C$.
\end{proof}

\begin{rem}
We have $\mathcal{R}_\Delta\cap \OED\supsetneq \mathcal{E}_\Delta^\dagger\cap \prod_{{\bf i}\in\mathbb{Z}^\Delta}\Zp{\bf X}^{\bf i}=\mathcal{E}_\Delta^\dagger\cap \OED=\OED^\dagger$ where the intersection is considered in the $\Qp$-vector space $\prod_{{\bf i}\in\mathbb{Z}^\Delta}\Qp{\bf X}^{\bf i}$.
\end{rem}
\begin{proof}
The inclusion $\OED^\dagger\subseteq \mathcal{R}_\Delta\cap\OED$ is proven in Prop.\ \ref{OEDdagpcomp}. To see that the containment is proper note that the element $\sum_{n\geq 1}p^nX_\alpha^{2^{2^n}}X_\beta^{-2^n}$ belongs to $\mathcal{R}_\Delta\cap\OED$ but not to $\OED^\dagger$. Now assume that $f=\sum_{{\bf i}\in\mathbb{Z}^\Delta}a_{\bf i}{\bf X}^{\bf i}$ lies in $\mathcal{E}_\Delta^\dagger$ and $a_{\bf i}\in\Zp$ for all ${\bf i}\in \mathbb{Z}^\Delta$. Then there exist real numbers $0<C$ and $0<\varepsilon<1$ such that $|f|_{\bf r}\leq C$ whenever $1-\varepsilon< r_\alpha<1$ for all $\alpha\in\Delta$. In particular, we have $|a_{\bf i}|_p\prod_{\alpha\in\Delta}r_\alpha^{i_\alpha}\leq C$ which implies $$|a_{\bf i}|_p\prod_{\alpha\in\Delta}r_\alpha^{i_\alpha}\leq |a_{\bf i}|_p\prod_{\alpha\in\Delta}r_\alpha^{\min(0,i_\alpha)}\leq C\ ,$$ too, by letting $r_\beta$ go to $1$ for all $\beta\in\Delta$ with $i_\beta>0$. Now the function 
$$H({\bf r}):= \sup_{\bf i}(|a_{\bf i}|_p\prod_{\alpha\in\Delta}r_\alpha^{\min(0,i_\alpha)})$$ 
is monotone decreasing and bounded by $C$ on the polyannulus ${\bf r}\in (1-\varepsilon,1)^\Delta$. So it suffices to show that $\limsup_{r\to 1}H_0(r)\leq 1$ where $H_0(r):=H((r)_{\alpha\in\Delta})$ as a function on the interval $(1-\varepsilon,1)$. Choose a real number $\delta>0$ and $1-\varepsilon<r'<1$. Now for large enough integer $N$ we have $r'^n>C$ for $n<-N$, therefore there exists a real number $r'<r''<1$ such that we still have $(r'/r'')^n>C$ for $n<-N$. By possibly increasing $r''$ further (for fixed $N$) we may assume that $r''^{-N}\leq 1+\delta$. So for those ${\bf i}\in\mathbb{Z}^\Delta$ with $\sum_{\alpha\in\Delta}\min(0,i_\alpha)<-N$ we have $$|a_{\bf i}|_p\prod_{\alpha\in\Delta}r''^{\min(0,i_\alpha)}<(\frac{r''}{r'})^{-N} |a_{\bf i}|_p\prod_{\alpha\in\Delta}r'^{\min(0,i_\alpha)}<C^{-1}C=1\ .$$ On the other hand for those ${\bf i}\in\mathbb{Z}^\Delta$ with $\sum_{\alpha\in\Delta}\min(0,i_\alpha)\geq-N$ we have  $$|a_{\bf i}|_p\prod_{\alpha\in\Delta}r''^{\min(0,i_\alpha)}\leq |a_{\bf i}|_pr''^{-N}\leq 1+\delta\ .$$ All in all we obtain $H_0(r'')\leq 1+\delta$, so we deduce $f\in \OED^\dagger$ as $\delta$ was arbitrary.
\end{proof}

We now equip the rings $\OED^\dagger$, $\mathcal{E}_\Delta^\dagger$, and $\mathcal{R}_\Delta$ with the action of the operators $\varphi_\alpha$ for all $\alpha\in\Delta$ and by the group $\Gamma_\Delta$ (as in section \ref{secnotations}). Let $p^{-1/(p-1)}<r_\alpha<1$ be a real number. Then by the ultrametric inequality we compute 
\begin{align*}
|\varphi(X_\alpha)-X_\alpha^p|_{r_\alpha}=|\sum_{j=1}^{p-1}\binom{p}{j}X_\alpha^j|_{r_\alpha}\leq\max_{0<j<p}|\binom{p}{j}X_\alpha^j|_{r_\alpha} =\frac{r_\alpha}{p}<r_\alpha^{p}=|X_\alpha^p|_{r_\alpha}\ .
\end{align*}
In particular, we obtain $|\varphi_\alpha(X_\alpha)|_{r_\alpha}=|X_\alpha^p|_{r_\alpha}=r_\alpha^p$. We deduce $$|\varphi(X_\alpha^i)-X_\alpha^{pi}|_{r_\alpha}=|(\varphi_\alpha(X_\alpha)-X_\alpha^p)\sum_{j=0}^{i-1}\varphi_\alpha(X_\alpha)^jX_\alpha^{p(i-1-j)}|_{r_\alpha}<r_\alpha^{p}r_\alpha^{p(i-1)}=|X_\alpha^{pi}|_{r_\alpha}$$ for all $i\in\mathbb{Z}^{>0}$. Further, since $\varphi_\alpha(X_\alpha)$ is invertible in $\mathcal{O}_{\mathcal{E}_\alpha}^\dagger$, it is also invertible in the rings $\OED^\dagger$, $\mathcal{E}_\Delta^\dagger$, and $\mathcal{R}_\Delta$. Moreover, we have  
\begin{equation}
|\varphi(X_\alpha^{-i})-X_\alpha^{-pi}|_{r_\alpha}=|\frac{X_\alpha^{pi}-\varphi_\alpha(X_\alpha)^i}{X_\alpha^{pi}\varphi_\alpha(X_\alpha)^i}|_{r_\alpha}<r_\alpha^{-pi}=|X_\alpha^{-pi}|_{r_\alpha}\label{phiestimate1}
\end{equation}
for all $i>0$. Therefore if $f=\sum_{{\bf i}\in\mathbb{Z}^\Delta}a_{\bf i}{\bf X}^{\bf i}\in\mathcal{R}_\Delta$ is convergent on the polyannulus $B^{({\bf \rho,1})}$ for some ${\bf \rho}\in (0,1)^\Delta$ with $\rho_\alpha>p^{-1/(p-1)}$ then the formal sum
$$\varphi_\alpha(f):= \sum_{{\bf i}\in\mathbb{Z}^\Delta}a_{\bf i}\varphi_\alpha(X_\alpha)^{i_\alpha}\prod_{\beta\in\Delta\setminus\{\alpha\}}X_\beta^{i_\beta}$$
also converges in the ${\bf r}$-norm whenever $r_\alpha>\rho_\alpha^{1/p}$ and $r_\beta>\rho_\beta$ for all $\beta\in\Delta\setminus\{\alpha\}$. This way we obtain an injective ring endomorphism $\varphi_\alpha\colon \mathcal{R}_\Delta\to \mathcal{R}_\Delta$ such that for all $f\in\mathcal{R}$ and ${\bf r}\in(0,1)^\Delta$ with $r_\alpha>p^{-1/(p-1)}$ we have $|\varphi_\alpha(f)|_{\bf r}=|f|_{\bf r'}$ where $r'_\alpha=r_\alpha^p$ and $r'_\beta=r_\beta$ for all $\alpha\neq \beta\in\Delta$. This, in particular, shows that $\varphi_\alpha(f)$ lies in the subring $\OED^\dagger$ (resp.\ in $\mathcal{E}_\Delta^\dagger$) if and only if so does $f$.

\begin{lem}\label{psiestimate}
Fix $\alpha\in\Delta$ and let ${\bf r}\in(0,1)^\Delta$ with $r_\alpha>p^{-1/(p-1)}$, and let ${\bf r'}\in(0,1)^\Delta$ such that $r'_\alpha=r_\alpha^p$ and $r'_\beta=r_\beta$ for all $\alpha\neq \beta\in\Delta$. Let $f_0,\dots,f_{p-1}\in \mathcal{R}_\Delta$.
\begin{enumerate}[$(1)$]
\item If $f_0,\dots,f_{p-1}$ are all convergent on the polyannulus $B^{\bf (r',1)}$, then we have
$$r^{p-1}_\alpha \max_j|f_j|_{\bf r'}\leq|\sum_{j=0}^{p-1}(1+X_\alpha)^j\varphi_\alpha(f_j)|_{\bf r}\leq \max_j|f_j|_{\bf r'}\ .$$ 
\item If $\sum_{j=0}^{p-1}(1+X_\alpha)^j\varphi_\alpha(f_j)$ is convergent on the polyannulus $B^{({\bf r, 1})}$, then each $f_i$ ($i=0,\dots,p-1$) is convergent on the polyannulus $B^{({\bf r', 1})}$.
\end{enumerate}
\end{lem}
\begin{proof}
The second inequality follows from the from the formula $|\varphi_\alpha(f)|_{\bf r}=|f|_{\bf r'}$ by the ultrametric inequality noting $|(1+X_\alpha)^j|_{\bf r}=1$ for all $j=0,\dots,p-1$. For the other inequality we may assume without loss of generality that $f_j$ lies in $\mathcal{R}_\alpha$ since the functions $|\cdot|_{\bf r}$ are defined termwise. In this case we choose $0\leq j_0\leq p-1$ such that $|f_{j_0}|_{\bf r'}$ is maximal (if the maximum is taken at more than one value of $j_0$ then we take the biggest $j_0$ among them) and choose the integer $i_0$ such that the supremum defining $|f_{j_0}|_{\bf r'}$ is taken on the coefficient of $X_\alpha^{i_0}$. We claim that the term with $X_\alpha^{pi_0+j_0}$ in $\sum_{j=0}^{p-1}(1+X_\alpha)^j\varphi_\alpha(f_j)$ has $|\cdot|_{\bf r}$ at least $r^{p-1}_\alpha \max_j|f_j|_{\bf r'}=r^{p-1}_\alpha |f_{j_0}|_{\bf r'}\leq r^{j_0}_\alpha |f_{j_0}|_{\bf r'}$. To show this write $f_j=\sum_{i=-\infty}^{\infty}a_{i,j}X_\alpha^i$ ($j=0,\dots,p-1$). By \eqref{phiestimate1} and our assumption $r_\alpha>p^{-1/(p-1)}$  we have $|\varphi_\alpha(f_j)-\sum_{i=-\infty}^{\infty}a_{i,j}X_\alpha^{pi}|_{\bf r}<|\varphi_\alpha(f_j)|_{\bf r}=|f_j|_{\bf r'}$. Adding all these estimates and using the ultrametric inequality we obtain $$|\sum_{j=0}^{p-1}(1+X_\alpha)^j\varphi_\alpha(f_j)-\sum_{j=0}^{p-1}\sum_{i=-\infty}^\infty a_{i,j}(1+X_\alpha)^jX_\alpha^{pi}|_{\bf r}<|\sum_{j=0}^{p-1}\sum_{i=-\infty}^\infty a_{i,j}(1+X_\alpha)^jX_\alpha^{pi}|_{\bf r}$$ whence
$$|\sum_{j=0}^{p-1}(1+X_\alpha)^j\varphi_\alpha(f_j)|_{\bf r}=|\sum_{j=0}^{p-1}\sum_{i=-\infty}^\infty a_{i,j}(1+X_\alpha)^jX_\alpha^{pi}|_{\bf r}\ .$$
Now in the infinite sum on the right hand side only the terms $\sum_{j_0\leq j\leq p-1}a_{i_0,j}(1+X_\alpha)^jX_\alpha^{pi_0}$ contribute to the coefficient of $X_\alpha^{j_0+pi_0}$, namely the coefficient is $\sum_{j_0\leq j\leq p-1}a_{i_0,j}\binom{j}{j_0}$. Note that $a_{i_0,j}X_\alpha^{i_0}$ is a term in $f_j$ whence $|a_{i_0,j}X_\alpha^{pi_0}|_{\bf r}=|a_{i_0,j}X_\alpha^{i_0}|_{\bf r'}\leq |f_j|_{\bf r'}\leq |f_{j_0}|_{\bf r'}$ with strict inequality at the end in case $j\neq j_0$ by the choice of $j_0$. By the choice $i_0$ we have $|f_{j_0}|_{\bf r'}=|a_{i_0,j_0}X_\alpha^{i_0}|_{\bf r'}$, so we deduce $|\sum_{j_0\leq j\leq p-1}a_{i_0,j}\binom{j}{j_0}X_\alpha^{j_0+pi_0}|_{\bf r}=|a_{i_0,j_0}X_\alpha^{j_0+pi_0}|_{\bf r}=r_\alpha^{j_0}|f_{j_0}|_{\bf r'}$ as claimed since the sum on the left hand side has $a_{i_0,j_0}X_\alpha^{j_0+pi_0}$ as dominant term with all the others having smaller $|\cdot|_{\bf r}$.

For the second statement we may write each $f_j=\sum_{{\bf k}\in\mathbb{Z}^\Delta}a_{{\bf k},j}\prod_{\alpha\in\Delta}X_\alpha^{k_\alpha}$ ($j=0,\dots,p-1$) as an infinite formal sum and put $f_j^{(N)}:=\sum_{{\bf k}\in(\mathbb{Z}\cap [-N,N])^\Delta}a_{{\bf k},j}\prod_{\alpha\in\Delta}X_\alpha^{k_\alpha}$. We have $$r^{p-1}_\alpha \max_j|f^{(N)}_j|_{\bf r'}\leq|\sum_{j=0}^{p-1}(1+X_\alpha)^j\varphi_\alpha(f^{(N)}_j)|_{\bf r}$$ by the first statement since finite sums converge on any polyannulus. Taking the limit as $N\to\infty$ we deduce that whenever $|\sum_{j=0}^{p-1}(1+X_\alpha)^j\varphi_\alpha(f_j)|_{\bf r}$ is finite for some fixed ${\bf r}\in(0,1)^\Delta$ with $r_\alpha>p^{-1/(p-1)}$, so is $|f_j|_{\bf r'}$ for all $j=0,\dots,p-1$.
\end{proof}

\begin{pro}
For all $\alpha\in \Delta$ we have 
\begin{eqnarray*}
\mathcal{R}_\Delta&=&\bigoplus_{j=0}^{p-1}(1+X_\alpha)^j\varphi_\alpha(\mathcal{R}_\Delta)\\
\mathcal{E}_\Delta^\dagger&=&\bigoplus_{j=0}^{p-1}(1+X_\alpha)^j\varphi_\alpha(\mathcal{E}_\Delta^\dagger)\\
\OED^\dagger&=&\bigoplus_{j=0}^{p-1}(1+X_\alpha)^j\varphi_\alpha(\OED^\dagger)\ .
\end{eqnarray*}
\end{pro}
\begin{proof}
By collecting the terms with $\prod_{\beta\in\Delta\setminus\{\alpha\}}X_\beta^{i_\beta}$ for each tuple ${\bf i}=(i_\beta)_{\beta\in\Delta\setminus\{\alpha\}}\in\mathbb{Z}^{\Delta\setminus\{\alpha\}}$ in the expansion of any element $f$ in $\mathcal{R}_\Delta$ (resp.\ in $\mathcal{E}_\Delta^\dagger$, resp.\ in $\OED^\dagger$) we may write $$f=\sum_{{\bf i}\in\mathbb{Z}^{\Delta\setminus\{\alpha\}}}f_{\bf i}\prod_{\beta\in\Delta\setminus\{\alpha\}}X_\beta^{i_\beta}$$
for some $f_{\bf i}$ in $\mathcal{R}_\alpha$ (resp.\ in $\mathcal{E}_\alpha^\dagger$, resp.\ in $\mathcal{O}_{\mathcal{E}_\alpha}^\dagger$). Since the operator $\varphi_\alpha$ respects this expansion, we deduce immediately that the sums in the statement are all direct. In order to prove these equalities we may write each $f_{\bf i}$ as a sum $f_{\bf i}=\sum_{j=0}^{p-1}(1+X_\alpha)^j\varphi_\alpha(f_{{\bf i},j})$ for some $f_{{\bf i},j}$ in $\mathcal{R}_\alpha$ (resp.\ in $\mathcal{E}_\alpha^\dagger$, resp.\ in $\mathcal{O}_{\mathcal{E}_\alpha}^\dagger$). Now whenever $f$ is convergent on the polyannulus $B^{({\bf r, 1})}$ for some ${\bf r}\in(0,1)^\Delta$ with $r_\alpha>p^{-1/(p-1)}$ then by Lemma \ref{psiestimate} applied to the sum $\sum_{j=0}^{p-1}(1+X_\alpha)^j\varphi_\alpha(f_{{\bf i},j})$ for each ${\bf i}\in\mathbb{Z}^\Delta$ we deduce that the formal sum $\sum_{{\bf i}\in\mathbb{Z}^{\Delta\setminus\{\alpha\}}}f_{{\bf i},j}\prod_{\beta\in\Delta\setminus\{\alpha\}}X_\beta^{i_\beta}$ converges on the polyannulus $B^{({\bf r', 1})}$ for each $j=0,\dots,p-1$. The statement on the decomposition of $\mathcal{E}_\Delta^\dagger$ and $\OED^\dagger$ also follows from Lemma \ref{psiestimate} noting that $r_\alpha^{p-1}$ tends to $1$ as $r_\alpha\to 1$.
\end{proof}

Now if $\gamma=(\gamma_\alpha)_{\alpha\in\Delta}\in\Gamma_\Delta$ is arbitrary then we put $\gamma(X_\alpha):=(1+X_\alpha)^{\chi_\alpha(\gamma_\alpha)}-1$ for all $\alpha\in\Delta$ and we extend this operator to all monomials ${\bf X}^{\bf i}=\prod_{\alpha\in\Delta}X_\alpha^{i_\alpha}$ multiplicatively (${\bf i}=(i_\alpha)_{\alpha\in\Delta}\in\mathbb{Z}^\Delta$). Now note that for any tuple ${\bf r}\in(0,1)^\Delta$ and ${\bf i}\in\mathbb{Z}^\Delta$ we have $|\gamma({\bf X}^{\bf i})|_{\bf r}=|{\bf X}^{\bf i}|_{\bf r}$, so this defines an action of the group $\Gamma_\Delta$ on each of the rings $\mathcal{R}_\Delta$, $\mathcal{E}_\Delta^\dagger$, and $\OED^\dagger$. This action commutes with the operators $\varphi_\alpha$ (which also commute with each other). Now an \'etale $(\varphi_\Delta,\Gamma_\Delta)$-module over $\OED^\dagger$ is a finitely generated free module $D^\dagger$ over $\OED^\dagger$ with commuting semilinear action of the group $\Gamma_\Delta$ and the operators $\varphi_\alpha$ for each $\alpha\in\Delta$ such that the map
\begin{equation*}
\id\otimes\varphi_\alpha\colon \OED^\dagger\otimes_{\varphi_\alpha,\OED^\dagger}D^\dagger\to D^\dagger
\end{equation*}
is an isomorphism for all $\alpha\in\Delta$. An \'etale $(\varphi_\Delta,\Gamma_\Delta)$-module over $\mathcal{R}_\Delta$ (resp.\ over $\mathcal{E}_\Delta^\dagger$) is a finitely generated free module over $\mathcal{R}_\Delta$ (resp.\ over $\mathcal{E}_\Delta^\dagger$) with commuting semilinear action of the group $\Gamma_\Delta$ and the operators $\varphi_\alpha$ for each $\alpha\in\Delta$ that comes as base extension from an \'etale $(\varphi_\Delta,\Gamma_\Delta)$-module over $\OED^\dagger$. We denote by $\mathcal{D}^{et}(\varphi_{\Delta},\Gamma_\Delta,\OED^\dagger)$, $\mathcal{D}^{et}(\varphi_{\Delta},\Gamma_\Delta,\mathcal{E}_\Delta^\dagger)$, and $\mathcal{D}^{et}(\varphi_{\Delta},\Gamma_\Delta,\mathcal{R}_\Delta)$ the categories of \'etale $(\varphi_\Delta,\Gamma_\Delta)$-modules over the respective rings.

We finish this section by proving certain ring theoretic properties of $\OED^\dagger$ and by deriving from them certain consequences on the structure of \'etale $(\varphi_\Delta,\Gamma_\Delta)$-modules over $\OED^\dagger$ and over $\mathcal{E}_\Delta^\dagger$. 

\begin{lem}\label{modpjacobson}
The Jacobson radical of the ring $E_\Delta=\OED/(p)=\OED^\dagger/(p)$ is zero.
\end{lem}
\begin{proof}
Suppose we have $0\neq \lambda\in \operatorname{Jac}(E_\Delta)$. By multiplying $\lambda$ by a monomial $\prod_{\alpha\in\Delta}X_\alpha^{k_\alpha}$ for some $k_\alpha\in\mathbb{Z}$ we may assume without loss of generality that $\lambda\in E_\Delta^+:=\Fp\bs X_\alpha\mid \alpha\in\Delta\js\subset E_\Delta=E_\Delta^+[X_\Delta^{-1}]$ and $\lambda$ is not divisible by any of the variables $X_\alpha$ ($\alpha\in\Delta$). Note that the constant term of $\lambda$ is zero since otherwise $\lambda$ would be invertible. Therefore $\lambda+\prod_{\alpha\in\Delta}X_\alpha$ is not invertible in $E_\Delta$ either as it is not divisible by any of the $X_\alpha$ and it is not invertible in $E_\Delta^+$ either. This contradicts our assumption that $\lambda\in \operatorname{Jac}(E_\Delta)$.
\end{proof}

\begin{pro}\label{jacoverconv}
We have $\operatorname{Jac}(\OED^\dagger)=(p)$.
\end{pro}
\begin{proof}
By Lemma \ref{modpjacobson} we are reduced to showing that $1+px$ is invertible in $\OED^\dagger$ for all $x\in\OED^\dagger$. Since $\lim_{{\bf r}\to (1)_{\alpha\in\Delta}}|x|_{\bf r}\leq 1$ there exists a real number $0<\rho=\rho(\varepsilon)<1$ such that for all ${\bf r}\in (\rho,1)^\Delta$ we have $|x|_{\bf r}<1+\varepsilon$ whence $|px|_{\bf r}=p^{-1}|x|_{\bf r}<p^{-1}+p^{-1}\varepsilon<1$ for $0<\varepsilon$ small enough. In particular, the formal inverse $(1+px)^{-1}=\sum_{j=0}^\infty(-px)^j$ converges in $|\cdot|_{\bf r}$ and we have $|(1+px)^{-1}|_{\bf r}=1$ for all ${\bf r}\in(\rho,1)^\Delta$. 
\end{proof}

\begin{rem}
It is also true (and easier to prove) that we also have $\operatorname{Jac}(\OED)=(p)$.
\end{rem}

\begin{pro}\label{noetheroverconv}
The ring $\OED^\dagger$ is noetherian.
\end{pro}
\begin{proof}
This follows the same way as Lemma 1.3 in \cite{K2}. We are going to show that the ring $\OED^\dagger$ is a weakly complete finitely generated algebra over $\OED^+:=\Zp\bs X_\alpha\mid \alpha\in\Delta\js$ with ideal $(p)$ and generator $X_\Delta^{-1}$ in the sense of Fulton \cite{Ful}, hence $\OED^\dagger$ is noetherian. Pick an element $f=\sum_{{\bf k}\in\mathbb{Z}^{\Delta}}a_{\bf k}\prod_{\alpha\in\Delta}X_\alpha^{k_\alpha}\in \OED^\dagger$ and for all $n>0$ let $h_n$ be the smallest positive integer such that $f$ modulo $p^n$ lies in $X_\Delta^{-h_n}\mathbb{Z}/(p^n)\bs X_\alpha,\alpha\in\Delta\js$. In other words there exists an index ${\bf k_n}=(k_{n,\alpha})_{\alpha\in\Delta}\in\mathbb{Z}^\Delta$ and an $\alpha_n\in\Delta$ such that $k_{n,\alpha_n}=-h_n$ and $p^n\nmid a_{\bf k_n}$ and $h_n$ is maximal with this property. For a fixed real number $\varepsilon>0$ there exists a $\rho=\rho(\varepsilon)\in (0,1)$ such that $|f|_{\bf r}<1+\varepsilon$ for all ${\bf r}\in (\rho,1)^\Delta$. Now we fix a real number $\rho_1\in (\rho,1)$ and pick ${\bf r_n}=(r_{n,\alpha})_{\alpha\in\Delta}\in (\rho,1)^\Delta$ such that $r_{n,\alpha_n}=\rho_1$ and $r_{n,\beta}\in (\rho,1)$ is arbitrary for all $\beta\in \Delta\setminus\{\alpha_n\}$ and $n>0$. We compute
\begin{equation*}
1+\varepsilon>|f|_{\bf r}\geq |a_{\bf k_n}|_{\bf r_n}\rho_1^{-h_n}\prod_{\beta\in\Delta\setminus\{\alpha_n\}}r_{n,\beta}^{k_{n,\beta}}\geq p^{1-n}\rho_1^{-h_n}\prod_{\beta\in\Delta\setminus\{\alpha_n\}}r_{n,\beta}^{k_{n,\beta}}\ .
\end{equation*}
Now for fixed $n$ we let $r_{n,\beta}$ tend to $1$ for all $\beta\in\Delta\setminus\{\alpha_n\}$ and deduce $1+\varepsilon\geq p^{1-n}\rho_1^{-h_n}$. Taking logarithm we obtain $h_n\leq \frac{\log p}{-\log\rho_1}n+\frac{\log(1+\varepsilon)-\log p}{-\log \rho_1}$ showing the weakly completeness as the constants $\varepsilon$ and $\rho_1$ are chosen independently of $n$.
\end{proof}

\begin{cor}
$\mathcal{E}_\Delta^\dagger=\OED^\dagger[p^{-1}]$ is noetherian.
\end{cor}

\begin{thm}\label{overconvfree}
Let $D^\dagger$ be an object in $\mathcal{D}^{et}(\varphi_{\Delta},\Gamma_\Delta,\OED^\dagger)$ that is $p$-torsion free. Then $D^\dagger$ is free as a module over $\OED^\dagger$.
\end{thm}
\begin{proof}
$D^\dagger/pD^\dagger$ is an object in $\mathcal{D}^{et}(\varphi_{\Delta},\Gamma_\Delta,E_\Delta)$ therefore $D^\dagger/pD^\dagger$ is free as a module over $E_\Delta$ by Cor.\ 3.16 in \cite{MultVarGal}. Let $e_1,\dots,e_k\in D$ be an arbitrary lift of a set $\{\overline{e_1},\dots,\overline{e_k}\}$ of free generators of $D^\dagger/pD^\dagger$. We claim that $e_1,\dots,e_k$ freely generate $D^\dagger$. Since $D^\dagger$ has no $p$-torsion, the multiplication-by-$p$ map $p^{n-1}D^\dagger/p^nD^\dagger\to p^nD^\dagger/p^{n+1}D^\dagger$ is an isomorphism for all $n\geq 1$. Therefore the $\OED^\dagger$-submodule $M$ of $D^\dagger$ generated by $e_1,\dots,e_k$ is free of rank $k$ over $\OED^\dagger$. On the other hand, the inclusion $M\hookrightarrow D^\dagger$ induces an isomorphism $M/pM\cong D^\dagger/pD^\dagger$ by construction. We deduce $M=D^\dagger$ by Nakayama's lemma that we may use by Propositions \ref{jacoverconv} and \ref{noetheroverconv}.
\end{proof}

\begin{cor}
Any object in $\mathcal{D}^{et}(\varphi_{\Delta},\Gamma_\Delta,\mathcal{E}_\Delta^\dagger)$ (resp.\ in $\mathcal{D}^{et}(\varphi_{\Delta},\Gamma_\Delta,\mathcal{R}_\Delta)$) is a free module over $\mathcal{E}_\Delta^\dagger$ (resp.\ over $\mathcal{R}_\Delta$).
\end{cor}

\subsection{Overconvergence of $\GQpD$-representations}

Recall \cite{CC} that there exists a subfield $\widehat{\mathcal{E}^{ur}}^\dagger$ (denoted by ${\bf B}^\dagger$ in op.\ cit.) of the field $\widehat{\mathcal{E}^{ur}}$ (denoted by ${\bf B}$ in op.\ cit.) with ring of integers $\mathcal{O}_{\widehat{\mathcal{E}^{ur}}}^\dagger=\widehat{\mathcal{E}^{ur}}^\dagger\cap \mathcal{O}_{\widehat{\mathcal{E}^{ur}}}$ (denoted by ${\bf A}^\dagger$ in op.\ cit.) satisfying certain convergence conditions such that $(\widehat{\mathcal{E}^{ur}}^\dagger)^{H_{\Qp}}=\mathcal{E}^\dagger$ (denoted by ${\bf B}_{\Qp}^\dagger$ in op.\ cit.) and $(\mathcal{O}_{\widehat{\mathcal{E}^{ur}}}^\dagger)^{H_{\Qp}}=\mathcal{O}_{\mathcal{E}}^\dagger$ (denoted by ${\bf A}_{\Qp}^\dagger$ in op.\ cit.). The overconvergent $(\varphi,\Gamma)$-module corresponding to a continuous $p$-adic representation $V$ of $G_{\Qp}$ is defined as
\begin{equation*}
\mathbb{D}^\dagger(V):=\left(\mathcal{O}_{\widehat{\mathcal{E}^{ur}}}^\dagger\otimes_{\Zp}V\right)^{H_{\Qp}} \ .
\end{equation*}
The main result of Cherbonnier and Colmez (in case $K=\Qp$) states that any such $V$ is \emph{overconvergent}, ie.\ we have $\dim_{\mathcal{E}^\dagger}\mathbb{D}^\dagger(V)=\dim_{\Qp}V$. In particular, we have $\mathbb{D}(V)\cong \mathcal{E}\otimes_{\mathcal{E}^\dagger}\mathbb{D}^\dagger(V)$. Now we consider a copy $\widehat{\mathcal{E}_\alpha^{ur}}^\dagger$ (resp.\ $\mathcal{O}_{\widehat{\mathcal{E}_\alpha^{ur}}}^\dagger$) of the ring $\widehat{\mathcal{E}^{ur}}^\dagger$ for each $\alpha\in\Delta$ and put 
\begin{align*}
\mathcal{O}_{\mathcal{E}_{\Delta,\circ}}^\dagger:=\bigotimes_{\alpha\in\Delta,\Zp}\mathcal{O}_{\mathcal{E}_{\alpha}}^\dagger\ ,\qquad \mathcal{O}_{\widehat{\mathcal{E}^{ur}_{\Delta,\circ}}}^\dagger:=\bigotimes_{\alpha\in\Delta,\Zp}\mathcal{O}_{\widehat{\mathcal{E}^{ur}_{\alpha}}}^\dagger\ ,\qquad \mathcal{O}_{\widehat{\mathcal{E}^{ur}_{\Delta}}}^\dagger:=\OED^\dagger\otimes_{\mathcal{O}_{\mathcal{E}_{\Delta,\circ}}^\dagger}\mathcal{O}_{\widehat{\mathcal{E}^{ur}_{\Delta,\circ}}}^\dagger\ ,\\
\mathcal{E}_{\Delta,\circ}^\dagger:=\bigotimes_{\alpha\in\Delta,\Qp}\mathcal{E}_{\alpha}^\dagger=\mathcal{O}_{\mathcal{E}_{\Delta,\circ}}^\dagger[p^{-1}]\ ,\qquad \widehat{\mathcal{E}^{ur}_{\Delta,\circ}}^\dagger:=\bigotimes_{\alpha\in\Delta,\Qp}\widehat{\mathcal{E}^{ur}_{\alpha}}^\dagger=\mathcal{O}_{\widehat{\mathcal{E}^{ur}_{\Delta,\circ}}}^\dagger[p^{-1}]\ ,\ \text{and}\\
\widehat{\mathcal{E}^{ur}_{\Delta}}^\dagger:=\mathcal{E}_\Delta^\dagger\otimes_{\mathcal{E}_{\Delta,\circ}^\dagger}\widehat{\mathcal{E}^{ur}_{\Delta,\circ}}^\dagger=\mathcal{O}_{\widehat{\mathcal{E}^{ur}_{\Delta}}}^\dagger[p^{-1}]\ .
\end{align*}
The rings $\mathcal{O}_{\widehat{\mathcal{E}^{ur}_{\Delta,\circ}}}^\dagger$, $\mathcal{O}_{\widehat{\mathcal{E}^{ur}_{\Delta}}}^\dagger$,  $\widehat{\mathcal{E}^{ur}_{\Delta,\circ}}^\dagger$, and $\widehat{\mathcal{E}^{ur}_{\Delta}}^\dagger$ admit an action of the group $\GQpD$ and the operators $\varphi_\alpha$ for all $\alpha\in\Delta$ the following way: Both $\GQpa$ and $\varphi_\alpha$ act on the term $\mathcal{O}_{\mathcal{E}_{\alpha}}^\dagger$ (resp.\ $\widehat{\mathcal{E}_\alpha^{ur}}^\dagger$) in the tensor product defining $\mathcal{O}_{\widehat{\mathcal{E}^{ur}_{\Delta,\circ}}}^\dagger$ (resp.\ $\widehat{\mathcal{E}^{ur}_{\Delta,\circ}}^\dagger$) and leaves the other terms inert. On the other hand, the group $\GQpD$ acts on $\mathcal{O}_{\widehat{\mathcal{E}^{ur}_{\Delta}}}^\dagger$ (resp.\ on $\widehat{\mathcal{E}^{ur}_{\Delta}}^\dagger$) via the second term and all the operators $\varphi_\alpha$ ($\alpha\in\Delta$) via acting on both terms.

\begin{lem}\label{convergence}
We have $(\widehat{\mathcal{E}^{ur}_{\Delta,\circ}}^\dagger)^{\HQpD}=\mathcal{E}_{\Delta,\circ}^\dagger$, $(\widehat{\mathcal{E}^{ur}_{\Delta}}^\dagger)^{\HQpD}=\mathcal{E}_\Delta^\dagger$, $(\mathcal{O}_{\widehat{\mathcal{E}^{ur}_{\Delta,\circ}}}^\dagger)^{\HQpD}=\mathcal{O}_{\mathcal{E}_{\Delta,\circ}}^\dagger$, and $(\mathcal{O}_{\widehat{\mathcal{E}^{ur}_{\Delta}}}^\dagger)^{\HQpD}=\OED^\dagger$.
\end{lem}
\begin{proof}
The first statement follows by induction noting that the tensor product is taken over a field $\Qp$ and the action is componentwise. The second statement is also proven by the same inductional argument using the identification $\widehat{\mathcal{E}^{ur}_{\Delta}}^\dagger\cong \mathcal{E}_\Delta^\dagger\prod_{\alpha\in\Delta}\otimes_{\mathcal{E}_{\alpha}^\dagger}\widehat{\mathcal{E}^{ur}_{\alpha}}^\dagger$ since in the latter expression the tensor products are again taken over the fields $\mathcal{E}_{\alpha}^\dagger$. The integral versions follow by taking intersections with $\mathcal{O}_{\widehat{\mathcal{E}^{ur}_{\Delta,\circ}}}^\dagger$, resp.\ with $\mathcal{O}_{\widehat{\mathcal{E}^{ur}_{\Delta}}}^\dagger$. 
\end{proof}

We say that an \'etale $(\varphi_\Delta,\Gamma_\Delta)$-module over $\OED$ (resp.\ over $\mathcal{E}_\Delta$) is \emph{overconvergent} if it comes as base extension from an object in $\mathcal{D}^{et}(\varphi_{\Delta},\Gamma_\Delta,\OED^\dagger)$ (resp.\ in $\mathcal{D}^{et}(\varphi_{\Delta},\Gamma_\Delta,\mathcal{E}_\Delta^\dagger)$). An object $T$ in $\operatorname{Rep}_{\Zp}(\GQpD)$ (resp.\ $V$ in $\operatorname{Rep}_{\Qp}(\GQpD)$) is said to be \emph{overconvergent} if
\begin{eqnarray*}
\mathbb{D}^\dagger(T/T[p^\infty]):=\left(\mathcal{O}_{\widehat{\mathcal{E}^{ur}_{\Delta}}}^\dagger\otimes_{\Zp}T/T[p^\infty]\right)^{\HQpD}\text{ , resp. }
\mathbb{D}^\dagger(V):=\left(\widehat{\mathcal{E}^{ur}_{\Delta}}^\dagger\otimes_{\Qp}V\right)^{\HQpD}
\end{eqnarray*}
is a free \'etale $(\varphi_\Delta,\Gamma_\Delta)$-module of rank $\rk_{\Zp}T$ over $\OED^\dagger$ (resp.\ of rank $\dim_{\Qp}V$ over $\mathcal{E}_\Delta^\dagger$). The overconvergence of $p$-adic representations of $\GQpD$ is proven in the following multivariable analogue of the grounbreaking result of Cherbonnier and Colmez \cite{CC}. 

\begin{pro}\label{overconverge}
Any object $T$ in $\operatorname{Rep}_{\Zp}(\GQpD)$ (resp.\ $V$ in $\operatorname{Rep}_{\Qp}(\GQpD)$) is overconvergent.
\end{pro}
\begin{proof}
By compactness of $\GQpD$ there is a $\Zp$-lattice $T$ in any object $V$ in $\operatorname{Rep}_{\Qp}(\GQpD)$ that is stable under the action of $\GQpD$. In particular, $\mathbb{D}^\dagger(V)\cong \mathbb{D}^\dagger(T)[p^{-1}]$, so it suffices to show the integral statement. Further, we may assume without loss of generality that $T$ is $p$-torsion free. We start the proof by a general Lemma of independent interest in group representation theory that will be important in the sequel.
\begin{lem}\label{descentrep}
Let $R\leq S$ be two discrete valuation rings with maximal ideals $\mfp\lhd R$ and $P\lhd S$ such that $R\cap P=\mfp$ and the residue field $R/\mfp$ is infinite. Assume that $V$ and $W$ are two finite free modules over $R$ with linear actions of a group $G$ such that $S\otimes_R V\cong S\otimes_R W$ as representations of $G$. Then $V\cong W$.
\end{lem}
\begin{proof}
This is classical, but for the convenience of the reader (and the lack of reference treating this generality) we give a proof. Pick a basis $v_1,\dots,v_n$ (resp.\ $w_1,\dots,w_n$) in the $R$-module $V$ (resp.\ $W$) and denote by $\rho(g)\in \GL_n(R)$ (resp.\ by $\tau(g)\in \GL_n(R)$) the matrix of the action of $g\in G$ on $V$ (resp.\ on $W$) in this basis. The isomorphism $S\otimes_R V\cong S\otimes_R W$ provides us with a matrix $B\in \GL_n(S)$ such that we have $B\rho(g)=\tau(g)B$ for all $g\in G$. Now the entries of $B$ together with $R=R\cdot 1\leq S$ generate an $R$-submodule of $S$ which is free since $R$ is a DVR and $S$ has no $\mfp$-torsion. We pick a basis $u_0=1,u_1,\dots,u_r\in S$ of this free $R$-module, so we may write $B=\sum_{i=0}^r B_iu_i$ with matrices $B_i\in M_n(R)$, $i=0,\dots,r$. Since $u_0,u_1,\dots,u_r$ are linearly independent over $R$, we deduce $B_i\rho(g)=\tau(g)B_i$ for all $i=0,\dots,r$. Moreover, since $B\in \GL_n(S)$, we have $\det(\sum_{i=0}^rB_iu_i)\in S^\times=S\setminus P$. Therefore the polynomial $\det(B_0+\sum_{i=1}^rB_iX_i)\in R[X_1,\dots,X_r]$ is not identically $0$ even modulo $\mfp$ as it has an evaluation over $S$ with value nonzero modulo $\mfp S\subseteq P$. By our assumption that $R/\mfp$ is infinite, there exists elements $a_1,\dots,a_r\in R$ such that $\det(B_0+\sum_{i=1}^rB_ia_i)\in R^\times=R\setminus\mfp$. Hence $B':=B_0+\sum_{i=1}^rB_ia_i\in \GL_n(R)$ gives an ismorphism between $V\cong 1\otimes V\subset S\otimes_R V$ and $W\cong 1\otimes W\subset S\otimes_R W$ since we have $B'\rho(g)=\tau(g)B'$ for all $g\in G$.
\end{proof}
Now we prove the proposition by induction on $|\Delta|$. The case $|\Delta|=1$ is the main result in \cite{CC}. Let $\alpha\in \Delta$ be fixed for some set $|\Delta|>1$ and pick a continuous representation $T$ of $\GQpD$, free of rank $n$ over $\Zp$. We put
\begin{equation*}
\mathbb{D}^\dagger_\circ(T):=\left(\mathcal{O}_{\widehat{\mathcal{E}^{ur}_{\Delta,\circ}}}^\dagger\otimes_{\Zp}T\right)^{\HQpD}\ .
\end{equation*}
\begin{lem}\label{circextend}
We have $\mathbb{D}^\dagger(T)\cong \OED^\dagger\otimes_{\mathcal{O}_{\mathcal{E}_{\Delta,\circ}}^\dagger}\mathbb{D}^\dagger_\circ(T)$.
\end{lem}
\begin{proof}
Since $\HQpD$ acts trivially on $\OED^\dagger$, we compute
\begin{align*}
\mathbb{D}^\dagger(T)\cong \left(\mathcal{O}_{\widehat{\mathcal{E}^{ur}_{\Delta}}}^\dagger\otimes_{\Zp}T\right)^{\HQpD}\cong \left(\left(\OED^\dagger\otimes_{\mathcal{O}_{\mathcal{E}_{\Delta,\circ}}^\dagger}\mathcal{O}_{\widehat{\mathcal{E}^{ur}_{\Delta,\circ}}}^\dagger\right)\otimes_{\Zp}T\right)^{\HQpD} \cong\\
\cong \OED^\dagger\otimes_{\mathcal{O}_{\mathcal{E}_{\Delta,\circ}}^\dagger}\left(\mathcal{O}_{\widehat{\mathcal{E}^{ur}_{\Delta,\circ}}}^\dagger\otimes_{\Zp}T\right)^{\HQpD}\cong \OED^\dagger\otimes_{\mathcal{O}_{\mathcal{E}_{\Delta,\circ}}^\dagger}\mathbb{D}^\dagger_\circ(T)
\end{align*}
as claimed.
\end{proof}

By the case $|\Delta|=1$ the rank of the $\mathcal{O}_{\mathcal{E}_\alpha}^\dagger$-module $\mathbb{D}_\alpha^\dagger(T):=(\mathcal{O}_{\widehat{\mathcal{E}^{ur}_{\alpha}}}^\dagger\otimes_{\Zp}T)^{\HQpa}$ equals $n$. In particular, we have
\begin{equation*}
\mathcal{O}_{\widehat{\mathcal{E}^{ur}_{\alpha}}}^\dagger\otimes_{\Zp}T\cong \mathcal{O}_{\widehat{\mathcal{E}^{ur}_{\alpha}}}^\dagger\otimes_{\mathcal{O}_{\mathcal{E}_\alpha}^\dagger}\mathbb{D}_\alpha^\dagger(T)
\end{equation*}
as representations of $\GQpD$. Here $\mathbb{D}_\alpha^\dagger(T)$ is stable under the action of $\GQpD$ as a subspace of $\mathcal{O}_{\widehat{\mathcal{E}^{ur}_{\alpha}}}^\dagger\otimes_{\Zp}T$ since $\HQpa$ is a normal subgroup in $\GQpD$. While the action of $\GQpa$ is semilinear, that of $\GQpDa$ is linear. Note that both $\mathcal{O}_{\mathcal{E}_\alpha}^\dagger$ and $\mathcal{O}_{\widehat{\mathcal{E}^{ur}_{\alpha}}}^\dagger$ are discrete valuation rings (see also Prop.\ \ref{jacoverconv}) with uniformizer $p$ and infinite residue fields $E_\alpha=\Fp\bg X_\alpha\jg$ (resp.\ $E_\alpha^{sep}$). So we may apply Lemma \ref{descentrep} to deduce the---non-canonical---isomorphism
\begin{equation*}
\mathcal{O}_{\mathcal{E}_\alpha}^\dagger\otimes_{\Zp}T\cong \mathbb{D}_\alpha^\dagger(T)
\end{equation*}
as representations of $\GQpDa$. In particular, we find a $\Zp$-submodule $T_{\Delta\setminus\{\alpha\}}\subset \mathbb{D}_\alpha^\dagger(T)$ of rank $n$ over $\Zp$ such that we have $T_{\Delta\setminus\{\alpha\}}\cong T$ as representations of $\GQpDa$ and $T_{\Delta\setminus\{\alpha\}}$ contains a basis of the $\mathcal{O}_{\mathcal{E}_\alpha}^\dagger$-module $\mathbb{D}_\alpha^\dagger(T)$. Hence we compute
\begin{align*}
\mathbb{D}^\dagger_\circ(T)=\left(\mathcal{O}_{\widehat{\mathcal{E}^{ur}_{\Delta,\circ}}}^\dagger\otimes_{\Zp}T\right)^{\HQpD}\cong \left(\mathcal{O}_{\widehat{\mathcal{E}^{ur}_{\Delta\setminus\{\alpha\},\circ}}}^\dagger\otimes_{\Zp}(\mathcal{O}_{\widehat{\mathcal{E}^{ur}_{\alpha}}}^\dagger\otimes_{\Zp}T)^{\HQpa}\right)^{\HQpDa}\cong\\
\cong\left(\mathcal{O}_{\widehat{\mathcal{E}^{ur}_{\Delta\setminus\{\alpha\},\circ}}}^\dagger\otimes_{\Zp}(\mathcal{O}_{\mathcal{E}_\alpha}^\dagger\otimes_{\Zp}T_{\Delta\setminus\{\alpha\}})\right)^{\HQpDa}\cong
\mathcal{O}_{\mathcal{E}_\alpha}^\dagger\otimes_{\Zp}\left(\mathcal{O}_{\widehat{\mathcal{E}^{ur}_{\Delta\setminus\{\alpha\},\circ}}}^\dagger\otimes_{\Zp}T_{\Delta\setminus\{\alpha\}}\right)^{\HQpDa}
\end{align*}
as mere $\mathcal{O}_{\mathcal{E}_{\Delta,\circ}}^\dagger$-modules. By induction $\left(\mathcal{O}_{\widehat{\mathcal{E}^{ur}_{\Delta\setminus\{\alpha\},\circ}}}^\dagger\otimes_{\Zp}T_{\Delta\setminus\{\alpha\}}\right)^{\HQpDa}$ is a free module of rank $n$ over $\mathcal{O}_{\mathcal{E}_{\Delta\setminus\{\alpha\},\circ}}^\dagger$, so $\mathbb{D}^\dagger_\circ(T)$ is a free module of rank $n$ over $\mathcal{O}_{\mathcal{E}_{\Delta,\circ}}^\dagger\cong \mathcal{O}_{\mathcal{E}_\alpha}^\dagger\otimes_{\Zp}\mathcal{O}_{\mathcal{E}_{\Delta\setminus\{\alpha\},\circ}}^\dagger$. Hence $\mathbb{D}^\dagger(T)\cong \OED^\dagger\otimes_{\mathcal{O}_{\mathcal{E}_{\Delta,\circ}}^\dagger}\mathbb{D}^\dagger_\circ(T)$ is free of rank $n$ over $\OED^\dagger$ as claimed.

Finally, the map $$\id\otimes\varphi_\alpha\colon \OED^\dagger\otimes_{\varphi_\alpha,\OED^\dagger}\mathbb{D}^\dagger(T)\to \mathbb{D}^\dagger(T)$$ is an isomorphism for all $\alpha\in\Delta$ by Nakayama's Lemma and Prop.\ \ref{jacoverconv}, since it is an isomorphism modulo $p$ (we have $\OED^\dagger/(p)\cong E_\Delta$ and $\mathbb{D}^\dagger(T)/p\mathbb{D}^\dagger(T)\cong \mathbb{D}(T/pT)$). We deduce $\mathbb{D}^\dagger(V)=\mathbb{D}^\dagger(T)[p^{-1}]$ is an object in $\mathcal{D}^{et}(\varphi_{\Delta},\Gamma_\Delta,\mathcal{E}_\Delta^\dagger)$.
\end{proof}

We end this section by proving a result that will be needed in the last section, but might be of independent interest, too.

\begin{pro}\label{separatevariablebasis}
Let $T$ be a $p$-torsion free object in $\operatorname{Rep}_{\Zp}(\GQpD)$ and put $D^\dagger:=\mathbb{D}^\dagger(T)$. Then for all $\alpha\in\Delta$, we have $\OED^\dagger\otimes_{\mathcal{O}_{\mathcal{E}_\alpha}^\dagger}\mathbb{D}^\dagger_\alpha(T_{\mid \GQpa})\cong D^\dagger$ as $(\varphi_\alpha,\Gamma_\alpha)$-modules. In particular, there exists a basis (system of free generators) of $D^\dagger$---depending on $\alpha\in \Delta$---such that the matrix of $\varphi_\alpha$ and $\gamma_\alpha\in\Gamma_\alpha$ lie in $\mathcal{O}_{\mathcal{E}_\alpha}^\dagger$.
\end{pro}
\begin{proof}
By induction on $|\Delta|$ and the main argument in the proof of Prop.\ \ref{overconverge}, the $\GQpa$-representation $\left(\mathcal{O}_{\widehat{\mathcal{E}^{ur}_{\Delta\setminus\{\alpha\},\circ}}}^\dagger\otimes_{\Zp}T\right)^{\HQpDa}$ is isomorphic to $\mathcal{O}_{\mathcal{E}_{\Delta\setminus\{\alpha\},\circ}}^\dagger\otimes_{\Zp}T$. Applying $\mathbb{D}^\dagger_\alpha=(\mathcal{O}_{\widehat{\mathcal{E}^{ur}_{\alpha}}}^\dagger\otimes_{\Zp}\cdot)^{\HQpa}$ on this isomorphism yields $\mathcal{O}_{\mathcal{E}_{\Delta,\circ}}^\dagger\otimes_{\mathcal{O}_{\mathcal{E}_\alpha}^\dagger}\mathbb{D}^\dagger_\alpha(T_{\mid \GQpa})\cong \mathbb{D}_\circ^\dagger(T)$ as $(\varphi_\alpha,\Gamma_\alpha)$-modules. The statement follows from Lemma \ref{circextend}.
\end{proof}
\begin{rem}
The statement of Prop.\ \ref{separatevariablebasis} is true for \'etale $(\varphi_\Delta,\Gamma_\Delta)$-modules over $\OED$, too.
\end{rem}

\subsection{Extended multivariable Robba rings}

Now our main goal is to show that the basechange functor from $\mathcal{D}^{et}(\varphi_{\Delta},\Gamma_\Delta,\mathcal{E}_\Delta^\dagger)$ to $\mathcal{D}^{et}(\varphi_{\Delta},\Gamma_\Delta,\mathcal{E}_\Delta)$ is an equivalence of categories. We proceed along the lines of the proof by Kedlaya (see Remark 1.7.4, Prop.\ 1.2.7, and 2.5.8 in \cite{K}) in the one variable case. For this recall that the extended Robba ring $\widetilde{\mathcal{R}}$ over $\Qp$ is the ring of formal generalized Laurent series $\sum_{i\in\mathbb{Q}}a_iu^i$ with $a_i\in\Qp$ satisfying the following conditions:
\begin{enumerate}[$(i)$]
\item For each $c>0$ the set of $i\in\mathbb{Q}$ such that $|a_i|_p\geq c$ is well-ordered.
\item There exists a real number $0<\rho<1$ such that for all $\rho<r<1$ we have $\sup_{i\in\mathbb{Q}}|a_i|_pr^i<\infty$.
\end{enumerate}
\begin{rem}
Let $\sum_{i\in\mathbb{Q}}a_iu^i$ as above. Then the finiteness of the supremum $\sup_{i\in\mathbb{Q}}|a_i|_p(r-\varepsilon)^i$ for some $\varepsilon\in (0,r-\rho)$ implies $|a_i|_pr^i\to 0$ as $i\to -\infty$, and the finiteness of $\sup_{i\in\mathbb{Q}}|a_i|_p(r_1+\varepsilon)^i$ with $\max(\rho-r_1,0)<\varepsilon<1-r_1$ implies $|a_i|_pr_1^i\to 0$ as $i\to +\infty$ for all $0<r_1<1$. This shows that our definition of $\widetilde{\mathcal{R}}$ is equivalent to that given in Def.\ 2.2.4 of \cite{K}.
\end{rem}
Further we denote by $\widetilde{\mathcal{R}}^{bd}$ (resp.\ $\widetilde{\mathcal{R}}^{int}$) the subring of $\widetilde{\mathcal{R}}$ with bounded (resp.\ integral, ie.\ bounded by $1$) coefficients. The Frobenius $\varphi$ is defined on these rings by sending $\sum_{i\in\mathbb{Q}}a_iu^i$ to $\sum_{i\in\mathbb{Q}}a_iu^{pi}$ and is therefore bijective. By Prop.\ 2.2.6 in \cite{K} there exists a $\varphi$-equivariant embedding $\iota\colon\mathcal{R}\hookrightarrow \widetilde{\mathcal{R}}$ that preserves the $r$-norm of each element that is convergent on an annulus $(\rho,1)$ with $p^{-p/(p-1)}\leq\rho<r<1$. More concretely, the variable $X$ is sent to $\lim_l\iota_l(X)\in \widetilde{\mathcal{R}}^{int}$ where for $l\geq 1$ the sequence $\iota_l(X)$ is defined inductively by putting $\iota_1(X):=u$ and
\begin{equation}\label{Xtoextended}
\iota_{l+1}(X):=\iota_l(X)+u\sum_{j=0}^\infty p^{j}\varphi^{-j-1}\left(\frac{(\iota_l(X)+1)^p-1-\varphi(\iota_l(X))}{u^p}\right)\ .
\end{equation}
As before, we consider a copy $\widetilde{\mathcal{R}}^{int}_\alpha$, $\widetilde{\mathcal{R}}^{bd}_\alpha$, and $\widetilde{\mathcal{R}}_\alpha$ of these rings (with variable $u_\alpha$) for each $\alpha\in\Delta$. Now we define a multivariable analogue of these rings as follows. We consider the set $\widetilde{\mathcal{R}}^{int}_\Delta$ of multivariable generalized Laurent series with coefficients in $\Qp$ of the form 
\begin{equation}\label{expansionQ}
a:=\sum_{{\bf i}=(i_\alpha)_{\alpha\in\Delta}\in\mathbb{Q}^\Delta}a_{\bf i}\prod_{\alpha\in\Delta}u_\alpha^{i_\alpha}\in \Qp^{\mathbb{Q}^\Delta}
\end{equation}
satisfying the following conditions:
\begin{enumerate}[$(i)$]
\item For each fixed $c>0$ and $\alpha\in\Delta$ the set of $i_\alpha\in\mathbb{Q}$ such that there exists an ${\bf i}\in\mathbb{Q}^\Delta$ having $i_\alpha$ in coordinate $\alpha$ with $|a_{\bf i}|\geq c$ is well-ordered.
\item There exists a real number $0<\rho<1$ such that for any tuple ${\bf r}=(r_\alpha)_{\alpha\in\Delta}\in (\rho,1)^\Delta$ we have $|a|_{\bf r}:=\sup_{i\in\mathbb{Q}}|a_i|\prod_{\alpha\in\Delta}r_\alpha^{i_\alpha}<\infty$.
\item $\limsup_{{\bf r}\to (1)_{\alpha\in\Delta}}|a|_{\bf r}\leq 1$.
\end{enumerate}

Note that for any formal sum $a\in \Qp^{\mathbb{Q}^\Delta}$ of the form \eqref{expansionQ} the supremum $$|a|_{\bf r}:=\sup_{i\in\mathbb{Q}}|a_i|\prod_{\alpha\in\Delta}r_\alpha^{i_\alpha}\in\mathbb{R}^{\geq 0}\cup\{\infty\}$$ makes sense. We say that $a \in \Qp^{\mathbb{Q}^\Delta}$ converges on the polyannulus $(\rho,1)^\Delta$ for some $0<\rho<1$ if $|a|_{\bf r}<\infty$ for all tuples ${\bf r}\in (\rho,1)^\Delta$. Further, a sequence $(a_n)_{n\geq 1}$ of formal expressions in $\Qp^{\mathbb{Q}^\Delta}$ is said to be Cauchy in the ${\bf r}$-norm if $|a_n|_{\bf r}<\infty$ for all $n\geq 1$ and for all $\varepsilon>0$ there exists an integer $N\geq 1$ such that for all $n,m\geq N$ we have $|a_n-a_m|_{\bf r}<\varepsilon$ (or by the ultramtric inequality it suffices to assume this for $m=n+1$ only). Note that if a sequence $(a_n)$ is Cauchy in $|\cdot |_{\bf r}$ then so are the coordinates $(a_{{\bf i},n})_{n\geq 1}$ for all ${\bf i}\in\mathbb{Q}^\Delta$. In particular, any Cauchy sequence has a unique limit in $|\cdot|_{\bf r}$ and this limit does not depend on ${\bf r}$ in the sense that whenever $(a_n)_{n\geq 1}$ is also Cauchy in $|\cdot|_{\bf r'}$ for some tuple ${\bf r'}\in (0,1)^\Delta$ then the limit is the same in $|\cdot|_{\bf r'}$ as in $|\cdot|_{\bf r}$. However, a priori it is unclear whether this limit also satisfies conditions $(i)-(iii)$ even if it exists and each formal expression $a_n$ ($n\geq 1$) satisfies these conditions.

\begin{lem}\label{integralandzeroseq}
Assume the above conditions $(i)-(iii)$. Then $a_{\bf i}$ lies in $\Zp$ for all ${\bf i}\in\mathbb{Q}^\Delta$ and we have $|a_{\bf i}|\prod_{\alpha\in\Delta}r_\alpha^{i_\alpha}\to 0$ as $\max_\alpha(|i_\alpha|)\to \infty$.
\end{lem}
\begin{proof}
The statement on the integrality of the coefficients $a_{\bf i}$ follows from $(iii)$. The second statement follows from applying the finiteness of the supremum $\sup_{i\in\mathbb{Q}}|a_i|\prod_{\alpha\in\Delta}r'^{i_\alpha}_\alpha$ with $r'_\beta:=r_\beta\pm \varepsilon$ and $r'_\alpha:=r_\alpha$ ($\alpha\in\Delta\setminus\{\beta\}$) for all choices of $\beta\in\Delta$.
\end{proof}

We are going to show that $\widetilde{\mathcal{R}}^{int}_\Delta$ is in fact a ring with respect to formal addition and multiplication. Moreover, it is a subring of the ring $W(\widetilde{k}_\Delta)$ of $p$-typical Witt vectors of the perfect ring 
\begin{align*}
\widetilde{k}_\Delta=\{b:=\sum_{{\bf i}=(i_\alpha)_{\alpha\in\Delta}\in\mathbb{Q}^\Delta}b_{\bf i}\prod_{\alpha\in\Delta}\overline{u_\alpha}^{i_\alpha} \mid b_{\bf i}\in\Fp\text{ and for all }\alpha\in\Delta\text{ the set of }i_\alpha\\
\text{s.\ t.\ there exists  an }{\bf i}\in\mathbb{Q}^\Delta\text{ having }i_\alpha\text{ in coordinate }\alpha\text{ with }b_{\bf i}\neq 0\text{ is well-ordered}\} 
\end{align*}
of characteristic $p$.

\begin{lem}\label{modpQseries}
$\widetilde{k}_\Delta$ forms a perfect ring of characteristic $p$ with respect to formal addition and multiplication.
\end{lem}
\begin{proof}
The fact that $\widetilde{k}_\Delta$ is an $\Fp$-vector space with respect to formal addition follows from noting that the set of well-ordered subsets of $\mathbb{Q}$ is closed under finite union and under taking subsets. The multiplication is well-defined since for two well-ordered subsets $A,B\subset \mathbb{Q}$ the Minkowski sum $A+B=\{a+b\in\mathbb{Q}\mid a\in A, b\in B\}$ is also well-ordered. Finally, perfectness follows from the fact that for a well-ordered subset $A\subset \mathbb{Q}$ the set $p^{-1}A=\{p^{-1}a\in \mathbb{Q}\mid a\in A\}$ is well-ordered, too.
\end{proof}

Now the ring $W(\widetilde{k}_\Delta)$ of $p$-typical Witt vectors of $\widetilde{k}_\Delta$ is a strict $p$-ring. For any $\alpha\in \Delta$ and positive integer $n$ we denote by $u_\alpha^{1/n}$ the multiplicative (Teichm\"uller) representative of $\overline{u_\alpha}^{1/n}$. This is consistent with the notations as we have $(u_\alpha^{1/nm})^n=u_\alpha^{1/m}$ for all $n,m\geq 1$ by multiplicativity. Further, for any tuple ${\bf i}=(i_\alpha)_\alpha\in\mathbb{Q}^\Delta$ the product $\prod_{\alpha\in\Delta}u_\alpha^{i_\alpha}\in W(\widetilde{k}_\Delta)$ makes sense and is the multiplicative representative of $\prod_{\alpha\in\Delta}\overline{u_\alpha}^{i_\alpha}$. In particular, the ring $A$ of formal expressions \eqref{expansionQ} with coefficients $a_{\bf i}\in\Zp$ satisfying condition $(i)$ above is a strict $p$-ring with $A/pA\cong \widetilde{k}_\Delta$ whence we have $A\cong W(\widetilde{k}_\Delta)$ by Thm.\ 1.1.8 in \cite{K3}. In particular, $\widetilde{\mathcal{R}}^{int}_\Delta$ can be viewed as a subset of $W(\widetilde{k}_\Delta)$ by the first statement in Lemma \ref{integralandzeroseq}.

\begin{lem}\label{RintDsubring}
The subset $\widetilde{\mathcal{R}}^{int}_\Delta\subset W(\widetilde{k}_\Delta)$ is a dense subring in the $p$-adic topology. In particular, we have $W(\widetilde{k}_\Delta)\cong \varprojlim_h \widetilde{\mathcal{R}}^{int}_\Delta/(p^h)$.
\end{lem}
\begin{proof}
$\widetilde{\mathcal{R}}^{int}_\Delta$ is clearly closed under addition. For a generalized formal Laurent series $a$ of the form \eqref{expansionQ} and $\alpha\in\Delta$ we define the sets
\begin{eqnarray*}
I_\alpha(a)&:=&\{i_\alpha\in\mathbb{Q}\mid \text{ there exists  an }{\bf i}\in\mathbb{Q}^\Delta\text{ having }i_\alpha\text{ in coordinate }\alpha\text{ with }a_{\bf i}\neq 0\}\ ;\\
I_\alpha(a,c)&:=&\{i_\alpha\in\mathbb{Q}\mid \text{ there exists  an }{\bf i}\in\mathbb{Q}^\Delta\text{ having }i_\alpha\text{ in coordinate }\alpha\text{ with }|a_{\bf i}|\geq c\}
\end{eqnarray*}
for all real number $c>0$. Assume that $I_\alpha(a)$ is well-ordered for all $\alpha\in\Delta$ and some fixed $a\in W(\widetilde{k}_\Delta)$. In particular, there exists a rational number $s$ such that $s\leq i_\alpha$ for all $i_\alpha\in I_\alpha(a)$ and for all $\alpha\in\Delta$. Hence $|a_{\bf i}|\prod_{\alpha\in\Delta}r_\alpha^{i_\alpha}\leq \prod_\alpha r_\alpha^{s}$ for all ${\bf i}\in\mathbb{Q}^\Delta$ and ${\bf r}\in (0,1)^\Delta$ as $a_{\bf i}\in\Zp$. In particular, $|a|_{\bf r}<\infty$ and $\limsup_{{\bf r}\to (1)_{\alpha\in\Delta}}|a|_{\bf r}\leq 1$. We deduce that $a$ belongs to $\widetilde{\mathcal{R}}^{int}_\Delta$. In particular, $\widetilde{\mathcal{R}}^{int}_{\circ,\Delta}:=\{a\in W(\widetilde{k}_\Delta)\mid I_\alpha(a)\text{ is well-ordered for all }\alpha\in\Delta\}$ is a subring in $W(\widetilde{k}_\Delta)$ (by the same argument as in the proof of Lemma \ref{modpQseries}) contained in $\widetilde{\mathcal{R}}^{int}_\Delta$ on which each function $|\cdot|_{\bf r}$ (${\bf r}\in (0,1)^\Delta$) is finite and is a multiplicative norm. By construction we have $\widetilde{\mathcal{R}}^{int}_{\circ,\Delta}/(p^n)\cong W(\widetilde{k}_\Delta)/(p^n)$, ie.\ $\widetilde{\mathcal{R}}^{int}_{\circ,\Delta}$ is dense $p$-adically in $W(\widetilde{k}_\Delta)$. Therefore $\widetilde{\mathcal{R}}^{int}_\Delta$ is also dense $p$-adically in $W(\widetilde{k}_\Delta)$.

Finally, let $a,b\in \widetilde{\mathcal{R}}^{int}_\Delta$ be two elements both converging on the polyannulus $(\rho,1)^\Delta$ for some $0<\rho<1$. Then for any tuple ${\bf r}=(r_\alpha)_{\alpha\in\Delta}$ with $\rho<r_\alpha<1$ for all $\alpha\in\Delta$ we may write $a=\lim_{n\to\infty} a_n$ and $b=\lim_{n\to\infty} b_n$ convergent in the ${\bf r}$-norm with elements $a_n,b_n\in \widetilde{\mathcal{R}}^{int}_{\circ,\Delta}$ such that the sets $I_\alpha(a_n)$ and $I_\alpha(b_n)$ are bounded (below and above) for all $\alpha\in\Delta$ and any fixed $n\geq 0$. Indeed, the boundedness can be achieved using a combination of $(i)$ and the second statement in Lemma \ref{integralandzeroseq} applied to both $a$ and $b$. Now the sequence $(a_nb_n)_n$ tends to $ab$ coefficientwise (ie.\ for each fixed ${\bf i}\in\mathbb{Q}^\Delta$ in the coefficient of $\prod_{\alpha\in\Delta}u_\alpha^{i_\alpha}$) and is a Cauchy sequence in the ${\bf r}$-norm. We deduce that $|ab|_{\bf r}=\lim_n |a_nb_n|_{\bf r}=\lim_n |a_n|_{\bf r}|b_n|_{\bf r}=|a|_{\bf r}|b|_{\bf r}<\infty$ whence $ab$ satisfies both $(ii)$ and $(iii)$.
\end{proof}

Note that the absolute $p$-Frobenius $\varphi$ lifts to the Witt ring $W(\widetilde{k}_\Delta)$ (by the formula $\varphi(u_\alpha)=u_\alpha^p$ for all $\alpha\in\Delta$) and is bijective. Moreover, it is also bijective on the subring $\widetilde{\mathcal{R}}^{int}_\Delta$. Further, we have the partial Frobenii $\varphi_\alpha$ ($\alpha\in\Delta$) acting on both these rings by the rule $\varphi_\alpha(u_\alpha):=u_\alpha^p$ and $\varphi_\alpha(u_\beta)=u_\beta$ for $\beta\in\Delta\setminus\{\alpha\}$.

\begin{lem}\label{uniquelimit}
Assume that a sequence $(a_n)_n$ of elements in $\widetilde{\mathcal{R}}^{int}_\Delta$ is Cauchy in the ${\bf r}$-norm. Then $(a_n)_n$ converges coefficientwise to an element in $\prod_{{\bf i}\in\mathbb{Q}^\Delta}\Zp \prod_{\alpha\in\Delta}u_\alpha^{i_\alpha}$. In particular, if it converges to an element $a\in W(\widetilde{k}_\Delta)$ in the ${\bf r}$-norm then $a$ does not depend on ${\bf r}$.
\end{lem}
\begin{proof}
This follows noting that whenever $(a_n)_n$ is Cauchy in the {\bf r}-norm for some ${\bf r}$ then so is the coefficients of any fixed $\prod_{\alpha\in\Delta}u_\alpha^{i_\alpha}$ in  $a_n$.
\end{proof}

\begin{pro}
There exists an embedding $\iota\colon \OED^\dagger\hookrightarrow \widetilde{\mathcal{R}}^{int}_\Delta$ that is norm-preserving and $\varphi_\alpha$-equivariant for all $\alpha\in\Delta$.
\end{pro}
\begin{rem}
By `norm-preserving' we mean that for any tuple ${\bf r}\in (0,1)^\Delta$ and $\lambda\in \OED^\dagger$ we have $|\iota(\lambda)|_{\bf r}=|\lambda|_{\bf r}$ including that one side is $+\infty$ if and only if so is the other.
\end{rem}
\begin{proof}
Note that the construction of $\widetilde{\mathcal{R}}^{int}_\Delta$ is functorial in the finite set $\Delta$. In particular, we have a ring embedding $\widetilde{\mathcal{R}}^{int}_\alpha\hookrightarrow \widetilde{\mathcal{R}}^{int}_\Delta$ sending $u_\alpha\in \widetilde{\mathcal{R}}^{int}_\alpha$ to $u_\alpha\in \widetilde{\mathcal{R}}^{int}_\Delta$. This is norm-preserving and $\varphi_\alpha$-equivariant. Precomposed by the $\varphi_\alpha$-equivariant and norm-preserving embedding $\mathcal{O}_{\mathcal{E}_\alpha}^\dagger\hookrightarrow \widetilde{\mathcal{R}}^{int}_\alpha$ defined by the sequence \eqref{Xtoextended} we obtain an embedding $\iota_\alpha\colon \mathcal{O}_{\mathcal{E}_\alpha}^\dagger\hookrightarrow\widetilde{\mathcal{R}}^{int}_\Delta$. On monomials of the form $\prod_{\alpha\in\Delta}X_\alpha^{n_\alpha}$ we define $\iota$ by putting $\iota(\prod_{\alpha\in\Delta}X_\alpha^{n_\alpha}):=\prod_{\alpha\in\Delta}\iota_\alpha(X_\alpha)^{n_\alpha}$. This extends to a ring homomorphism $\iota\colon \Zp[X_\alpha,X_\alpha^{-1}\mid \alpha\in\Delta]=\bigotimes_{\Zp,\alpha\in\Delta}\Zp[X_\alpha,X_\alpha^{-1}]$ that is norm-preserving in each ${\bf r}$-norm (${\bf r}\in(0,1)^\Delta$) and $\varphi_\alpha$-equivariant for all $\alpha\in\Delta$. For any element $\lambda\in \OED^\dagger$ there exists a tuple ${\bf r}$ such that $\lambda$ is convergent in the ${\bf r}$-norm whence it is the limit of a sequence $(\lambda_n)_n\subset \Zp[X_\alpha^{\pm 1}\mid \alpha\in\Delta]$ in the ${\bf r}$-norm. Since $\iota$ preserves the ${\bf r}$-norm, the sequence $\iota(\lambda_n)$ is Cauchy in the ${\bf r}$-norm, so it converges to an element in $\prod_{{\bf i}\in\mathbb{Q}^\Delta}\Zp \prod_{u_\alpha^{i_\alpha}}$ that we denote by $\iota(\lambda)$. We need to show that $\iota(\lambda)$ satisfies $(i),(ii),(iii)$. 

By construction $I_\alpha(\iota_\alpha(X_\alpha))\subset [1,2)$ is bounded. Therefore we have $I_\alpha(\iota_\alpha(X_\alpha^n))\subset [n,2n)$ for all positive integers $n$. In particular, any element of $I_\alpha(\iota_\alpha(X_\alpha^m))$ is bigger than any element of $I_\alpha(\iota_\alpha(X_\alpha^n))$ if $m\geq 2n$. This shows that for any integer $N>0$ and real number $c$ the set $\bigcup_{n\geq N}I_\alpha(\iota_\alpha(X_\alpha^n),c)$ is well ordered as any strictly decreasing infinite sequence would be contained in $\bigcup_{M\geq n\geq N}I_\alpha(\iota_\alpha(X_\alpha^n),c)$ for some integer $M$ which is well-ordered being a finite union of well-ordered subsets of $\mathbb{Q}$. Now note that any element $\lambda \in \OED^\dagger$ has bounded denominators modulo $p^h$ for any $h$ showing that $I_\alpha(\lambda,c)$ is contained in $\bigcup_{n\geq N}I_\alpha(\iota_\alpha(X_\alpha^n),c)$ for some $N$ (depending on $c$). This shows $(i)$ for $\iota(\lambda)$.

Since $\iota$ preserves $|\cdot|_{\bf r}$ for all ${\bf r}$ we deduce that $|\iota(\lambda)|_{\bf r}=|\lambda|_{\bf r}<\infty$ for all ${\bf r}\in(\rho,1)^\Delta$ for some $0<\rho<1$ and $$\limsup_{{\bf r}\to (1)_{\alpha\in\Delta}}|\iota(\lambda)|_{\bf r}\leq \limsup_{{\bf r}\to (1)_{\alpha\in\Delta}}|\lambda|_{\bf r}\leq 1\ .$$ 
\end{proof}

Passing to the $p$-adic completion we obtain an embedding $\OED=\varprojlim_h \OED^\dagger/(p^h)\hookrightarrow \varprojlim_h \widetilde{\mathcal{R}}^{int}_\Delta/(p^h)=W(\widetilde{k}_\Delta)$ that we still denote by $\iota$.

\begin{pro}\label{extendedoverconvintersect}
We have $\iota(\OED^\dagger)=\iota(\OED)\cap \widetilde{\mathcal{R}}^{int}_\Delta$ as subrings in $W(\widetilde{k}_\Delta)$.
\end{pro}
\begin{proof}
The containment $\iota(\OED^\dagger)\subseteq\iota(\OED)\cap \widetilde{\mathcal{R}}^{int}_\Delta$ is clear, so let $a:=\sum_{{\bf i}=(i_\alpha)_{\alpha\in\Delta}\in\mathbb{Q}^\Delta}a_{\bf i}\prod_{\alpha\in\Delta}u_\alpha^{i_\alpha}$ be both in $\iota(\OED)$ and $\widetilde{\mathcal{R}}^{int}_\Delta$. In particular, we have a $\lambda=\sum_{{\bf k}\in\mathbb{Z}^\Delta}\lambda_{\bf k}\prod_\alpha X_\alpha^{k_\alpha}$ in $\OED$ with $a=\iota(\lambda)$ and we are bound to show that $\lambda$ lies in $\OED^\dagger$. Now let $0<\rho<1$ be such that $a=\iota(\lambda)$ converges on the polyannulus $(\rho,1)^\Delta$ and let ${\bf r}\in (\rho,1)^\Delta$ be arbitrary. Assume there exists an index ${\bf k}\in\mathbb{Z}^\Delta$ such that $|\lambda_{\bf k}\prod_\alpha X_\alpha^{k_\alpha}|_{\bf r}>|a|_{\bf r}$ and choose ${\bf k_0}=(k_{0,\alpha})_{\alpha\in\Delta}$ such that $|\lambda_{\bf k_0}|$ is maximal among these. We may further assume that ${\bf k_0}$ is minimal for the lexicographical ordering in some fixed ordering $\Delta=\{\alpha_1<\dots<\alpha_{|\Delta|}\}$ among the indices ${\bf k}$ with $|\lambda_{\bf k}\prod_\alpha X_\alpha^{k_\alpha}|_{\bf r}>|a|_{\bf r}$ and $|\lambda_{\bf k}|$ maximal (there exists such since for fixed absolute value of the coefficient the indices are bounded below for elements in $\OED$). Now note that by construction we have $$\iota(X_\alpha^{k_{\alpha}})=u_\alpha^{k_{\alpha}}+\text{terms of higher degree}$$ for any $\alpha\in\Delta$ and integer $k_\alpha\in\mathbb{Z}$. By the choice of ${\bf k_0}$ we deduce $|a_{\bf k_0}|=|\lambda_{\bf k_0}|$ hence $|a|_{\bf r}<|\lambda_{\bf k_0}\prod_\alpha X_\alpha^{k_{0,\alpha}}|_{\bf r}=|a_{\bf k_0}\prod_{\alpha\in\Delta}u_\alpha^{k_{0,\alpha}}|_{\bf r}\leq |a|_{\bf r}$, contradiction. We deduce $|\lambda|_{\bf r}\leq |a|_{\bf r}<\infty$ (a posteriori we have equality) and $\limsup_{{\bf r}\to (1)_{\alpha\in\Delta}}|\lambda|_{\bf r}\leq \limsup_{{\bf r}\to (1)_{\alpha\in\Delta}}|a|_{\bf r}\leq 1$ showing that $\lambda$ belongs to $\OED^\dagger$.
\end{proof}

\subsection{The equivalence of categories}

\begin{pro}\label{h0phioverconv}
Let $D^\dagger$ be an object in $\mathcal{D}^{et}(\varphi_{\Delta},\Gamma_\Delta,\OED^\dagger)$ and put $D:=\OED\otimes_{\OED^\dagger}D^\dagger$. Then the natural map
\begin{equation*}
h^0\Phi^\bullet(1\otimes\id)\colon h^0\Phi^\bullet(D^\dagger)\to h^0\Phi^\bullet(D)
\end{equation*}
is bijective.
\end{pro}
\begin{proof}
We proceed in three steps as follows. The proof is inspired by the proof of the $1$-variable case (Prop.\ 2.5.8 in \cite{K}).

\emph{Step 1. We pass to the extended rings and reduce to the case when $D^\dagger$ is free.} Since $\OED/(p^h)\cong \OED^\dagger/(p^h)$ for all $h\geq 1$, we have $D^\dagger[p^\infty]\cong D[p^\infty]$. In particular, $h^i\Phi^\bullet(D^\dagger[p^\infty])\cong h^i\Phi^\bullet(D[p^\infty])$ for all $i\geq 0$. Therefore by the long exact sequence of $h^i\Phi^\bullet(\cdot)$ we may assume without loss of generality that $D^\dagger$ has no $p$-torsion whence $D^\dagger$ is free as a module over $\OED^\dagger$ by Thm.\ \ref{overconvfree}. Denote its rank by $n$ and put $\widetilde{D}^\dagger:=\widetilde{\mathcal{R}}^{int}_\Delta\otimes_{\OED^\dagger}D^\dagger$ and $\widetilde{D}:=W(\widetilde{k}_\Delta)\otimes_{\widetilde{\mathcal{R}}^{int}_\Delta}\widetilde{D}^\dagger$. Hence $\widetilde{D}^\dagger$ (resp.\ $\widetilde{D}$) is free of rank $n$ as a module over $\widetilde{\mathcal{R}}^{int}_\Delta$ (resp.\ over $W(\widetilde{k}_\Delta)$). The injectivity of $h^0\Phi^\bullet(1\otimes\id)$ is clear. Choosing a basis of the free module $D^\dagger$, we need to verify that the coordinates of elements in $h^0\Phi^\bullet(D)$ belong to $\OED^\dagger$. Therefore by Prop.\ \ref{extendedoverconvintersect} it suffices to show that the natural map
\begin{equation*}
h^0\Phi^\bullet(1\otimes\id)\colon h^0\Phi^\bullet(\widetilde{D}^\dagger)\to h^0\Phi^\bullet(\widetilde{D})
\end{equation*}
is bijective.

Pick a basis of $\widetilde{D}^\dagger$ and for each $\beta\in\Delta$ we put $B_\beta\in \GL_n(\widetilde{\mathcal{R}}^{int}_\Delta)$ for the matrix of $\varphi_\beta$ in the chosen basis and $A_\beta:=B_\beta^{-1}\in \GL_n(\widetilde{\mathcal{R}}^{int}_\Delta)$. Assume that $v\in W(\widetilde{k}_\Delta)^n$ is the coordinate vector of an element in $h^0\Phi^\bullet(\widetilde{D})$, ie.\ we have $A_\beta v=\varphi_\beta(v)$ for all $\beta\in\Delta$. We are going to show that $v\in (\widetilde{\mathcal{R}}^{int}_\Delta)^n$, ie.\ it satisfies conditions $(i)-(iii)$. Since $v\in W(\widetilde{k}_\Delta)^n$, condition $(i)$ is clear. For a matrix $A=((a_{ij}))_{1\leq i,j\leq n}\in \GL_n(\widetilde{\mathcal{R}}^{int}_\Delta)$ (resp.\ column vector $v=(v_1,\dots,v_n)\in W(\widetilde{k}_\Delta)^n$) and ${\bf r}\in (0,1)^\Delta$ we put $|A|_{\bf r}:=\max_{1\leq i,j\leq n}|a_{ij}|_{\bf r}$ (resp.\ $|v|_{\bf r}:=\max_{1\leq i\leq n}|v_i|_{\bf r}$). We write $A_\beta=\sum_{{\bf j}\in\mathbb{Q}^\Delta}A_{\beta,{\bf j}}\prod_{\alpha\in\Delta}u_\alpha^{j_\alpha}$ and $v_i=\sum_{{\bf j_i}\in \mathbb{Z}^\Delta}a_{\bf j_i}\prod_{\alpha\in\Delta}u_\alpha^{j_{i,\alpha}}$ with $A_{\beta,{\bf j}}\in \Zp^{n\times n}$, $a_{\bf j_i}\in \Zp$, and $i=1,\dots,n$. 

Let $\varepsilon>0$ be a real number. Since $A_\beta\in \GL_n(\widetilde{\mathcal{R}}^{int}_\Delta)$ for all $\beta\in\Delta$, there exists a radius $\rho=\rho(\varepsilon)\in (0,1)$ such that we have  $|A_\beta|_{\bf r}\leq 1+\varepsilon$ (apply $(iii)$) for any tuple ${\bf r}=(r_\alpha)_{\alpha\in\Delta}\in(\rho,1)^\Delta$ and for all $\beta\in\Delta$. Pick a tuple ${\bf r}\in (\rho,1)^\Delta$.

\emph{Step 2. We suppose $|A_\beta|_{\bf r}\leq 1$ for all $\beta\in\Delta$ and show $|v|_{\bf r}\leq 1$.}  Assume for contradiction that $v=(v_1,\dots,v_n)^T$ has $|v|_{\bf r}>1$. We define 
\begin{equation*}
c:=\sup_{1\leq i\leq n}\sup_{{\bf j_i}\in\mathbb{Q}^\Delta,\ |a_{\bf j_i}\prod_{\alpha\in\Delta}u_\alpha^{j_{i,\alpha}}|_{\bf r}>1}|a_{\bf j_i}|\in [0,1]\ .
\end{equation*}
By $(i)$ the set $U_{i,\alpha}:=\{j_{i,\alpha}\in \mathbb{Q}\mid |a_{\bf j_i}|=c\}$ is well-ordered for all $1\leq i\leq n$, $\alpha\in\Delta$. Hence $\{\prod_{\alpha\in\Delta}r_\alpha^{-j_{i,\alpha}}\mid |a_{\bf j_i}|=c\}\subset\mathbb{R}$ is also well-ordered since $r_\alpha^{-1}>1$ for all $\alpha\in\Delta$. So the supremum $$\sup_{1\leq i\leq n}\sup_{{\bf j_i}\in\mathbb{Q}^\Delta,\ |a_{\bf j_i}|=c}|a_{\bf j_i}\prod_{\alpha\in\Delta}u_\alpha^{j_{i,\alpha}}|_{\bf r}=\sup_{1\leq i\leq n}\sup_{{\bf j_i}\in\mathbb{Q}^\Delta,\ |a_{\bf j_i}|=c}c\prod_{\alpha\in\Delta}r_\alpha^{j_{i,\alpha}}$$ is taken at an index ${\bf j_i}^{(0)}=(j_{i,\beta}^{(0)})_{\beta\in\Delta}\in\mathbb{Q}^\Delta$ for some $1\leq i\leq n$. Since $|a_{{\bf j_i}^{(0)}}\prod_{\alpha\in\Delta}u_\alpha^{j_{i,\alpha}^{(0)}}|_{\bf r}>1$, we have $j_{i,\beta}^{(0)}<0$ for some $\beta\in\Delta$. We claim that the coefficient of the monomial $$\varphi_\beta(\prod_{\alpha\in\Delta}u_\alpha^{j_{i,\alpha}^{(0)}})=u_\alpha^{pj_{i,\beta}^{(0)}}\prod_{\alpha\in\Delta\setminus\{\beta\}}u_\alpha^{j_{i,\alpha}^{(0)}}$$ in the $i$th coordinate of $A_\beta v$ cannot have absolute value as big as $|a_{{\bf j_i}^{(0)}}|$ contradicting to the assumption that $A_\beta v=\varphi_\beta(v)$. At first note that $j_{i,\beta}^{(0)}<0$ implies $|a_{{\bf j_i}^{(0)}}\varphi_\beta(\prod_{\alpha\in\Delta}u_\alpha^{j_{i,\alpha}^{(0)}})|_{\bf r}>|a_{{\bf j_i}^{(0)}}\prod_{\alpha\in\Delta}u_\alpha^{j_{i,\alpha}^{(0)}}|_{\bf r}>1$, so the terms in $v$ with $|\cdot|_{\bf r}\leq 1$ can only produce terms with smaller $|\cdot|_{\bf r}$ since $|A_\beta|_{\bf r}\leq 1$. On the other hand, the coefficients $a_{\bf j_i}$ of the terms in the coordinates of $v$ with $|a_{\bf j_i}\prod_{\alpha\in\Delta}u_\alpha^{j_{i,\alpha}}|_{\bf r}>1$ have absolute value at most $c=|a_{{\bf j_i}^{(0)}}|$, so the terms with $|a_{\bf j_i}|<c$ cannot contribute either as we have $A_{\beta,{\bf j}}\in \Zp^{n\times n}$. Finally, the terms with $|a_{\bf j_i}|=c$ all have $|\cdot|_{\bf r}$ at most $|a_{{\bf j_i}^{(0)}}\prod_{\alpha\in\Delta}u_\alpha^{j_{i,\alpha}^{(0)}}|_{\bf r}$ therefore cannot add up to a term $|a_{{\bf j_i}^{(0)}}\varphi_\beta(\prod_{\alpha\in\Delta}u_\alpha^{j_{i,\alpha}^{(0)}})|_{\bf r}>|a_{{\bf j_i}^{(0)}}\prod_{\alpha\in\Delta}u_\alpha^{j_{i,\alpha}^{(0)}}|_{\bf r}$ (using again $|A_\beta|_{\bf r}\leq 1$).

\emph{Step 3. The general case.} For each $\alpha\in\Delta$ put $n_\alpha:=\left[\frac{\log(1+\varepsilon)}{(1-p)\log r_\alpha}\right]+1\in\mathbb{Z}^{>0}$ and divide the basis of $\widetilde{D}^\dagger$ by $\prod_{\alpha\in\Delta}u_\alpha^{n_\alpha}$. Put $A'_\beta$ ($\beta\in\Delta$) for the matrix of $\varphi_\beta^{-1}$ in the new basis. Since dividing the basis by $u_\beta$ changes $A_\beta$ to $A_\beta u_\beta^{p-1}$ and does not change $A_{\beta'}$ ($\beta'\neq\beta\in\Delta$), we deduce $$|A'_\beta|_{\bf r}=|A_\beta|_{\bf r}r_\alpha^{(p-1)n_\beta}\leq (1+\varepsilon)r_\beta^{\frac{\log(1+\varepsilon)}{-\log r_\beta}}= 1$$ for all $\beta\in\Delta$. On the other hand, the coordinate vector $v$ changes to $v\prod_{\alpha\in\Delta}u_\alpha^{n_\alpha}$ in the new basis. By Step 2 we obtain
\begin{align*}
|v|_{\bf r}=\prod_{\alpha\in\Delta}r_\alpha^{-n_\alpha}|v\prod_{\alpha\in\Delta}u_\alpha^{n_\alpha}|_{\bf r}\leq \prod_{\alpha\in\Delta}r_\alpha^{-n_\alpha}\leq \prod_{\alpha\in\Delta}r_\alpha^{-\frac{\log(1+\varepsilon)}{(1-p)\log r_\alpha}-1}=\\
=(1+\varepsilon)^{\frac{|\Delta|}{p-1}}\prod_{\alpha\in\Delta}r_\alpha^{-1}\leq (1+\varepsilon)^{\frac{|\Delta|}{p-1}}\rho^{-|\Delta|}\to 1
\end{align*}
as $\varepsilon\to 0$ and $\rho\to 1$.
\end{proof}

\begin{thm}\label{basechange_OED}
The basechange functor from $\mathcal{D}^{et}(\varphi_{\Delta},\Gamma_\Delta,\OED^\dagger)$ to $\mathcal{D}^{et}(\varphi_{\Delta},\Gamma_\Delta,\OED)$ is an equivalence of categories.
\end{thm}
\begin{proof}
The essential surjectivity follows from Prop.\ \ref{overconverge} combined with the equivalence of categories between $\Zp$-representations of $\GQpD$ and $\mathcal{D}^{et}(\varphi_{\Delta},\Gamma_\Delta,\OED)$ (Thm.\ 4.11 in \cite{MultVarGal}). The faithfulness is clear from Thm.\ \ref{overconvfree} noting that the basechange functor is the identity on the $p$-power torsion part. Finally, for objects $D_1^\dagger,D_2^\dagger$ in $\mathcal{D}^{et}(\varphi_{\Delta},\Gamma_\Delta,\OED^\dagger)$ the $\OED^\dagger$-module $\Hom_{\OED^\dagger}(D_1^\dagger,D_2^\dagger)$ is also an object in $\mathcal{D}^{et}(\varphi_{\Delta},\Gamma_\Delta,\OED^\dagger)$ via the operators $\varphi_\alpha(f)(\varphi_\alpha(x)):=\varphi_\alpha(f(x))$ and $\gamma_\alpha(f)(\gamma_\alpha(x)):=\gamma_\alpha(f(x))$ ($\alpha\in\Delta$ and $\gamma_\alpha\in\Gamma_\alpha$). Moreover, the morphisms as $(\varphi_\Delta,\Gamma_\Delta)$-modules are exactly those elements of $\Hom_{\OED^\dagger}(D_1^\dagger,D_2^\dagger)$ that are $\varphi_\alpha$ and $\Gamma_\alpha$-invariant for all $\alpha\in\Delta$. The statement follows from applying Prop.\ \ref{h0phioverconv} to $D^\dagger:=\Hom_{\OED^\dagger}(D_1^\dagger,D_2^\dagger)$.
\end{proof}

Inverting $p$ we obtain

\begin{cor}
The basechange functor from $\mathcal{D}^{et}(\varphi_{\Delta},\Gamma_\Delta,\mathcal{E}_\Delta^\dagger)$ to $\mathcal{D}^{et}(\varphi_{\Delta},\Gamma_\Delta,\mathcal{E}_\Delta)$ is an equivalence of categories.
\end{cor}

\begin{cor}\label{ddaggerequiv}
$\mathcal{D}^{et}(\varphi_{\Delta},\Gamma_\Delta,\OED^\dagger)$ (resp.\ $\mathcal{D}^{et}(\varphi_{\Delta},\Gamma_\Delta,\mathcal{E}_\Delta^\dagger)$) is equivalent to the category of continuous representations of $\GQpD$ on finitely generated $\Zp$-modules (resp.\ on finite dimensional $\Qp$-vectorspaces). The equivalence is realized by the functor $\mathbb{D}^\dagger$.
\end{cor}

\begin{cor}\label{basisforalpharestrict}
Let $D^\dagger$ be an object in $\mathcal{D}^{et}(\varphi_{\Delta},\Gamma_\Delta,\OED^\dagger)$ and $\alpha\in\Delta$. There exists a basis $(e_1,\dots,e_d)$ of $D^\dagger$ (depending on $\alpha$) in which the matrices of both $\varphi_\alpha$ and $\gamma_\alpha\in \Gamma_\alpha$ lie in the subring $\mathcal{O}_{\mathcal{E}_\alpha}^\dagger\subset\OED^\dagger$. In particular, $D^\dagger(\alpha):=\sum_{i=1}^d\mathcal{O}_{\mathcal{E}_\alpha}^\dagger e_i$ is an \'etale $(\varphi_\alpha,\Gamma_\alpha)$-module over $\mathcal{O}_{\mathcal{E}_\alpha}^\dagger$ that corresponds to the restriction of $\mathbb{V}(D^\dagger)$ to the component $\GQpa$. 
\end{cor}
\begin{proof}
This follows combining Prop.\ \ref{separatevariablebasis} and Cor.\ \ref{ddaggerequiv}.
\end{proof}

\subsection{Overconvergent Herr complex}
In this section, we extend the definition of Herr complex from Section \ref{secherrcomplexmodpn} to objects in $\mathcal{D}^{et}(\varphi_{\Delta},\Gamma_\Delta,\mathcal{E}_\Delta^\dagger)$ and show it computes the Galois cohomology.  In the overconvergent case, we first deal with Iwasawa complex and use it to deduce the results for overconvergent Herr complex.

We first show $\Psi^\bullet(D^{\dagger})$ (defined below) calculates the Iwasawa cohomology. 
We have an injective ring endomorphism $\varphi_\alpha\colon \OED^\dagger\to \OED^\dagger$ and define $\psi_\alpha:=\varphi_\alpha^{-1}\circ \frac{1}{p}\Tr_{\OED^\dagger/\varphi(\OED^\dagger)}$ as a distinguished left-inverse of $\varphi_\alpha$ for any $\alpha \in \Delta$. In more concrete terms $\psi_\alpha$ is the unique left inverse of $\varphi_\alpha$ that vanishes on $(1+X_\alpha)^j\varphi_\alpha(\OED^\dagger)$ for all $j$ not divisible by $p$.

For a $(\varphi_\Delta,\Gamma_\Delta)$-module $D^{\dagger}$ over $\OED^{\dagger}$ and $x\in D^\dagger$ we may write $x=\sum_{i=0}^{p-1}(1+X_\alpha)^i\varphi_\alpha(x_i)$ for some elements $x_i\in D^\dagger$ ($i=0,\dots,p-1$) and put $\psi_\alpha(x):=x_0$. We define the cochain complex
\begin{align*}
\Psi^\bullet(D^{\dagger})\colon 0\to D^{\dagger}\to \bigoplus_{\alpha\in\Delta}D^{\dagger}\to \dots\to \bigoplus_{\{\alpha_1,\dots,\alpha_r\}\in \binom{\Delta}{r}}D^{\dagger}\to\dots \to D^{\dagger}\to 0
\end{align*}
where for all $0\leq r\leq |\Delta|-1$ the map $d_{\alpha_1,\dots,\alpha_r}^{\beta_1,\dots,\beta_{r+1}}\colon D^{\dagger}\to D^{\dagger}$ from the component in the $r$th term corresponding to $\{\alpha_1,\dots,\alpha_r\}\subseteq \Delta$ to the component corresponding to the $(r+1)$-tuple $\{\beta_1,\dots,\beta_{r+1}\}\subseteq\Delta$ is given by
\begin{equation*}
d_{\alpha_1,\dots,\alpha_r}^{\beta_1,\dots,\beta_{r+1}}=\begin{cases}0&\text{if }\{\alpha_1,\dots,\alpha_r\}\not\subseteq\{\beta_1,\dots,\beta_{r+1}\}\\ (-1)^{\eta}(\id-\psi_\beta)&\text{if }\{\beta_1,\dots,\beta_{r+1}\}=\{\alpha_1,\dots,\alpha_r\}\cup\{\beta\}\ ,\end{cases}
\end{equation*}
where $\eta=\eta(\alpha_1,\dots,\alpha_r,\beta)$ is the number of elements in the set $\Delta\setminus\{\alpha_1,\dots,\alpha_r\}$ smaller than $\beta$.

In order to compute the cohomology of the above complex $\Psi^\bullet(D^{\dagger})$ we need mixed rings that behave like ``overconvergent'' for a subset $S\subseteq \Delta$ of variables and like $p$-adically completed rings for the other variables. Recall that $\OED$ consists of Laurent series of the form
$$f=\sum_{{\bf k}\in \mathbb{Z}^\Delta} c_{\bf k} \prod_{\alpha\in\Delta}X_\alpha^{k_\alpha}$$
with $c_{\bf k}\in \Zp$ such that $|c_{\bf k}|_p\to 0$ as long as $\min_{\alpha}k_\alpha\to-\infty$. For a subset $S\subseteq \Delta$ we define the mixed ring $\mathcal{O}_{\mathcal{E}_{\Delta,S}}^\dagger$ as the subset of those $f$ as above with the following two convergence properties:
\begin{enumerate}[$(i)$]
\item There exist real numbers $0<\rho_\beta<1$ for all $\beta\in S$ such that
$$|f|_{{\bf r}_S}=\sup_{{\bf k}\in \mathbb{Z}^\Delta} |c_{\bf k}|_p \prod_{\beta\in S}r_\beta^{k_\beta}$$
is finite for any sequence ${\bf r}_S=(r_\beta)_{\beta\in S}$ with $\rho_\beta<r_\beta<1$. 
\item We have $\limsup_{{\bf r}_S\to (1)_{\beta\in S}}|f|_{{\bf r}_S}\leq 1$.
\end{enumerate}
Note that we have $\mathcal{O}_{\mathcal{E}_{\Delta,\Delta}}^\dagger=\OED^\dagger$ and $\mathcal{O}_{\mathcal{E}_{\Delta,\emptyset}}^\dagger=\OED$. We put $D_S^\dagger:=\mathcal{O}_{\mathcal{E}_{\Delta,S}}^\dagger\otimes_{\OED^\dagger}D^\dagger$.

\begin{lem}\label{mixedringcontain}
For subsets $S\subseteq S'\subseteq\Delta$ we have $\mathcal{O}_{\mathcal{E}_{\Delta,S'}}^\dagger\subseteq \mathcal{O}_{\mathcal{E}_{\Delta,S}}^\dagger$.
\end{lem}
\begin{proof}
We may assume without loss of generality that $S'=S\cup \{\alpha\}$ for some $\alpha\in\Delta$. Pick an element $$f=\sum_{{\bf k}\in \mathbb{Z}^\Delta} c_{\bf k} \prod_{\alpha\in\Delta}X_\alpha^{k_\alpha}\in \mathcal{O}_{\mathcal{E}_{\Delta,S\cup\{\alpha\}}}^\dagger \ .$$
By condition $(ii)$ above there exists a real number $\rho=\rho(\varepsilon)\in (0,1)$ such that $|f|_{{\bf r}_{S\cup\{\alpha\}}}\leq 1+\varepsilon$ for any ${\bf r}_{S\cup\{\alpha\}}\in (\rho,1)^{S\cup\{\alpha\}}$. Fixing $\rho<r_\beta<1$ for $\beta\in S$ and letting $r_\alpha\to 1$ we deduce $|f|_{{\bf r}_S}\leq 1+\varepsilon$. This implies both $(i)$ and $(ii)$ (the latter by letting $\varepsilon\to 0$).
\end{proof}

\begin{lem}\label{mixedringintersect}
For any two subsets $S,S'\subseteq \Delta$ we have $\mathcal{O}_{\mathcal{E}_{\Delta,S'}}^\dagger\cap \mathcal{O}_{\mathcal{E}_{\Delta,S}}^\dagger=\mathcal{O}_{\mathcal{E}_{\Delta,S\cup S'}}^\dagger$.
\end{lem}
\begin{proof}
The containment $\mathcal{O}_{\mathcal{E}_{\Delta,S'}}^\dagger\cap \mathcal{O}_{\mathcal{E}_{\Delta,S}}^\dagger\supseteq\mathcal{O}_{\mathcal{E}_{\Delta,S\cup S'}}^\dagger$ is covered by Lemma \ref{mixedringcontain}. For sequences ${\bf r}_S\in (0,1)^S$ and ${\bf r'}_{S'}\in (0,1)^{S'}$ we have
\begin{equation*}
|c_{\bf k}\prod_{\beta\in\Delta}X_\beta^{k_\beta}|_{{\bf \sqrt{rr'}}_{S\cup S'}}= \sqrt{|c_{\bf k}\prod_{\beta\in\Delta}X_\beta^{k_\beta}|_{{\bf r}_{S}}|c_{\bf k}\prod_{\beta\in\Delta}X_\beta^{k_\beta}|_{{\bf r'}_{S'}}}
\end{equation*} 
where the sequence ${\bf \sqrt{rr'}}_{S\cup S'}$ is defined in coordinate $\alpha\in S\cup S'$ by the formula $\sqrt{r_\alpha r'_\alpha}$ where $r_\alpha$ (resp.\ $r'_\alpha$) is defined as $1$ for all $\alpha\in S'\setminus S$ (resp.\ for all $\alpha\in S\setminus S'$). This yields the estimate $|f|_{{\bf \sqrt{rr'}}_{S\cup S'}}\leq \sqrt{|f|_{{\bf r}_{S}}|f|_{{\bf r'}_{S'}}}$ for all $f\in \mathcal{O}_{\mathcal{E}_{\Delta,S'}}^\dagger\cap \mathcal{O}_{\mathcal{E}_{\Delta,S}}^\dagger$ and we are done.
\end{proof}

Assume $f\in \mathcal{R}$ is an element in the one variable Robba ring converging on the annulus $[\rho,1)$ and $\rho<\rho_1<\rho_2<1$. Then we have $|f|_{\rho_1}\leq \max (|f|_\rho,|f|_{\rho_2})$. Indeed, this is clear for any monomial and also for any Laurent-polynomial, therefore it is also true for any $f$ converging on the annulus $[\rho,1)$ by continuity. We call this the maximum principle for elements of $\mathcal{R}$ which is crucial in the proof of the following Lemma.

\begin{lem}\label{alphaexpansion}
For any $\alpha\in\Delta$ the ring $\mathcal{O}_{\mathcal{E}_{\Delta,\{\alpha\}}}^\dagger$ consists of all Laurent series of the form
$$ f=\sum_{{\bf k}\in \mathbb{Z}^{\Delta\setminus\{\alpha\}}} f_{\bf k} \prod_{\beta\in\Delta\setminus\{\alpha\}}X_\beta^{k_\beta}$$
where $f_{\bf k}\in \mathcal{O}_{\mathcal{E}_\alpha}^\dagger$ for all ${\bf k}\in \mathbb{Z}^{\Delta\setminus\{\alpha\}}$ satisfying the following properties:
\begin{enumerate}[$(a)$]
\item There exist real numbers $0<\rho<1$ and $C>0$ independent of ${\bf k}$ such that $f_{\bf k}$ converges on the annulus $\rho\leq |X_\alpha|_p < 1$ and we have $|f_{\bf k}|_{\rho}\leq C$.
\item We have $\limsup_{r\to 1}|f_{\bf k}|_r\to 0$ as long as $\min_{\beta\in\Delta\setminus\{\alpha\}}k_\beta\to -\infty$.
\end{enumerate}
\end{lem}
\begin{proof}
At first note that for any real number $0<r_\alpha<1$ the $r_\alpha$-norm of $f$ (for the subset $S=\{\alpha\}$) equals $|f|_{r_\alpha}=\sup_{{\bf k}\in\mathbb{Z}^{\Delta\setminus\{\alpha\}}}|f_{\bf k}|_{r_\alpha}$ by definition. Therefore condition $(a)$ follows for any $f\in \mathcal{O}_{\mathcal{E}_{\Delta,\{\alpha\}}}^\dagger$. Since $f$ lies in $\OED$, the denominators in $f$ modulo $p^n$ must be bounded, ie.\ for any $n\geq 1$ there exists an integer $k=k(n)\in\mathbb{Z}$ such that $f_{\bf k}$ is divisible by $p^n$ in $\mathcal{O}_{\mathcal{E}_\alpha}^\dagger$ whenever $\min_{\beta\in\Delta\setminus\{\alpha\}}k_\beta\leq k$. However, $f_{\bf k}$ is divisible by $p^n$ if and only if $\limsup_{r\to 1}|f_{\bf k}|_r\leq p^{-n}$ therefore $(b)$ follows for any $f\in \mathcal{O}_{\mathcal{E}_{\Delta,\{\alpha\}}}^\dagger$.

Conversely, let $f$ be a Laurent series as above satisfying $(a)$ and $(b)$. Combining the maximum principle with $(a)$ we deduce $|f_{\bf k}|_r\leq\max(C,1)$ for all ${\bf k}\in\mathbb{Z}^{\Delta\setminus\{\alpha\}}$ and $\rho\leq r<1$. In particular, $|f|_{r_\alpha}\leq\max(C,1)$ for all $\rho\leq r_\alpha<1$ since all the coefficients of the $X_\alpha$-expansion of $f_{\bf k}$ lie in $\Zp$. This is condition $(i)$ in the definition of $\mathcal{O}_{\mathcal{E}_{\Delta,\{\alpha\}}}^\dagger$. In order to show $(ii)$ we need a quantitative version of the above cited maximum principle: There exists an integer $N=N(C,\rho)<0$ such that $C\rho^{-N}<1$. So for any $\varepsilon>0$ there is a real number $0<\rho_1=\rho_1(\varepsilon,N)<1$ such that $r_\alpha^n\leq 1+\varepsilon$ for all $\rho_1\leq r_\alpha<1$ and $N\leq n$. On the other hand, there is another real number $0<\rho_2=\rho_2(C,\rho,N)$ such that we have $C\rho_2^N\rho^{-N}<1$. So if we have a monomial $aX_\alpha^n$ with $a\in\Zp$ such that $|aX_\alpha^n|_\rho\leq C$ then for any real number $\max(\rho_1,\rho_2)<r_\alpha<1$ we compute 
$$|aX_\alpha^n|_{r_\alpha}=|a|r_\alpha^n\leq \begin{cases}1+\varepsilon&\text{ for }n\geq N\\ Cr_\alpha^n\rho^{-n}<C\rho_2^n\rho^{-n}< C\rho_2^N\rho^{-N}<1&\text{ for }n<N\ . \end{cases}$$
We deduce $|f_{\bf k}|_{r_\alpha}\leq 1+\varepsilon$ for all ${\bf k}$ which yields $|f|_{r_\alpha}\leq 1+\varepsilon$ showing $(ii)$. Finally, the coefficient of $X_\alpha^{-n}$ in $f_{\bf k}$ tends $p$-adically to $0$ uniformly in ${\bf k}$. Combining this with $(b)$ we deduce that $f$ lies in $\OED$ therefore by the above discussion it also lies in $ \mathcal{O}_{\mathcal{E}_{\Delta,\{\alpha\}}}^\dagger$.
\end{proof}

For an inductional proof of the comparison between cohomologies of the overconvergent and completed Herr complexes our key is the following

\begin{pro}\label{keyovherr}
For any subset $S\subset \Delta$ and $\alpha\in\Delta\setminus S$ the natural inclusion $D_{S\cup\{\alpha\}}^\dagger\hookrightarrow D_S^\dagger$ induces a quasi-isomorphism between the cochain complexes $0\to D_{S\cup\{\alpha\}}^\dagger\overset{\psi_\alpha-1}{\to} D_{S\cup\{\alpha\}}^\dagger\to 0$ and $0\to D_{S}^\dagger\overset{\psi_\alpha-1}{\to} D_{S}^\dagger\to 0$.
\end{pro}
\begin{proof}
By Cor.\ \ref{basisforalpharestrict} we may choose a basis $(e_1,\dots,e_d)$ of $D^\dagger$ in which the matrices of both $\psi_\alpha$ and $\gamma_\alpha$ lie in the subring $\mathcal{O}_{\mathcal{E}_\alpha}^\dagger\subset\OED^\dagger$ and put $D^\dagger(\alpha):=\sum_{i=1}^d\mathcal{O}_{\mathcal{E}_\alpha}^\dagger e_i$. 

We prove the isomorphism on $h^0$ first. Pick an element $x=\sum_{i=1}^df^{(i)}e_i\in \sum_{i=1}^d \mathcal{O}_{\mathcal{E}_{\Delta,S}}^\dagger e_i=D_S^\dagger$ that is a fixed point of $\psi_\alpha$. Further, we write 
$$ f^{(i)}=\sum_{{\bf k}\in \mathbb{Z}^{\Delta\setminus\{\alpha\}}} f^{(i)}_{\bf k} \prod_{\beta\in\Delta\setminus\{\alpha\}}X_\beta^{k_\beta}\ .$$
By Lemma \ref{mixedringintersect} it suffices to show that $x$ lies in $D^\dagger_{\{\alpha\}}$. Note that for any ${\bf k}\in \mathbb{Z}^{\Delta\setminus\{\alpha\}}$, $x_{\bf k}:=\sum_{i=1}^df^{(i)}_{\bf k}e_i\in D^\dagger(\alpha)$ is also a fixed point of $\psi_\alpha$. Hence by Prop.\ III.3.2(ii) in \cite{CC2}, $f^{(i)}_{\bf k}$ lies in $\mathcal{O}_{\mathcal{E}_\alpha}^\dagger$ and converges on an annulus $(\rho_\alpha(D^\dagger),1)$ (independent of ${\bf k}$) for all ${\bf k}\in \mathbb{Z}^{\Delta\setminus\{\alpha\}}$. Moreover, $D^\dagger(\alpha)^{\psi_\alpha=\id}$ is compact (Prop.\ I.5.6(i) in \cite{CC2}), so we have a uniform bound for $|f^{(i)}_{\bf k}|_{r_\alpha}$ for any real number $r_\alpha\in (\rho_\alpha(D^\dagger),1)$ showing condition $(a)$ in Lemma \ref{alphaexpansion}. Condition $(b)$ is automatic since $f^{(i)}$ lies in $\OED$ (see the proof of Lemma \ref{alphaexpansion}).

For the injectivity on $h^1$ pick an element $x=\sum_{i=1}^df^{(i)}e_i\in D_{S\cup\{\alpha\}}^\dagger$ such that $x=\psi_\alpha(y)-y$ for some $y=\sum_{i=1}^dg^{(i)}e_i\in D_{S}^\dagger$ and let $A=(a_{i,j})_{i,j}\in (\mathcal{O}_{\mathcal{E}_\alpha}^\dagger)^{d\times d}$ be the matrix of $\psi_\alpha$. As before, we write
$$ f^{(i)}=\sum_{{\bf k}\in \mathbb{Z}^{\Delta\setminus\{\alpha\}}} f^{(i)}_{\bf k} \prod_{\beta\in\Delta\setminus\{\alpha\}}X_\beta^{k_\beta}\ ,\quad g^{(i)}=\sum_{{\bf k}\in \mathbb{Z}^{\Delta\setminus\{\alpha\}}} g^{(i)}_{\bf k} \prod_{\beta\in\Delta\setminus\{\alpha\}}X_\beta^{k_\beta}\ ,$$
and put $x_{\bf k}:=\sum_{i=1}^df_{\bf k}^{(i)}e_i$, $y_{\bf k}:=\sum_{i=1}^dg_{\bf k}^{(i)}e_i$ as elements in $D^\dagger(\alpha)$ such that $\psi_\alpha(y_{\bf k})-y_{\bf k}=x_{\bf k}$ for all ${\bf k}\in \mathbb{Z}^{\Delta\setminus\{\alpha\}}$. Following \cite{CC2}, we define $w_n(f)$ for an element $f\in \mathcal{O}_{\mathcal{E}_\alpha}$ as the smallest integer $k$ such that $f$ lies in $X_\alpha^{-k}\Zp\bs X_\alpha,\frac{p}{X_\alpha^{p-1}}\js +p^{n+1} \mathcal{O}_{\mathcal{E}_\alpha}$. By definition we have $|f|_\rho=\sup_{n\geq 0}\rho^{-w_n(f)-n(p-1)}p^{-n}$ for any $0<\rho<1$ (cf.\ Prop.\ III.2.1 in \cite{CC2}). Further, we put $w_n(A):=\max_{i,j}w_n(a_{i,j})$ and $w_n(z)$ for the maximum of the coordinates of $z\in D^\dagger(\alpha)$ in the basis $e_1,\dots,e_d$. We similarly extend all the functions $|\cdot|_\rho\colon \mathcal{O}_{\mathcal{E}_\alpha}^\dagger\to \mathbb{R}^{\geq 0}\cup\{+\infty\}$ to elements of $D^\dagger(\alpha)$ as the maximum of the values on the coordinates in the basis $e_1,\dots,e_d$. Lemma I.6.4 in \cite{CC2} yields $w_n(y_{\bf k})\leq \max(w_n(x_{\bf k}),\frac{p}{p-1}w_n(A)+1)$ for all $n\geq 0$. By Lemma \ref{alphaexpansion} there exist real numbers $p^{-\frac{1}{p-1}}<\rho<1$ and $C>0$ such that and $|f_{\bf k}^{(i)}|_\rho<C$ for all ${\bf k}\in\mathbb{Z}^{\Delta\setminus\{\alpha\}}$ and $1\leq i\leq d$. By possibly enlarging $\rho$ and $C$ further, we may assume that all the entries of the matrix $A$ satisfy $|a_{i,j}|_{\rho^{\frac{p}{p-1}}}<C\rho$. In particular, we have $\rho^{-w_n(x_{\bf k})-n(p-1)}p^{-n}<C$ and $\rho^{-\frac{p}{p-1}w_n(A)-np-1}p^{-n}<C$ for all $n\geq 0$. Putting these together we compute
\begin{align*}
|y_{\bf k}|_\rho=\sup_{n\geq 0}\rho^{-w_n(y_{\bf k})-n(p-1)}p^{-n}\leq \sup_{n\geq 0}\rho^{-\max(w_n(x_{\bf k}),\frac{p}{p-1}w_n(A)+1)-n(p-1)}p^{-n}\leq\\
\leq \max(\sup_{n\geq 0}\rho^{-w_n(x_{\bf k})-n(p-1)}p^{-n},\sup_{n\geq 0}\rho^{-\frac{p}{p-1}w_n(A)-1-n(p-1)}p^{-n})<C\ .
\end{align*}
We deduce using Lemma \ref{alphaexpansion} that $y$ in fact lies in $D^\dagger_{\{\alpha\}}$ whence also in $D^\dagger_{S\cup\{\alpha\}}=D_{S}^\dagger\cap D^\dagger_{\{\alpha\}}$ by Lemma \ref{mixedringintersect}.

For the surjectivity on $h^1$ pick an arbitrary $x\in D_S^\dagger$ and write, as before, $$x=\sum_{{\bf k}\in\mathbb{Z}^\Delta}(\prod_{\beta\in\Delta\setminus\{\alpha\}}X_\beta^{k_\beta})x_{\bf k}$$ with $x_{\bf k}$ in $D(\alpha)$ for any ${\bf k}\in\mathbb{Z}^{\Delta\setminus\{\alpha\}}$ where $D(\alpha):=\mathcal{O}_{\mathcal{E}_\alpha}\otimes_{\mathcal{O}_{\mathcal{E}_\alpha}^\dagger}D^\dagger(\alpha)$. By Lemma 3.6 in \cite{Liu} for each ${\bf k}$ there are $z_{\bf k}\in D^\dagger(\alpha)$ and $y_{\bf k}\in D(\alpha)$ such that $x_{\bf k}=z_{\bf k}+\psi_\alpha(y_{\bf k})-y_{\bf k}$. Further, by Prop.\ I.5.6 in \cite{CC2} we may choose all the $z_{\bf k}$ from a finitely generated $\Zp$-submodule of $D^\dagger(\alpha)$. Hence by a compactness argument the coordinates of $z_{\bf k}$ in the basis $e_1,\dots,e_d$ satisfy condition $(a)$ in Lemma \ref{alphaexpansion}. Moreover, whenever $x_{\bf k}$ lies in $p^nD(\alpha)$ for some integer $n\geq 0$ then $z_{\bf k}$ also belongs to $p^n D^\dagger(\alpha)$. This, on one hand, shows that $(b)$ in Lemma \ref{alphaexpansion} is also satisfied, so $z:=\sum_{{\bf k}\in\mathbb{Z}^{\Delta\setminus\{\alpha\}}}(\prod_{\beta\in\Delta\setminus\{\alpha\}}X_\beta^{k_\beta})z_{\bf k}$ makes sense and lies in $D^\dagger_{\{\alpha\}}$. On the other hand, it also follows that $z$ belongs to $D^\dagger_{\{\beta\}}$ for any $\beta\in S$:  For any real number $0<\rho_\beta<1$, we have $$|(\prod_{\beta\in\Delta\setminus\{\alpha\}}X_\beta^{k_\beta})z_{\bf k}|_{\rho_\beta}=\rho_\beta^{k_\beta}p^{-n}\leq |(\prod_{\beta\in\Delta\setminus\{\alpha\}}X_\beta^{k_\beta})x_{\bf k}|_{\rho_\beta}$$ where $n$ is the largest integer such that $z_{\bf k}\in p^nD^\dagger(\alpha)$. Using Lemma \ref{mixedringintersect} we conclude $z\in D^\dagger_{S\cup\{\alpha\}}$. Now $x-z$ is coordinatewise in the image of $\psi_\alpha-1$, so it remains to show that $y_{\bf k}$ glue together to an element of $D^\dagger_S$. Using again Lemma I.6.4 in \cite{CC2} we find that $w_n(y_{\bf k})$ is bounded for any fixed $n\geq 0$. Moreover, since $D(\alpha)/(\psi_\alpha-1)(D(\alpha))$ is finitely generated over $\Zp$, there is an integer $r\geq 0$ such that the $\Zp$-torsion part of $D(\alpha)/(\psi_\alpha-1)(D(\alpha))$ is killed by $p^r$. In particular, whenever $x_{\bf k}-z_{\bf k}$ is divisible by $pn$ then $y_{\bf k}$ can be chosen so that it is divisible by $p^{n-r}$. We deduce that $y:=\sum_{{\bf k}\in\mathbb{Z}^{\Delta\setminus\{\alpha\}}}(\prod_{\beta\in\Delta\setminus\{\alpha\}}X_\beta^{k_\beta})y_{\bf k}$ makes sense in $D$. Finally, the above discussion also yields the estimate
  $$|(\prod_{\beta\in\Delta\setminus\{\alpha\}}X_\beta^{k_\beta})y_{\bf k}|_{\rho_\beta}\leq p^r |(\prod_{\beta\in\Delta\setminus\{\alpha\}}X_\beta^{k_\beta})x_{\bf k}|_{\rho_\beta}$$
for any real number $0<r_\beta<1$ and $\beta \in S$, so we obtain $y\in D^\dagger_S$ by Lemmata \ref{alphaexpansion} and \ref{mixedringintersect}.
\end{proof}

\begin{pro}\label{psiovquasi}
Let $D^\dagger$ be an object in $\mathcal{D}^{et}(\varphi_{\Delta},\Gamma_\Delta,\OED^\dagger)$. The natural morphism $\Psi^\bullet(D^{\dagger}) \rightarrow \Psi^\bullet(D)$ is a quasi-isomorphism.
\end{pro}
\begin{proof}
Since we have $\mathcal{O}_{\mathcal{E}_{\Delta,\Delta}}^\dagger=\OED^\dagger$ and $\mathcal{O}_{\mathcal{E}_{\Delta,\emptyset}}^\dagger=\OED$, we are reduced to showing that the natural inclusion $\Psi^\bullet(D_{S\cup\{\alpha\}}^{\dagger}) \hookrightarrow \Psi^\bullet(D_S^\dagger)$ is a quasi-isomorphism for all $S\subset \Delta$ and $\alpha\in\Delta\setminus S$. However, this follows from Prop.\ \ref{keyovherr} noting that $\Psi^\bullet(D_{S\cup\{\alpha\}}^{\dagger})$ (resp.\ $\Psi^\bullet(D_S^\dagger)$) is the total complex of the double complex $0\to \Psi^\bullet_{\Delta\setminus\{\alpha\}}(D_{S\cup\{\alpha\}}^\dagger)\overset{\psi_\alpha-1}{\to} \Psi^\bullet_{\Delta\setminus\{\alpha\}}(D_{S\cup\{\alpha\}}^\dagger)\to 0$ (resp.\ $0\to \Psi^\bullet_{\Delta\setminus\{\alpha\}}(D_{S}^\dagger)\overset{\psi_\alpha-1}{\to} \Psi^\bullet_{\Delta\setminus\{\alpha\}}(D_{S}^\dagger)\to 0$) where $\Psi^\bullet_{\Delta\setminus\{\alpha\}}(D_{S\cup\{\alpha\}}^\dagger)$ (resp.\ $\Psi^\bullet_{\Delta\setminus\{\alpha\}}(D_{S}^\dagger)$) denotes the Koszul complex of the operators $\psi_\beta-1$ ($\beta\in\Delta\setminus\{\alpha\}$) on $D_{S\cup\{\alpha\}}^\dagger$ (resp.\ on $D_{S}^\dagger$), ie.\ it is the subcomplex of $\Psi^\bullet(D_{S\cup\{\alpha\}}^\dagger)$ (resp.\ of $\Psi^\bullet(D_{S}^\dagger)$) consisting of the direct summands $D_{S\cup\{\alpha\}}^\dagger$ (resp.\ $D_{S}^\dagger$) in each degree $r$ corresponding to subsets $\alpha\notin\{\alpha_1,\dots,\alpha_r\}\subset\Delta$.
\end{proof}

\begin{thm}
We have an isomorphism $$H^i_{Iw}(\GQpD,\cdot)\cong  h^{i-d}\Psi^\bullet(\mathbb{D}(\cdot)) \cong h^{i-d}\Psi^\bullet(\mathbb{D^{\dagger}}(\cdot))$$
of cohomological $\delta$-functors.
\end{thm}

\begin{proof}
The left isomorphism follows from Theorem \ref{iwasawamodp} and the right from Prop.\ \ref{psiovquasi}.
\end{proof}

Let $D^\dagger$ be an object in $\mathcal{D}^{et}(\varphi_{\Delta},\Gamma_\Delta,\OED^\dagger)$.
we denote by $\Psi\Gamma^\bullet(D^{\dagger})$ the total complex of the double complex $\Gamma^\bullet(\Psi^\bullet(D^{\dagger})^{C_\Delta})$.

\begin{pro}\label{psigammaov=psigamma}
The complex  $\Psi\Gamma^\bullet(D^{\dagger})$ is quasi-isomorphic to $\Psi\Gamma^\bullet(D)$. In particular, both compute the Galois cohomology groups $H^\bullet(\GQpD,\mathbb{V}(D))$. 
\end{pro}
\begin{proof}
This follows from the quasi-isomorphism in Prop.\ \ref{psiovquasi}   and definition of the complex. The second statement follows from Theorem \ref{psigamma=phigamma}.
\end{proof}

Let $D^\dagger$ be an object in $\mathcal{D}^{et}(\varphi_{\Delta},\Gamma_\Delta,\OED^\dagger)$. Analogus to Section \ref{secherrcomplex}, we define the cochain complex
\begin{align*}
\Phi^\bullet(D^{\dagger})\colon 0\to D^{\dagger}\to \bigoplus_{\alpha\in\Delta}D^{\dagger}\to \dots\to \bigoplus_{\{\alpha_1,\dots,\alpha_r\}\in \binom{\Delta}{r}}D^{\dagger}\to\dots \to D^{\dagger}\to 0
\end{align*}
where for all $0\leq r\leq |\Delta|-1$ the map $d_{\alpha_1,\dots,\alpha_r}^{\beta_1,\dots,\beta_{r+1}}\colon D^{\dagger}\to D^{\dagger}$ from the component in the $r$th term corresponding to $\{\alpha_1,\dots,\alpha_r\}\subseteq \Delta$ to the component corresponding to the $(r+1)$-tuple $\{\beta_1,\dots,\beta_{r+1}\}\subseteq\Delta$ is given by
\begin{equation*}
d_{\alpha_1,\dots,\alpha_r}^{\beta_1,\dots,\beta_{r+1}}=\begin{cases}0&\text{if }\{\alpha_1,\dots,\alpha_r\}\not\subseteq\{\beta_1,\dots,\beta_{r+1}\}\\ (-1)^{\varepsilon}(\id-\varphi_\beta)&\text{if }\{\beta_1,\dots,\beta_{r+1}\}=\{\alpha_1,\dots,\alpha_r\}\cup\{\beta\}\ ,\end{cases}
\end{equation*}
where $\varepsilon=\varepsilon(\alpha_1,\dots,\alpha_r,\beta)$ is the number of elements in the set $\{\alpha_1,\dots,\alpha_r\}$ smaller than $\beta$. 

Further, the cochain complex $\Phi\Gamma_\Delta^\bullet(D^{\dagger})$ is defined as the total complex of the double complex $\Gamma_\Delta^\bullet(\Phi^\bullet(D^{\dagger, C_\Delta}))$ and is called the \emph{Herr-complex} of $D^{\dagger}$, where $\Gamma_\Delta^\bullet$ is defined in Section \ref{secherrcomplex}.

\begin{lem}\label{tauinverseoverconv}
For each $\alpha\in\Delta$ the map $(\gamma_\alpha-1)$ on $ \bigcap_{\beta\in\Delta}\Ker(\psi_\beta\colon D^\dagger\to D^\dagger)$ is bijective.
\end{lem}
\begin{proof}
Since $(\gamma_\alpha-1)$ divides $(\gamma_\alpha^{p^n}-1)$, it suffices to check the bijectivity of the latter for some large enough $n\in\mathbb{Z}$. Further, $ \bigcap_{\beta\in\Delta}\Ker(\psi_\beta\colon D^\dagger\to D^\dagger)\subseteq \Ker(\psi_\alpha\colon D^\dagger\to D^\dagger)$ is a direct summand (with projection $\prod_{\beta\in\Delta\setminus\{\alpha\}}(1-\varphi_\beta\circ\psi_\beta)\colon \Ker(\psi_\alpha\colon D^\dagger\to D^\dagger)\to \bigcap_{\beta\in\Delta}\Ker(\psi_\beta\colon D^\dagger\to D^\dagger)$ commuting with $(\gamma_\alpha^{p^n}-1)$), so we are reduced to showing the bijectivity of $(\gamma_\alpha^{p^n}-1)$ as a map on $\Ker(\psi_\alpha\colon D^\dagger\to D^\dagger)$. By Prop.\ \ref{separatevariablebasis} and Cor.\ \ref{ddaggerequiv} there exists a basis $(e_1,\dots,e_d)$ of $D^\dagger$ in which the matrices of both $\psi_\alpha$ and $\gamma_\alpha$ lie in the subring $\mathcal{O}_{\mathcal{E}_\alpha}^\dagger\subset\OED^\dagger$. In particular, $D^\dagger(\alpha):=\sum_{i=1}^d\mathcal{O}_{\mathcal{E}_\alpha}^\dagger e_i$ is an \'etale $(\varphi_\alpha,\Gamma_\alpha)$-module over $\mathcal{O}_{\mathcal{E}_\alpha}^\dagger$ that corresponds to the restriction of $\mathbb{V}(D^\dagger)$ to the component $\GQpa$. Therefore the injectivity of $(\gamma_\alpha^{p^n}-1)$ follows directly from the one variable case (Prop.\ II.6.1 in \cite{CC}) as $D^\dagger\subset \prod_{{\bf k}\in \mathbb{Z}^{\Delta\setminus\{\alpha\}}}((\prod_{\beta\in\Delta\setminus\{\alpha\}}X_\beta^{k_\beta})D^\dagger(\alpha))$. For the surjectivity, we are going to use the same principle, but we need to show that the obtained preimage under $(\gamma_\alpha^{p^n}-1)$ indeed belongs to the subset $D^\dagger$ (ie.\ it has the required convergence properties). So we pick an element $x=\sum_{i=1}^d f^{(i)}e_i\in \Ker(\psi_\alpha\colon D^\dagger\to D^\dagger)$ where $f^{(i)}\in\OED^\dagger$ have expansion
$$ f^{(i)}=\sum_{{\bf k}\in \mathbb{Z}^{\Delta\setminus\{\alpha\}}} f_{\bf k}^{(i)} \prod_{\beta\in\Delta\setminus\{\alpha\}}X_\beta^{k_\beta}$$
for all $i=1,\dots,d$ converging on a polyannulus $(\rho,1)^\Delta$. By Lemma \ref{alphaexpansion} we have $x_{\bf k}=\sum_{i=1}^d f^{(i)}_{\bf k}e_i\in D^\dagger(\alpha)$ for all ${\bf k}\in \mathbb{Z}^{\Delta\setminus\{\alpha\}}$. Using again Prop.\ II.6.1 in \cite{CC} there exist real numbers $0<\rho_\alpha(D^\dagger)<1$ and $0<c(D^\dagger)$ not depending on $x$ and an element $y_{\bf k}=\sum_{i=1}^d g^{(i)}_{\bf k}e_i\in D^\dagger(\alpha)$ with $(\gamma_\alpha^{p^n}-1)(y_{\bf k})=x_{\bf k}$ for all ${\bf k}\in \mathbb{Z}^{\Delta\setminus\{\alpha\}}$ such that $$\max_i |g^{(i)}_{\bf k}|_{r_\alpha}\leq r_\alpha^{-c(D^\dagger)}\max_i |f^{(i)}_{\bf k}|_{r_\alpha}$$ for any real number $r_\alpha\in (\max(\rho,\rho_\alpha(D^\dagger)),1)$. In particular,
$$ g^{(i)}:=\sum_{{\bf k}\in \mathbb{Z}^{\Delta\setminus\{\alpha\}}} g_{\bf k}^{(i)} \prod_{\beta\in\Delta\setminus\{\alpha\}}X_\beta^{k_\beta}$$
lies in $\OED^\dagger$ and satisfies $\max_i |g^{(i)}|_{\bf r}\leq r_\alpha^{-c(D^\dagger)}\max_i |f^{(i)}_{\bf k}|_{\bf r}$ for any ${\bf r}=(r_\beta)_{\beta\in\Delta}$ satisfying $r_\beta\in (\rho,1)$ for all $\beta\in\Delta$ and, in addition, $r_\alpha\in (\rho_\alpha(D^\dagger),1)$. Putting $y:=\sum_{i=1}^d g^{(i)}e_i\in D^\dagger$ we find $(\gamma_\alpha^{p^n}-1)(y)=x$ as desired.
\end{proof}

\begin{thm}\label{phigammaov=phigamma}
Let $D^\dagger$ be an object in $\mathcal{D}^{et}(\varphi_{\Delta},\Gamma_\Delta,\OED^\dagger)$. Then the complex $\Psi\Gamma_\Delta^\bullet(D^{\dagger})$ is quasi-isomorphic to $\Phi\Gamma_\Delta^\bullet(D^{\dagger})$.  In particular, both compute the Galois cohomology groups $H^\bullet(\GQpD,\mathbb{V}(D))$.
\end{thm}

\begin{proof}
The proof follows closely the proof of Theorem \ref{psigamma=phigamma}. Consider the morphism
\begin{align*}
\psi^\bullet\colon \Phi^\bullet(D^\dagger)^{C_\Delta}\to \Psi^\bullet(D^\dagger)^{C_\Delta}
\end{align*}
of cochain complexes that is given by $(-1)^{\varepsilon(S)}\prod_{\alpha\in S}\psi_\alpha$ on the copy of $D^\dagger$ corresponding to a subset $S\subseteq \Delta$ with $|S|=r$ in $\Phi^r(D)^{C_\Delta}$ mapping onto the copy of $D^\dagger$ corresponding to $S$ in $\Psi^\bullet(D^\dagger)^{C_\Delta}$. As this is surjective in each degree via similar argument to Theorem \ref{psigamma=phigamma}, we are reduced to showing that the total complex of the double complex $\Gamma^\bullet(\Ker(\psi^\bullet))$ is acyclic. This follows the same way as Lemma \ref{kerpsigammaacyclic} using Lemma \ref{tauinverseoverconv} instead of Prop.\ \ref{freegammamod}. Finally, the second statement is a consequence of Prop.\ \ref{psigammaov=psigamma}.
\end{proof}

Combining all the previous discussion, we can summarize in the following
\begin{cor}
1) Let $T$ be an object in $\operatorname{Rep}_{\Zp}(\GQpD)$. We have isomorphisms 
\begin{eqnarray*}
H^i(\GQpD,T)&\cong& h^i\Phi\Gamma_\Delta^\bullet(\mathbb{D}(T)) \cong h^i\Phi\Gamma_\Delta^\bullet(\mathbb{D}^{\dagger}(T))\ \cong h^i\Psi\Gamma_\Delta^\bullet(\mathbb{D}^{\dagger}(T)) \ ;\\
H^i_{Iw}(\GQpD,T)&\cong& h^{i-d}\Psi^\bullet(\mathbb{D}(T)) \cong h^{i-d}\Psi^\bullet(\mathbb{D}^{\dagger}(T))
\end{eqnarray*}
natural in $T$ for all $i\geq 0$.\\
2) Let $V$ be an object in $\operatorname{Rep}_{\Qp}(\GQpD)$. We have isomorphisms 
\begin{eqnarray*}
H^i(\GQpD,V)&\cong& h^i\Phi\Gamma_\Delta^\bullet(\mathbb{D}(V)) \cong h^i\Phi\Gamma_\Delta^\bullet(\mathbb{D}^{\dagger}(V)) \ \cong h^i\Psi\Gamma_\Delta^\bullet(\mathbb{D}^{\dagger}(V)) \ ;\\
H^i_{Iw}(\GQpD,V)&\cong& h^{i-d}\Psi^\bullet(\mathbb{D}(V)) \cong h^{i-d}\Psi^\bullet(\mathbb{D}^{\dagger}(V))
\end{eqnarray*}
natural in $V$ for all $i\geq 0$.
\end{cor}

\begin{rem}
The arguments in this section relied heavily on Prop.\ \ref{separatevariablebasis} which is only valid a priori for objects in the essential image of the functor $\mathbb{D}^\dagger$. So one cannot, at least trivially, replace the use of extended Robba rings with the arguments in this section in order to show Prop.\ \ref{h0phioverconv} (and hence Thm.\ \ref{basechange_OED}). 
\end{rem}

\subsection{Acknowledgements}

The second named author was supported by a Hungarian NKFIH Research grants K-100291 and FK-127906, by the J\'anos Bolyai Scholarship of the Hungarian Academy of Sciences, and by the MTA Alfr\'ed R\'enyi Institute of Mathematics Lend\"ulet Automorphic Research Group. He would like to thank the Arithmetic Geometry and Number Theory group of the University of Duisburg--Essen, campus Essen, for its hospitality where parts of this paper was written. Both authors acknowledge financial support from SFB TR45. We would like to thank Jan Kohlhaase and Kiran Kedlaya for valuable comments and feedback. We thank the referee for their careful reading of the manuscript.

\end{document}